\documentclass{article}%
\usepackage{amsmath}
\usepackage{amsfonts}
\usepackage{amssymb}
\usepackage{graphicx}%
\usepackage{color}
\usepackage{amsthm,amstext,bm}
 \usepackage{parskip}
\usepackage{hyperref}

\newcommand{\R}{\mathbb{R}}

\newcommand{\btau}{\bar{\tau}}

\newcommand{\C}{\mathbb{C}}
\newcommand{\Z}{\mathbb{Z}}

\newcommand{\N}{\mathbb{N}}

\newcommand{\eps}{\varepsilon}
\newcommand{\bx}{\bar{x}}
\newcommand{\bu}{\bar{u}}
\newcommand{\bv}{\bar{v}}
\newcommand{\bw}{\bar{w}}
\newcommand{\bb}{\bar{\beta}}

\newcommand{\bU}{\bar{U}}

\newcommand{\bydef}{\stackrel{\mbox{\tiny\textnormal{\raisebox{0ex}[0ex][0ex]{def}}}}{=}} 

\usepackage{bbm}
\usepackage[scr=boondoxo,scrscaled=1.05]{mathalfa}

\newtheorem{theorem}{Theorem}

\newtheorem{lemma}[theorem]{Lemma}

\newtheorem{proposition}[theorem]{Proposition}
\newtheorem{remark}[theorem]{Remark}

\title{Computer-assisted proofs of periodic solutions of state-dependent delay equations beyond perturbative regimes}

\begin{document}

\author{
Jan Bouwe van den Berg 
\thanks{VU Amsterdam, Department of Mathematics, De Boelelaan 1111,
            1081 HV Amsterdam, The Netherlands, {\tt janbouwe.vanden.berg@vu.nl}}
\and
Maxime Breden
\thanks{CMAP, CNRS, \'Ecole polytechnique, Institut Polytechnique de Paris, 91120 Palaiseau, France, {\tt maxime.breden@polytechnique.edu}}
\and
Matthieu Cadiot\thanks{CMAP, CNRS, \'Ecole polytechnique, Institut Polytechnique de Paris, 91120 Palaiseau, France, {\tt matthieu.cadiot@polytechnique.edu}}
\and
Kevin E. M. Church 
\thanks{Centre de Recherches Mathématiques, Université de Montréal, 
Pavillon André-Aisenstadt,
2920 chemin de la Tour,
Montréal, QC, H3T 1J4, Canada, \texttt{kevin.church@umontreal.ca}}
\and
Jean-Philippe~Lessard\thanks{%
	Department of Mathematics and Statistics, McGill University,
	Montreal, QC, H3A 0B9, Canada, \texttt{jp.lessard@mcgill.ca}}
}

\maketitle


\begin{abstract}
Rigorous existence results for periodic solutions of state-dependent delay differential equations (SDDEs) are scarce outside perturbative regimes, where the equation is close to an ordinary or constant-delay differential equation. We develop a computer-assisted framework that proves the existence of periodic solutions of SDDEs directly. The numerical approximation is a trigonometric polynomial, and a true periodic solution nearby is obtained by a Newton--Kantorovich argument in the space of $C^1$ periodic functions. The main difficulty lies in the terms in which the solution is evaluated at a time that itself depends on the solution. We obtain rigorous enclosures of their Fourier coefficients by adapting validated discrete Fourier transform techniques for analytic nonlinearities. We apply the framework to three representative examples of increasing complexity: a periodically forced toy model, an autonomous scalar equation with two state-dependent delays, and a threshold-type delay model from operon dynamics in gene regulation. To our knowledge, these are the first computer-assisted proofs of existence of periodic solutions of SDDEs outside perturbative regimes. This opens the way to computer-assisted proofs for SDDEs more broadly. All computer-assisted proofs, including the requisite codes and the approximate objects, are accessible on GitHub at \cite{code_proof}.
\end{abstract}

\begin{center}
{\bf \small Keywords:}
{\small State-dependent delay equations, periodic orbits, Fourier expansions, \\ composition operator, Newton--Kantorovich, computer-assisted proofs}

\end{center}

\section{Introduction}

State-dependent delay differential equations (SDDEs) 
%
%
describe processes in which the instantaneous rate of change of the system depends not only on its current state but also on a past state, with the time lag connecting present and past determined by the current state itself. Such equations arise naturally in diverse applications. A classical example comes from electrodynamics: the interaction between charged particles is mediated by electromagnetic signals that propagate at finite speed, so the relevant delay is given by the distance between particles divided by the speed of light, a quantity that evolves with the system itself (see, e.g., \cite{MR2536104,MR2639911,MR3456815}). Similar mechanisms appear in population dynamics, where maturation or reproduction times depend on population density (see, e.g., \cite{MR4032632,CassidyHumphries2020,MR3456830}), in physiology, where differentiation delays vary with blood cell concentrations (see, e.g., \cite{MAHAFFY1998135,MR3571295}), and in engineering systems such as machining, where the delay is linked to the tool’s instantaneous speed (see, e.g., \cite{Johnson_1972,Insperger_Turi_2007}). We refer the interested reader to the comprehensive survey report \cite{MR2457636} for more general background, details and applications. 

While SDDEs provide a flexible and realistic modeling framework, their mathematical analysis is notoriously challenging, especially due to the presence of delays that depend on the evolving state through nonlinear composition terms. In contrast to ordinary differential equations (ODEs), parabolic partial differential equations (PDEs), or delay differential equations (DDEs) with constant delays, SDDEs do not, in general, generate a well-behaved semiflow on the classical phase spaces $C^0$ or $C^1$; under mild conditions, they instead generate a semiflow on an infinite-dimensional solution manifold \cite{MR2019242,MR4261206,MR3642771,MR1345150,MR2084753}. This distinction has far-reaching consequences: several basic questions remain open, including the regularity of the semiflow and the existence and structure of solution manifolds and invariant manifolds \cite{MR3748505}. 

In applications, periodic solutions are among the most important dynamical features of SDDEs. Such solutions represent recurrent phenomena and play a central role in understanding the long-term dynamics of the system. From a theoretical perspective, however, their analysis is highly nontrivial. For instance, while periodic solutions of analytic DDEs with constant delays are themselves analytic \cite{MR322315}, the question of analyticity versus nonanalyticity for SDDEs is considerably more subtle \cite{MR3229655}, although analyticity is known for threshold-type delays \cite{Krisztin2014} (see Section~\ref{sec:threshold_SDDE}). On the computational side, several approaches have recently been developed to study periodic orbits in SDDEs. These include numerical investigations of their emergence through Hopf bifurcations \cite{humphries_2017_resonance}, convergence results for pseudo-spectral methods \cite{ando2024}, and the development of \emph{a posteriori} validation techniques \cite{MR4856722}. Progress has also been made using the parameterization method for periodic orbits in state-dependent delayed perturbations of ODEs \cite{MR4469141,MR4308692,MR4287353,gimeno2025persistencehyperbolicsolutionsodes}, and even for quasi-periodic solutions under small state-dependent delays and exponential dichotomy assumptions \cite{MR3501842,MR3736145}.

Despite the importance of SDDEs in applications, rigorous methods for proving the existence of periodic solutions remain limited, particularly in regimes where the equations cannot be viewed as perturbations of ODEs and DDEs with constant delays. In the present work, we develop a constructive computer-assisted framework for the validation of periodic orbits in genuinely state-dependent delay equations. To the best of our knowledge, these results constitute the first rigorous existence proofs of periodic solutions obtained through a fully constructive approach in this setting. Beyond establishing the existence of periodic orbits of SDDEs away from perturbative regimes, our work also provides new tools for applications in which oscillatory dynamics play a central role. We present the methodology through three examples of increasing complexity. The first is a periodically forced non-autonomous scalar toy model designed to illustrate the main ideas in the simplest possible setting:
\begin{equation}\label{eq : toy model intro}
u'(t) = - u(t) - \eps_1 u(t-1-\eps_2u(t)) + \phi(t),
\end{equation}
where $\phi(t) \bydef \cos(t) + \sin(t)$ and $\eps_1,\eps_2 \in \R$. In this example, the forcing $\phi$ fixes the period a priori, so that no additional equation for the period is required. The second model is the autonomous scalar SDDE introduced in \cite{HumphriesDeMasi2012} (see also \cite{humphries_2017_resonance}),
\begin{equation}\label{eq : autonomous SDDE_intro}
u'(t) = - \gamma u(t) - \kappa_1 u (t - a_1 - c_1 u(t)) - \kappa_2 u(t - a_2 - c_2 u(t)),
\end{equation}
where $\gamma,\kappa_1,\kappa_2,a_1,a_2,c_1,c_2 \in \R$ are parameters. In contrast with the first example \eqref{eq : toy model intro}, the period is now unknown and must be solved for simultaneously with the periodic orbit, which naturally leads to the introduction of a phase condition. Our third example is a threshold-type SDDE arising in operon dynamics in biology, studied in \cite{MR4959313},
\begin{equation}\label{eq : appli original_intro}
w'(t) = b e^{-\mu \tau(t)} \frac{v(w(t))}{v(w(t - \tau(t)))} q(w(t-\tau(t))) - \gamma w(t), 
\qquad
a = \int_{- \tau(t)}^0 v(w(s+t)) ds,
\end{equation}
where $a,b,\mu,\gamma>0$ are parameters, $q$ and $v$ are Hill functions (see Section~\ref{sec:threshold_SDDE} for details), and $\tau(t)$ is a threshold delay defined implicitly through the state variable $w$. As explained in Section~\ref{sec:threshold_SDDE}, the Leibniz rule is used to differentiate the threshold relation and reformulate \eqref{eq : appli original_intro} as a coupled system of SDDEs. Studying example \eqref{eq : appli original_intro} introduces three additional layers of difficulty. First, the problem becomes genuinely vector-valued, requiring the simultaneous computation of two coupled functions $w$ and $\tau$. Second, the reformulation process produces an additional scalar constraint, so that one must solve not only for the periodic orbit and its period, but also for an unfolding parameter. Third, the composition terms themselves become substantially more intricate, since nonlinear Hill functions must now be evaluated at state-dependent arguments, and their complex singularities must be avoided on a strip around the real axis to guarantee the required analyticity (see the terms $v(w(t - \tau(t)))$ and $q(w(t-\tau(t)))$ in \eqref{eq : appli original_intro}).

\begin{remark}
Although we focus on the three concrete examples \eqref{eq : toy model intro}, \eqref{eq : autonomous SDDE_intro} and \eqref{eq : appli original_intro}, each is representative of a broader class of SDDEs, and the techniques developed here are intended to be widely applicable. Since this appears to be the first constructive framework for rigorously validating periodic solutions in SDDEs, we believe that presenting the methodology in detail through carefully chosen examples is more valuable than pursuing the greatest possible level of generality at this stage.
\end{remark}

The approach presented in this work belongs to the broader category of computer-assisted proofs (CAPs) in nonlinear analysis (see, e.g., \cite{MR1420838,MR2807595,MR3444942,MR3971222,MR3990999} for general references). It is worth noting that CAPs for DDEs with constant delays are now well-established. Notable examples include rigorous existence proofs of periodic orbits in DDEs via Fourier series \cite{MR2592879,MR2871794,MR3736134,MR4413373,MR4063472}, the combined works resolving Wright’s conjecture \cite{MR3183042,MR3779642} and Jones’ conjecture \cite{MR3705681,MR3912700}, validated integration methods for Cauchy problems \cite{MR3875841,MR4783644,MR4333389}, computations of unstable manifolds \cite{MR4413389,kepley2024}, periodic solutions of impulsive delay equations \cite{MR4330857}, rigorous computations of eigenvalues \cite{MR4333397,MR4128805}, Hopf bifurcations \cite{MR4337868}, Hopf bubbles \cite{MR4825289}, and even chaotic dynamics via Sharkovskii's theorem \cite{MR4883129}.

Despite the considerable success of CAPs in the study of DDEs with constant delays, their application to SDDEs has remained limited. This is largely due to the intrinsic difficulty posed by composition-like terms of the form $\mathcal{G}(u(t - \tau(t)))$, where $\mathcal{G}$ is analytic and $u,\tau$ are periodic, which complicates both analytical and numerical approaches. Nevertheless, recent developments indicate that CAPs are beginning to play a role in this setting. Notable examples include the validated integration of differential equations with state-dependent delays \cite{MR4471505} and the use of perturbative methods to rigorously establish the existence of periodic solutions in state-dependent perturbations of ordinary differential equations \cite{MR4619367}. Related techniques have been successfully applied to problems with similarly challenging nonlinearities, often involving composition operators. For instance, the maximum modulus principle for holomorphic functions has been used to prove the existence of periodic solutions in the three-body problem \cite{MR2259202}, FFT-based methods have been employed to handle Fourier coefficients of analytic nonlinearities in KAM theory \cite{MR3709329}, and CAPs have recently been applied to demonstrate the existence of traveling waves in the two-dimensional suspension bridge equation with exponential nonlinearities \cite{MR4930548} as well as to study fixed points of renormalization (composition) operators via Chebyshev interpolation \cite{breden2024}. For further insight into the interaction between self-compositions and spectral methods, the interested reader is referred to Section 1.3 of \cite{MR4471505}.

Our approach builds on earlier computer-assisted proof techniques, but introduces several key ingredients tailored specifically to state-dependent delay equations. We represent candidate periodic solutions using Fourier expansions and work within a low-regularity Banach space of $C^1$ functions. Starting from numerical approximations $\bu(t)$ and $\btau(t)$, given as trigonometric polynomials, we derive rigorous enclosures for the Fourier coefficients of nonlinear state-dependent compositions of the form $\mathcal{G}(\bu(t - \btau(t)))$. This is achieved by adapting the validated discrete Fourier transform techniques of \cite{vandenberg2026spectral} for analytic nonlinearities: analyticity of the composition on a complex strip yields geometric decay of its Fourier coefficients, and the finitely many remaining coefficients are enclosed via trigonometric interpolation, with a rigorous control of the aliasing error. The ingredient specific to SDDEs is the rigorous evaluation of the composition on the complex strip, where the state-dependent argument $t-\btau(t)$ is itself enclosed using interval FFTs. The existence proof is then carried out via a Newton--Kantorovich argument formulated in the $C^1$ setting. Although the analysis takes place in this low-regularity space, the Fourier representation plays a central role throughout: it provides explicit computable bounds in the Newton--Kantorovich framework in terms of the Fourier coefficients of the numerical approximations, and it allows the relevant linear operators on $C^1$ functions to be effectively analyzed through their action on Fourier modes.

This last point relies on a tool that we believe is of independent interest. The Newton--Kantorovich estimates require sharp bounds on the norms of certain linear operators on $C^0$, equipped with the supremum norm, for instance operators built from an approximate inverse of the Jacobian at the numerical approximation. These operators are represented by finite matrices of Fourier coefficients. In the weighted $\ell^1$ spaces often used in Fourier-based computer-assisted proofs, the norm of such a matrix is a simple expression in its entries, whereas on $C^0$ it is much harder to compute. The key step is to write such an operator as an integral operator whose kernel is a trigonometric polynomial in two variables, with the matrix entries as coefficients. Its norm on $C^0$ is then the maximum over $t$ of the $L^1$ norm of the kernel in the second variable $s$. To bound this quantity we use classical results that control how much a trigonometric polynomial can vary between the points of a uniform grid, and that express it in terms of its values on the grid by an interpolation-type formula. We then only need the values of the kernel at the grid points, which are computed efficiently with the FFT. This is particularly useful for the large matrices needed in Section~\ref{sec:threshold_SDDE}, and complements the more direct approach of enclosing the kernel on small boxes with interval arithmetic, which works well for matrices of moderate size. None of these steps is new in itself, and some of the estimates we use are a century old, but we are not aware of this combination having been used before to compute such operator norms in computer-assisted proofs. The resulting bound is rigorous and close to sharp: the overestimation factor tends to $1$ as the grid is refined, and in our computations it stays below $1.5$. We expect this to be useful in other computer-assisted proofs that use norms of $C^0$ type rather than weighted $\ell^1$ norms of Fourier coefficients. The details are given in Section~\ref{sec:C0_operator_norms}.

Our paper is organized as follows. In Section~\ref{sec:forced_toy_model}, we present the full framework in detail for the periodically forced SDDE \eqref{eq : toy model intro}. This first section serves as the prototype for the two subsequent examples in Section~\ref{sec : autonomous SDDE} and Section~\ref{sec:threshold_SDDE}. More specifically, in Section~\ref{sec:zero_finding_problem_toy} we formulate the problem as a zero-finding equation of the form $\mathbb{F}(x)=0$. Section~\ref{sec:newton-kantorovich} introduces the Newton--Kantorovich scheme used to solve this equation. In Section~\ref{sec : approximate objects and projections}, we construct the finite-dimensional projections and describe how the numerical data are obtained, including the approximation $\bx$ of the periodic solution and an approximate inverse of the Fréchet derivative $D\mathbb{F}(\bx)$. Section~\ref{sec:Fourier_series_representation} develops the Fourier representation of the operators defining~$\mathbb{F}$. In Section~\ref{sec:fourier_coeff_delayed_term}, we present the main mechanism for obtaining rigorous bounds on Fourier coefficients of state-dependent delayed terms, which adapts the validated discrete Fourier transform techniques of \cite{vandenberg2026spectral} to state-dependent compositions.
Section~\ref{sec:bounds_toy_model} collects the explicit estimates required to implement the method in the toy model \eqref{eq : toy model intro}. In particular, Section~\ref{sec:C0_operator_norms} explains how to compute the necessary finite-dimensional operator norms. We then extend the methodology to the autonomous SDDE \eqref{eq : autonomous SDDE_intro} in Section~\ref{sec : autonomous SDDE} and to the threshold-type SDDE \eqref{eq : appli original_intro} in Section~\ref{sec:threshold_SDDE}.

\section{A periodically forced non-autonomous problem} \label{sec:forced_toy_model}

In this section we present the roadmap of our approach and introduce the main components of the computer-assisted framework in a simplified setting. Specifically, we look for a $2\pi$-periodic solution $u : \R \to \R$ to the following problem
\begin{equation}\label{eq : toy model}
 u'(t) = - u(t) - \eps_1 u(t-1-\eps_2u(t)) + \phi(t), \qquad \text{where }   \phi(t) \bydef \cos(t) + \sin(t).
\end{equation}
The above equation \eqref{eq : toy model} is introduced as a toy model chosen to possess a particularly simple structure at $\varepsilon_1=0$, for which the $2\pi$-periodic solution $u(t)=\sin(t)$ is isolated. This solution provides a convenient starting point from which the parameters $\varepsilon_1$ and $\varepsilon_2$ can be continued away from zero via numerical continuation, after which the resulting numerical solutions can be rigorously validated using computer-assisted proofs. 

This section is intended to serve as a prototype for the general method developed in the paper: rather than working at the level of maximal generality from the outset, we deliberately restrict the presentation to the periodically forced toy model \eqref{eq : toy model}. This choice is made in order to streamline the exposition and to isolate the core analytical and computational ideas without additional technical overhead. The same strategy and structure will then be systematically extended to the more involved examples in the subsequent sections. We begin in Section~\ref{sec:zero_finding_problem_toy} by reformulating the existence problem as a nonlinear operator equation of the form $\mathbb{F}(x)=0$. In Section~\ref{sec:newton-kantorovich}, we introduce the Newton–Kantorovich framework that provides the analytical backbone of the proof and underlies the construction of validated solutions. Section~\ref{sec : approximate objects and projections} is devoted to the finite-dimensional approximation stage, where we define suitable projections and describe the computation of the numerical data, including the approximate periodic solution $\bx$ and an approximate inverse for the Fréchet derivative $D\mathbb{F}(\bx)$. In Section~\ref{sec:Fourier_series_representation}, we develop the Fourier-based formulation of the operators appearing in $\mathbb{F}$, which plays a central role in both the numerical and rigorous parts of the argument. Section~\ref{sec:fourier_coeff_delayed_term} then introduces the key analytic mechanism used to control nonlinear state-dependent delay terms: we derive rigorous bounds for Fourier coefficients by combining analyticity-based decay estimates with trigonometric interpolation, following \cite{vandenberg2026spectral}. The explicit estimates required to implement the method for the toy model are collected in Section~\ref{sec:bounds_toy_model}. Finally, Section~\ref{sec:C0_operator_norms} explains how the relevant finite-dimensional operator norms are computed in practice. 

\subsection{The zero-finding problem} \label{sec:zero_finding_problem_toy}
%
Our goal in this section is to derive a zero-finding problem, for which isolated zeros will yield periodic solutions to \eqref{eq : toy model}. In order to treat all subsequent examples, especially the threshold SDDE of Section~\ref{sec:threshold_SDDE}, the notations are introduced for vector-valued functions $u : \R \to \R^m$ where $m \geq 1$. In the case of \eqref{eq : toy model}, we will have $m=1.$ Letting $k \in \mathbb{N}\cup \{0\}$,  denote by 
\begin{equation}\label{def : C0}
    C_{per}^k([-\pi,\pi]) \bydef \left\{ u\in C^k(\R,\R^m), ~ u^{(j)}(t+2\pi)=u^{(j)}(t)\ \forall t \in \R \text{ and } \forall j \in \{0, \dots, k\}\right\} ,
\end{equation}
the set of $k$ times continuously differentiable and $2\pi$-periodic functions. $C_{per}^k([-\pi,\pi])$ is associated with the usual $C^k$-norm 
\begin{align}\label{def : Ck norm}
    \|u\|_{C^k} \bydef \max\left\{ \|u\|_{C^0}, \|u^{(1)}\|_{C^0}, \dots, \|u^{(k)}\|_{C^0} \right\}, \quad \text{ for all } u \in C^k_{per}([-\pi,\pi]),
\end{align}
where $\|u\Vert_{C^0} = \max_{t \in [-\pi,\pi]} |u(t)|_\infty$ and where $|\alpha|_\infty \bydef \max\{|\alpha_1|, \dots, |\alpha_m|\}$ for all $\alpha = (\alpha_1, \dots, \alpha_m) \in \R^m$.
Since our zero-finding problem also involves finitely many scalar unknowns, we introduce the Banach space
%
\begin{equation} \label{eq:Banach_Space_H}
H \bydef \R^m \times C_{per}^{1}([-\pi,\pi]),
\end{equation}
with norm
\[
\|x\|_H = \| (\beta,u) \|_H = \max \left\{ |\beta|_\infty, \|u\Vert_{C^0}, \|u'\Vert_{C^0} \right\}.
\]
Since the norm on the space $H$ involves the $C^0$ norm of the function and of its derivative, it will be useful in practice to separate these two components. To this end, let $H_0 \bydef \R^m \times C^0_{per}([-\pi,\pi])$ and 
let $\mathbb{D} : H \to H_0 $ be defined as 
\begin{align*}
    \mathbb{D}(\beta,u) = (\beta, u').
\end{align*}
We then define the following norm and seminorm on $H$, respectively:
\begin{align}\label{eq : identity norm H1 and D}
    \|(\beta, u)\|_{H_0} = \max\{|\beta|_\infty, ~ \|u\Vert_{C^0}\} \quad\text{and}\quad \|(\beta,u)\|_{H_1} = \max\{|\beta|_\infty, ~ \|u'\Vert_{C^0}\} = \|\mathbb{D}(\beta, u)\|_{H_0} .
\end{align} 
We have that
\begin{align}\label{eq : identity norm H with H0 H1}
    \|(\beta,u)\|_H  = \max\{\|(\beta,u)\|_{H_0}, ~  \|(\beta,u)\|_{H_1}\} = \max\{\|(\beta,u)\|_{H_0}, ~  \|\mathbb{D}(\beta,u)\|_{H_0}\}. 
\end{align}
On top of introducing notation, the purpose of~\eqref{eq : identity norm H with H0 H1} is to emphasize that our entire analysis will essentially be performed in the space $H_0$, which only involves continuous functions. Notably, a strategy for getting computable and relatively sharp upper bounds of operator norms of finite-dimensional operators on $C^0$ is derived in Section \ref{sec:C0_operator_norms}. 

We want to work with the integral formulation of SDDEs, but have to take into account that the antiderivative of an arbitrary periodic function may not be periodic. We  first introduce
the mean value operator $\mathbb{M} :  C_{per}^0([-\pi,\pi]) \to \R^m$ 
\begin{equation}\label{eq : mean_integral operator}
    \mathbb{M} u \bydef \frac{1}{2\pi}\int_{-\pi}^\pi u(s)  ds,
\end{equation}
and observe that 
\[
\int_{-\pi}^{t + 2\pi} (u(s) - \mathbb{M}u) ds = \int_{-\pi}^t (u(s) - \mathbb{M}u) ds
\]
for all $t \in \R$ and all $u \in C_{per}^0([-\pi,\pi])$. Consequently, we define the bounded linear operator $\mathbb{L}_1 : C_{per}^0([-\pi,\pi]) \to C_{per}^1([-\pi,\pi])$ given by
\begin{equation}\label{eq : L1 operator functions}
    (\mathbb{L}_1u)(t) \bydef \int_{-\pi}^t (u(s) - \mathbb{M}u) ds.
\end{equation}
The antiderivative operator $\mathbb{L}_1$ is really the inverse of $\partial_t$ when restricted to mean-zero functions:
\begin{equation} \label{eq : id dt versus L}
    \partial_t \mathbb{L}_1 = I - \mathbb{M} \quad \text{and} \quad  
    (I - \mathbb{M}) \mathbb{L}_1\partial_t = I - \mathbb{M}.
\end{equation}
As for the derivative, we extend our antiderivative operator $\mathbb{L}_1$ to the product space $H_0$ by introducing $\mathbb{D}^\dagger : H_0 \to H $ defined as 
\begin{align}\label{def : D dagger}
    \mathbb{D}^\dagger(\beta,u) = (\beta, \mathbb{L}_1u).
\end{align}
The operators $\mathbb{D}$ and $\mathbb{D}^\dagger$ serve as natural maps between the spaces $H_0$ and $H$. The above notations will be used throughout the paper when discussing our different examples.  

Now we focus on the case of \eqref{eq : toy model}, and we fix the dimension $m$ to $1$. Introducing further notations in the context of \eqref{eq : toy model}, let $\psi : C^1_{per}([-\pi,\pi]) \to C^1_{per}([-\pi,\pi])$ be the composition term in \eqref{eq : toy model} given as 
\begin{align}\label{eq : psi toy model}
    \psi(u)(t) = u(t-1-\varepsilon_2u(t)).
\end{align}
Then, let $\mathbb{G} : \R^2 \to \R$ be defined as 
\begin{align*}
    \mathbb{G}(x_1,x_2) = - \varepsilon_1  x_1 -  x_2.
\end{align*}
Using the above notations, we have that if $u \in C^1_{per}([-\pi,\pi])$ solves \eqref{eq : toy model}, then it equivalently solves
\begin{align}\label{eq : toy equation with g}
    u' = g(u), \text{ where } g(u) \bydef \mathbb{G}(\psi(u),u) + \phi.
\end{align}
The notations $\psi$, $g$ and $\mathbb{G}$ will be used throughout the paper to describe the vector field of the SDDE at stake. 

From the above definitions, we define the map $f:H\to C^1_{per}$ as follows
\begin{align*}
f(\beta,u)(t) &\bydef u(t) - \beta - \int_{-\pi}^t \left( g(u)(s) - \left[ \frac{1}{2\pi} \int_{-\pi}^{\pi}  g(u)(\sigma) d\sigma \right] \right) ds 
\\
& = u(t) - \beta - (\mathbb{L}_1g(u))(t). 
\end{align*}
Denote $x=(\beta,u) \in H$, and define the map
\begin{equation} \label{eq:the_map}
\mathbb{F}(x) = \mathbb{F}(\beta,u) \bydef \begin{pmatrix} \mathbb{M}g(u) \\ f(\beta,u) \end{pmatrix}
= \begin{pmatrix} 
\displaystyle 
\mathbb{M}g(u)
\vspace{.1cm} \\ 
\displaystyle 
u - \beta - \mathbb{L}_1g(u)
\end{pmatrix}.
\end{equation}

We want to prove that $\mathbb{F}$ is well-defined from $H$ to itself and that a solution $x$ to $\mathbb{F}(x)=0$ provides a periodic solution to \eqref{eq : toy equation with g}.

\begin{lemma}\label{lem : zeros of F are solutions}
$\mathbb{F} : H \to H$ is well-defined and Fréchet differentiable.
Moreover, given $x = (\beta,u)$ we have that 
\begin{align}\label{eq : formula DF}
    D\mathbb{F}(x) = \begin{pmatrix} 
0 &  \mathbb{M} \partial_u{g}(u)
\vspace{.1cm} \\ 
\displaystyle 
-1 & I - \mathbb{L}_1 \partial_u g(u)
\end{pmatrix}.
\end{align}
Finally, if $x=(\beta,u) \in H$ solves $\mathbb{F}(x) = 0$, then $u$ is an infinitely differentiable $2\pi$-periodic solution of \eqref{eq : toy equation with g} satisfying $u(-\pi)=\beta$. 
\end{lemma}

\begin{proof}
Let $u \in C_{per}^1([-\pi,\pi])$. 
Since $\phi$ is $2\pi$-periodic and smooth, we have that $g(u)$ defined in \eqref{eq : toy equation with g} is $2\pi$-periodic. Hence $\mathbb{L}_1 g(u)$ is $2\pi$-periodic and continuously differentiable by definition \eqref{eq : L1 operator functions}.

A straightforward computation formally gives~\eqref{eq : formula DF}.
 To justify it, note that since $\mathbb{G}$ is linear and $\mathbb{M} : C^0_{per} \to \R$ and $\mathbb{L}_1 : C^0_{per} \to C^1_{per}$ are bounded linear operators, 
 $\mathbb{F}$ will be Fréchet differentiable if $u \mapsto \psi(u)$, given in \eqref{eq : psi toy model}, is Fréchet differentiable from $C^1_{per}$ to $C^0_{per}$. Now, given $u \in C^1_{per}([-\pi,\pi])$, 
\begin{equation}\label{eq : derivative of psi}
    (D\psi(u) h)(t) = - \varepsilon_2   u'(t-1-\varepsilon_2u(t)) h(t) + h(t-1-\varepsilon_2u(t)) \in C^0_{per}
\end{equation}
for all $h \in C^1_{per}([-\pi,\pi])$. Consequently, $\mathbb{F} : H \to H$ is Fréchet differentiable.


Now, assume that  $\mathbb{F}(\beta,u) = 0$. Readily, we get $\mathbb{M}g(u) = \frac{1}{2\pi}\int_{-\pi}^{\pi} g(u)(s) ds = 0$ and 
\[
u(t) = \beta + \int_{-\pi}^t \left( g(u)(s) - \mathbb{M}g(u) \right) ds
= \beta + \int_{-\pi}^t  g(u)(s) ds.
\]
Hence, $u(-\pi)=\beta$ and $u'(t) = g(u)(t)$ for all $t \in [-\pi,\pi]$. The smoothness of $u$ can be obtained using a classical bootstrapping argument.
\end{proof}
Consequently, we have defined a zero-finding problem $\mathbb{F} : H \to H$, which is Fréchet differentiable on $H$, and whose zeros provide smooth solutions to \eqref{eq : toy equation with g}.

Before providing our explicit results, we introduce the operator $\mathbb{Q}$, which will be helpful in decomposing the derivative $D\mathbb{F}(\bx)$ at a given $\bx = (\bar{\beta},\bu) \in H$.
Indeed,  by construction (cf. \eqref{eq : formula DF}), we have 
\begin{align}\label{eq : decomposition DF with DQ}
    D\mathbb{F}(\bx) = \mathbb{I} + \mathbb{D}^\dagger D\mathbb{Q}(\bx)
\end{align}
where 
\begin{align*}
\mathbb{Q}(x) = \begin{pmatrix}
     \mathbb{M} g(u)\\
    - g(u)
\end{pmatrix}, ~~    D\mathbb{Q}(\bx) = \begin{pmatrix} 
0 & \mathbb{M}\partial_u {g}(\bu)
\vspace{.1cm} \\ 
\displaystyle 
0 &  -   \partial_u g(\bu) 
\end{pmatrix} \text{ and } \mathbb{I} = \begin{pmatrix} 
0 & 0 \\ 
\displaystyle 
-1 &  I
\end{pmatrix}.
\end{align*}
In particular, we are able to factor out the regularization operator $\mathbb{D}^{\dagger}$, and write $D\mathbb{F}(\bar{x})$ as a perturbation of the operator $\mathbb{I}$. In fact, $\mathbb{D}^\dagger D\mathbb{Q}(\bx)$ is compact on $H$, since $\partial_u g(\bu)$ factors through the compact embedding $C^1_{per} \hookrightarrow C^0_{per}$, and Lemma~\ref{lem : Z1 forcing} quantifies this compactness through the truncation constant $\theta_N$ defined in \eqref{def : Kolmogorov constant expression}.

\subsection{A Newton--Kantorovich approach} \label{sec:newton-kantorovich}

%
We now assume that we have access to an approximate solution $\bar{x} \in H$, whose construction is detailed below, and we aim at proving the existence of a true solution in a vicinity of $\bar{x}$. Such an objective can be achieved using a Newton--Kantorovich approach. After having constructed a suitable approximate inverse $\mathbb{A} : H \to H$  for the Jacobian $D\mathbb{F}(\bar{x}) : H \to H$, we try to prove that the operator $x\mapsto x-\mathbb{A}\mathbb{F}(x)$ is a contraction on a neighborhood of $\bx$. The theorem (see, e.g., \cite{Yamamoto,vandenberg2026spectral}) below provides sufficient conditions supporting this approach. Here and below, for normed spaces $X$ and $Y$, $B(X,Y)$ denotes the space of bounded linear operators from $X$ to $Y$, endowed with the operator norm, and $B(X) \bydef B(X,X)$.
\begin{theorem}[\bf Newton--Kantorovich Theorem]
\label{thm:radPolyBanach}
Let $H$ be a Banach space, $\mathbb{F} : H \to H$ be Fréchet differentiable, $\bx \in H$, $\mathbb{A} \in B(H)$ an injective linear operator, and  $r_*>0$. Assume that there exist constants $Y, Z_1,Z_2 \ge 0$ such that
\begin{align}
\label{eq:Y_radPolyBanach}
	\| \mathbb{A} \mathbb{F}(\bx) \|_H & \leq Y,	\\
\label{eq:Z1_radPolyBanach}
\| I - \mathbb{A} D\mathbb{F}(\bx) \|_{B(H)} & \leq Z_1, \\
\label{eq:Z2_radPolyBanach}
\| \mathbb{A} [ D\mathbb{F}(x) - D\mathbb{F}(\bx) ] \|_{B(H)} &\leq Z_2 \|x- \bx\|_H, 
\quad \forall x \in H \text{ such that $\|x-\bx\|_H \le r_*$}. 
\end{align}
If there exists $r_0 \in (0,r_*]$ such that 
\begin{align*}
Y+Z_1 r_0 + \frac{1}{2}Z_2 r_0^2 \le r_0 ~ \text{ and } ~ 
Z_1+Z_2 r_0<1,
\end{align*}
then $\mathbb{F}$ has a unique zero $\tilde{x}$ such that $\|\tilde{x}-\bx\|_H \le r_0$.
\end{theorem}
Note that a successful application of the above result highly depends on the quality of the approximate objects $\bar{x} \in H$ and $\mathbb{A} : H \to H$. Rigorous numerics provide natural means for such objectives, and we detail the construction of $\bar{x}$ and $\mathbb{A}$ in the following section.
\subsection{Approximate objects and finite-dimensional projections}\label{sec : approximate objects and projections}
In order to numerically construct our approximate objects, we consider a finite-dimensional projection of the space $H$, which is obtained by approximating Fourier series with trigonometric polynomials of a certain fixed order $N \in \N$.
The finite-dimensional projection $\pi^{\le N}:C_{per}^{0}([-\pi,\pi])\to C_{per}^{0}([-\pi,\pi])$ is defined as follows:

\[
\pi^{\le N} u(t) = \sum_{|n| \le N} U_n e^{i n t},
\]
where $(U_n)_{n \in \mathbb{Z}}$ are the Fourier coefficients of $u$, defined as in Section~\ref{sec:Fourier_series_representation}.
We also define the complementary projection $\pi^{>N} : C_{per}^{0}([-\pi,\pi]) \to C_{per}^{0}([-\pi,\pi])$ as $\pi^{> N} \bydef I - \pi^{\le N}$ (here $I$ is the identity operator on $C_{per}^{0}([-\pi,\pi])$).
Now, given an element $x=(\beta,u)\in H$, we slightly abuse notation and consider the projection $\pi^{\le N}: H \to H$ defined by
\[
\pi^{\le N} x = \pi^{\le N}(\beta,u) = \left( \beta, \pi^{\le N} u \right) \in H.
\]
Similarly, the complementary projection $\pi^{> N}: H \to H$ is defined as $\pi^{> N} \bydef I - \pi^{\le N}$ (here $I$ is the identity operator on $H=\R^m \times C_{per}^{1}([-\pi,\pi])$), that is 
\[
\pi^{> N} x = \pi^{> N} (\beta,u) = \left( 0, \pi^{> N} u \right) \in H.
\]
Finally, we denote \[H^{\le N} \bydef \pi^{\le N} H.\]
A finite-dimensional restriction of $\mathbb{F}$ defined in \eqref{eq:the_map} is then naturally obtained by applying $\pi^{\le N}$ to both the domain and the codomain, resulting in a new operator $\mathbb{F}^{\le N}$ on $H^{\le N}$. This finite-dimensional operator $\mathbb{F}^{\le N}$ is what we work with computationally. In practice, this will also serve for the computation of an approximate solution, which we denote by
\begin{align}\label{def : approx solution forcing}
    \bar{x} = (\bar{\beta},\bar{u}) \in H^{\leq N}.
\end{align}
For instance, the approximate solution $\bar{u}$ obtained for $\varepsilon_1 = 0.5$ and $\varepsilon_2 = 2$, given as a trigonometric polynomial of order $150$, is represented in Figure~\ref{fig : toy model}, together with its Fourier coefficients.
\begin{figure}[ht]
\centering
\includegraphics[width=1\textwidth]{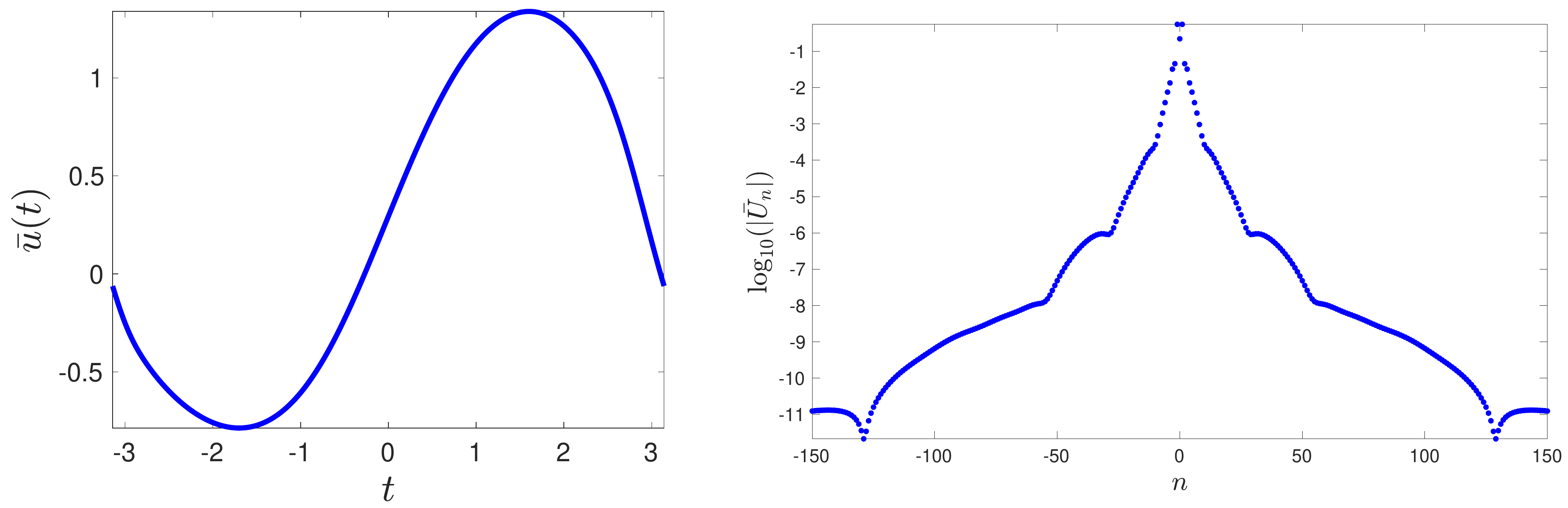}
\vspace{-.5cm}
\caption{Left: approximate solution $\bar{u}$ of \eqref{eq : toy model} for $\varepsilon_1 = 0.5$ and $\varepsilon_2 = 2$, given as a trigonometric polynomial of order $150$. Right: $\log_{10}$ of the absolute value of its Fourier coefficients $\bar{U}_n$.}
\label{fig : toy model}
\end{figure}
Now, we detail the construction of $\mathbb{A} : H \to H$, approximating the inverse of $D\mathbb{F}(\bar{x}).$ First, recall that
\[
D\mathbb{F}(\bx) = \begin{pmatrix} 
0 &  \mathbb{M} \partial_u{g}(\bu)
\vspace{.1cm} \\ 
\displaystyle 
-1 & I - \mathbb{L}_1 \partial_u g(\bu)
\end{pmatrix}
\]
at the approximation $\bx$. In fact, one can prove that $\mathbb{L}_1 \partial_u g(\bu) : C^1_{per}([-\pi,\pi]) \to C^1_{per}([-\pi,\pi]) $ is a compact operator (we will actually derive explicit estimates to quantify this notion of compactness, cf. Lemma \ref{lem : Z1 forcing}). This implies that $D\mathbb{F}(\bx)$ is a compact perturbation of the identity, and, crucially, that we should be able to construct an accurate approximation of its inverse as a finite-rank perturbation of the identity.

The operator $\mathbb{A}$ is therefore split into two components: a finite-dimensional part and an infinite-dimensional tail, using the projections $\pi^{\le N}$ and $\pi^{> N}$:

\begin{equation} \label{eq:general_A}
\mathbb{A} \bydef \mathbb{A}^{\le N}  +  \pi^{>N},
\end{equation}
where $\mathbb{A}^{\le N}$ is a linear map on the finite-dimensional space $H^{\le N}$. This decomposition allows us to treat the finite and infinite parts separately in our analysis. More precisely, $\mathbb{A}^{\le N}$ is constructed numerically as an approximate inverse of $D\mathbb{F}^{\le N}(\bar{x})$, and it is represented as a matrix. The $Z_1$ estimates derived later on in the paper show, in a quantitative manner, that such an operator $\mathbb{A}$ is indeed a suitable approximate inverse of $D\mathbb{F}(\bar{x})$, provided $N$ is taken large enough. Since $\bx$ is real-valued, $D\mathbb{F}(\bx)$ maps real-valued functions to real-valued functions; as the numerically computed matrix of $\mathbb{A}^{\le N}$ need not exactly preserve real-valued functions, we actually use $\mathrm{Re}\circ\mathbb{A} = \mathrm{Re}\circ\mathbb{A}^{\le N} + \pi^{>N}$, which is of the same form, and all the bounds derived below for $\mathbb{A}$ also hold for $\mathrm{Re}\circ\mathbb{A}$ since $\|\mathrm{Re}\, u\|_{C^0} \le \|u\|_{C^0}$.

In practice, it is usually simpler to numerically represent periodic functions by their Fourier coefficients. In fact, both $\bar{x}$ and $\mathbb{A}^{\le N}$ will be constructed using their Fourier coefficients representation. In order to do so, we provide in the next section the Fourier coefficients representation of our operators of interest.
\subsection{Fourier coefficients representation} \label{sec:Fourier_series_representation}

The objects introduced in Section \ref{sec:zero_finding_problem_toy} are defined on the infinite-dimensional space $H$. In practice, especially for computing the bounds of Theorem \ref{thm:radPolyBanach}, it will be convenient to consider a Fourier coefficients representation. Namely, we will separate a finite-dimensional polynomial part, which will be estimated thanks to operations at the level of Fourier coefficients, and a tail part which will be controlled by explicit estimates (cf. Lemma \ref{lem:piNC0C1} for instance). In what follows, we introduce notations for the Fourier coefficients representations of the objects of Section \ref{sec:zero_finding_problem_toy}, but restricted to analytic functions. This is sufficient for evaluating our bounds at the numerical approximations, which are trigonometric polynomials and hence analytic. This restriction also avoids technical issues concerning the (possible lack of) pointwise convergence of Fourier series for mere $C^0$ functions.

For $p\in[1,\infty)$, let $\ell^p$ be defined as 
\begin{equation}
    \ell^p \bydef \left\{ U = (U_n)_{n \in \mathbb{Z}}, ~ U_n \in \C^m ~ : ~ \|U\|_{\ell^p} < \infty\right\}, \text{ where } \|U\|_{\ell^p} \bydef \max_{j \in \{1, \dots, m\}} \left( \sum_{n \in \mathbb{Z}} |(U_{n})_j|^p \right)^{\frac{1}{p}}.
\end{equation} 
Moreover, define $L^2$ as 
\begin{align}
    L^2 \bydef (L^2([-\pi,\pi]))^m. 
\end{align}
Let $\mathcal{F} : L^2 \to \ell^2$ be the usual Fourier transform with inverse $\mathcal{F}^{-1} : \ell^2 \to L^2$, given as 
\begin{align}\label{def : fourier transform}
    \left(\mathcal{F}(u)\right)_n \bydef \frac{1}{2\pi} \int_{-\pi}^{\pi} u(t) e^{-int}dt \quad \text{and} \quad  \mathcal{F}^{-1}(U)(t) = \sum_{n \in \mathbb{Z}} U_n e^{int}
\end{align}
for all $u \in L^2$ and $U = (U_n)_{n \in \mathbb{Z}} \in \ell^2.$

In what follows, $u$ denotes a $2\pi$-periodic function which is analytic, that is, which extends analytically to a neighborhood of a strip $S_{\bar\rho} = \{z \in \C : |\mathrm{Im}(z)| \le \bar\rho\}$ for some $\bar\rho>0$. By the Paley--Wiener theorem (see also Section~\ref{sec:fourier_coeff_delayed_term}), its Fourier coefficients $U = \mathcal{F}(u)$ then decay geometrically, that is, $|U_n| \le C e^{-\bar\rho|n|}$ for all $n \in \Z$ and some $C>0$. In particular, $U \in \ell^1$, and the Fourier series of $u$ and of all its derivatives converge uniformly, so that $u = \mathcal{F}^{-1}(U)$ pointwise. Given an operator $\mathbb{B}$ introduced in Section~\ref{sec:zero_finding_problem_toy}, we then define its Fourier coefficients representation $B$ by
\[
BU \bydef \mathcal{F}\left(\mathbb{B}\,\mathcal{F}^{-1}(U)\right) = \mathcal{F}(\mathbb{B}u),
\]
that is, $B = \mathcal{F}\mathbb{B}\mathcal{F}^{-1}$ on such sequences.

First, note that $\mathbb{M} : C^0_{per}([-\pi,\pi]) \to \R^m$ has a simple expression acting on Fourier coefficients given as $M$, where
\begin{align*}
    (MU)_n = \begin{cases}
        U_0 &\text{ if } n=0\\
        0 &\text{ otherwise}.
    \end{cases}
\end{align*}
Moreover,  the operator $\mathbb{L}_1$ has a Fourier coefficients representation $L_1 \bydef  \mathcal{F} \mathbb{L}_1 \mathcal{F}^{-1}$ given as 
\begin{align*}
    (L_1(U))_n = \begin{cases}
    \displaystyle\sum_{k \in \mathbb{Z}\setminus \{0\}} \frac{(-1)^{k+1}U_k}{ik } ~~ &\text{ if } n = 0 \\
       \displaystyle \frac{U_n }{in }  ~~ &\text{ if } n \in \mathbb{Z}\setminus \{0\}.
    \end{cases}
\end{align*}
In practice, these formulas are applied to finitely many Fourier coefficients, so that the above operators can be understood as (finite) matrices, the contribution of the remaining coefficients being controlled analytically (see Section~\ref{sec:bounds_toy_model}).
Now, more abstractly, we define $G$ on Fourier coefficients as 
\begin{equation} \label{eq:G(U)}
    G(U) \bydef \mathcal{F}g\left(\mathcal{F}^{-1}U\right)
\end{equation}
which is the Fourier coefficients representation of $g$ given in \eqref{eq : toy equation with g}. In Section~\ref{sec:bounds_toy_model}, we shall provide a more explicit representation of $G$ in a concrete application.

Moreover, we define the Fourier coefficients representation $F$ of $\mathbb{F}$ as
\begin{align}
    F(X) \bydef \begin{pmatrix}
        M G(U)\\
        U - \mathcal{B} - L_1 G(U)
    \end{pmatrix},
\end{align}
where $X = (\beta,U) \in \R^m \times \ell^1$ and $\mathcal{B} \in \ell^1$ is defined as $\mathcal{B}_n = \beta$ if $n=0$ and  $\mathcal{B}_n = 0$ otherwise.

Furthermore, we define the equivalent of the projection operators $\pi^{\le N}$ and $\pi^{>N}$ in Fourier coefficients as follows
\begin{align}
    \Pi^{\le N} \bydef \mathcal{F}\pi^{\le N}\mathcal{F}^{-1}, ~~  \Pi^{> N} \bydef \mathcal{F}\pi^{> N}\mathcal{F}^{-1},
\end{align}
that is, $\Pi^{\le N}$ and $\Pi^{>N}$ keep the Fourier coefficients with $|n| \le N$ and $|n| > N$, respectively, and set the others to zero.
Finally, we define
\begin{align}
    A^{\leq N} \bydef \mathcal{F}\mathbb{A}^{\leq N}\mathcal{F}^{-1} \text{ and }  A \bydef \mathcal{F}\mathbb{A}\mathcal{F}^{-1}.
\end{align}
In fact, we have that 
\[
    A = A^{\leq N} + \Pi^{> N}, ~ \text{ with } A^{\leq N} = \Pi^{\le N} A^{\leq N}\Pi^{\le N}. 
\]

\subsection{Bounds for the Fourier coefficients of the state-dependent term} \label{sec:fourier_coeff_delayed_term}

In this section, we introduce the technique to obtain a rigorous enclosure of the Fourier coefficients of periodic functions of the form
\begin{equation} \label{eq:composition_term}
v(t) \bydef \bu(t - \btau(t)) = \sum_{n \in \Z} V_n e^{i n t},
\end{equation}
where $\bu(t)$ and $\btau(t)$ are real-valued trigonometric polynomials of order $N_u$ and $N_\tau$ given by 
\begin{equation} \label{eq:bu_btau}
\bu(t) = \sum_{|n| \le N_u} \bu_n e^{i n t}, \quad \bu_{-n} = \bu_{n}^* \qquad 
\text{and} \qquad
\btau(t) = \sum_{|n| \le N_\tau} \btau_n e^{i n t}, \quad \btau_{-n} = \btau_{n}^*,
\end{equation}
where $z^*$ denotes the complex conjugate of $z$. 
Note that nonlinear terms of the form \eqref{eq:composition_term} are present in all three examples considered in this paper. For instance, for the toy model \eqref{eq : toy model}, the nonlinear term is $u(t-1-\eps_2u(t))$ and hence for this case, $v(t) = \bu(t-1-\eps_2\bu(t)) = \bu(t - \btau(t))$, where $\btau(t) = 1 + \eps_2 \bu(t)$. 

We follow the approach of \cite{vandenberg2026spectral}, which we briefly recall and adapt to \eqref{eq:composition_term}. It consists of a geometric bound on all the coefficients $V_n$, used for $|n| > N_{\tt tail}$ (Section~\ref{sec:tail_coefficients}), and of an enclosure of the $V_n$ with $|n| \le N_{\tt tail}$ via trigonometric interpolation, whose aliasing error is controlled thanks to this geometric bound (Section~\ref{sec:finite_coefficients}). These calculations rely on interval arithmetic~\cite{MooKeaClo09}, which is used here both as a means of controlling rounding errors and in order to rigorously enclose the range of continuous maps over compact sets.

\subsubsection{Geometric bound on the Fourier coefficients} \label{sec:tail_coefficients}

Since $\bu$ and $\btau$ are trigonometric polynomials, $v$ extends to an entire $2\pi$-periodic function, which is real-valued on $\R$. Given $\bar\rho>0$, let $S_{\bar{\rho}} \bydef \{ z \in \C : |\mathrm{Im}(z)| \le \bar\rho\}$ and $\bar\nu \bydef e^{\bar\rho}$. By a classical Paley--Wiener estimate (shifting the line of integration from $\R$ to $\R \mp i\bar\rho$, see \cite{paley_wiener,vandenberg2026spectral}),
\begin{equation}\label{eq : tail estimates}
    |V_n| \le \frac{C}{\bar{\nu}^{|n|}}, \quad \text{for all } n \in \Z, \qquad \text{for any } C \ge \frac{1}{2\pi}\int_{0}^{2\pi} |v(t-i\bar\rho)|\,dt,
\end{equation}
where we used that $v(t+i\bar\rho) = v(t-i\bar\rho)^*$, since $v$ is real-valued on $\R$. For a complex-valued $v$ (such as the functions $e^{ik(t-1-\varepsilon_2\bu(t))}$ defining the columns of $L_3$ in Section~\ref{sec:bounds_toy_model}), \eqref{eq : tail estimates} holds for any $C$ bounding the means of $|v|$ on both lines $\mathrm{Im}(z) = \pm\bar\rho$. 

To obtain a computable $C$, we consider an even integer $N_{\tt FFT}$, the grid points $t_k \bydef \frac{2k\pi}{N_{\tt FFT}}$, $k=0,\dots,N_{\tt FFT}-1$, and $\boldsymbol{\delta} \bydef \left[-\frac{\pi}{N_{\tt FFT}},\frac{\pi}{N_{\tt FFT}}\right]$. Using the interval enclosure technique of \cite{vandenberg2026spectral} for the trigonometric polynomial $t \mapsto \btau(t-i\bar\rho)$, the interval enclosures
\begin{equation} \label{eq:definition_Uk}
\boldsymbol{\tau}_k \bydef \sum_{|n|\le N_\tau} \btau_n e^{n(\bar\rho + i\boldsymbol{\delta})} e^{int_k} \ni \btau(t-i\bar\rho), \qquad t \in t_k + \boldsymbol{\delta},
\end{equation}
are obtained for all $k$ with a single FFT. Then, for all $t \in t_k + \boldsymbol{\delta}$,
\[
v(t-i\bar\rho) = \sum_{|n| \le N_u} \bu_n e^{i n (t-i \bar{\rho} - \btau(t-i \bar{\rho}))} \in \boldsymbol{V}_k \bydef \sum_{|n| \le N_u} \bu_n e^{n (\bar{\rho} + i \boldsymbol{\delta} -  i \boldsymbol{\tau}_k)} e^{i n t_k},
\]
where, since $\boldsymbol{\tau}_k$ enters nonlinearly, the $\boldsymbol{V}_k$ are computed directly (as an interval matrix-vector product) rather than with an FFT. Consequently, \eqref{eq : tail estimates} holds with the computable constant
\begin{equation} \label{eq:asymptotic_constant}
C \bydef \frac{1}{N_{\tt FFT}} \sum_{k=0}^{N_{\tt FFT}-1} \sup |\boldsymbol{V}_k|.
\end{equation}

\subsubsection{Enclosure of the low-order coefficients} \label{sec:finite_coefficients}

Let $N_{\tt FFT} \ge 2N_{\tt tail}+2$ be even, and let $\bv$ be the trigonometric polynomial interpolating $v$ at the points $t_k$, whose coefficients
\[
\bar V_n \bydef \frac{1}{N_{\tt FFT}} \sum_{k=0}^{N_{\tt FFT}-1} v(t_k) e^{- i n t_k}, \qquad n = -\frac{N_{\tt FFT}}{2},\dots,\frac{N_{\tt FFT}}{2}-1,
\]
are enclosed with interval arithmetic and the FFT, from enclosures of
\begin{equation} \label{eq:v_at_tk}
v(t_k) = \sum_{|n| \le N_u} \bu_n e^{in(t_k - \btau(t_k))}, \qquad k = 0,\dots,N_{\tt FFT}-1,
\end{equation}
where the $\btau(t_k)$ are obtained as in \eqref{eq:definition_Uk} with $\bar\rho = 0$ and $\boldsymbol{\delta}=0$. Combining the discrete Poisson summation formula with \eqref{eq : tail estimates}, we get the aliasing error bound (see \cite{vandenberg2026spectral})
\begin{equation} \label{eq:aliasing_control}
|V_n - \bar{V}_n| \le C \, \frac{\bar{\nu}^n + \bar{\nu}^{-n}}{\bar{\nu}^{N_{\tt FFT}}-1}, \qquad |n| \le N_{\tt tail}.
\end{equation}

In practice, whenever we need to use the coefficients $V_n$ in an estimate, we therefore replace them by $\bar V_n$ if $|n|\leq N_{\tt tail}$ and by $0$ otherwise, and use~\eqref{eq:aliasing_control} (resp.~\eqref{eq : tail estimates}) to control the resulting errors.

\subsection{Computation of the bounds} \label{sec:bounds_toy_model}
To compute rigorous bounds for $Y$ and $Z_1$, we must in particular obtain a validated enclosure of $\mathbb{F}(\bx)$ and $D\mathbb{F}(\bx)$.  This requires controlling the coefficients of $g(\bu)$, that is, the sequence $G(\bU)$ with $G$ defined in \eqref{eq:G(U)}. 
Since $g(\bu)(t) = -\bu(t) - \varepsilon_1 \bu(t - 1 - \eps_2 \bu(t)) + \phi(t) $, this boils down to estimating the Fourier coefficients of the delayed periodic function $\bu(t - 1 - \eps_2 \bu(t))$, which we explained how to do in Section~\ref{sec:fourier_coeff_delayed_term}. 

\subsubsection{The \boldmath$Y$\unboldmath~bound}
%


%
%
Let us define 
\begin{align}\label{def : approx sol in sequence}
    \bar{X} \bydef (\bar{\beta}, \bar{U}) \in \R^m \times \ell^1 \quad  \text{ where } \bar{U} \bydef \mathcal{F}(\bar{u})  \text{ and } \bar{u} = \pi^{\leq N} \bar{u},
\end{align}
our approximate solution represented by its Fourier coefficients. Moreover, let us denote $(\beta_0, \bar{W}) \bydef AF(\bar{X})$, where $(\beta_0, \bar{W}) \in \R^m \times \ell^1$. 
\begin{lemma}\label{lem : Y bound toy model}
    Let $Y$ be a bound satisfying
    \begin{align}
        Y \geq \max\left\{|\beta_0|,~ \|\bar{W}\|_{\ell^1}, ~ \sum_{n \in \mathbb{Z}} |n\bar{W}_n| \right\}.
    \end{align}
    Then, $\|\mathbb{A}\mathbb{F}(\bx)\|_H \leq Y.$
\end{lemma}

\begin{proof}
    The proof is a direct consequence of the inequalities $\|\mathcal{F}^{-1}(U)\Vert_{C^0} \leq \|U\|_{\ell^1}$ and $\|\left(\mathcal{F}^{-1}(U)\right)'\Vert_{C^0} \leq \sum_{n\in\Z} |n|\vert U_n\vert$ for all $U \in \ell^1$ with sufficient decay, applied to $\bar{W}$.
\end{proof}
 In practice, the rigorous evaluation of $F(\bar{X})$ is achieved using the technique presented in Section~\ref{sec:fourier_coeff_delayed_term}. Note that $F(\bar{X})$ has infinitely many nonzero coefficients, since $v(t) = \bu(t-1-\varepsilon_2\bu(t))$ is not a trigonometric polynomial. Denoting by $V$ the Fourier coefficients of $v$, and since $\bu$ and $\phi$ have order at most $N$, we have $(F(\bar{X}))_n = \frac{\varepsilon_1 V_n}{in}$ for $|n|>N$. As $A = A^{\le N} + \Pi^{>N}$, the coefficients $\bar{W}_n$ with $|n|>N$ are therefore controlled by the enclosures \eqref{eq:aliasing_control} for $N<|n|\le N_{\tt tail}$ (assuming $N \le N_{\tt tail}$), and by the geometric bound \eqref{eq : tail estimates} for $|n|>N_{\tt tail}$, which yields for instance $\sum_{|n|>N_{\tt tail}} |n\bar{W}_n| \le \frac{2|\varepsilon_1| C \bar\nu^{-N_{\tt tail}}}{\bar\nu-1}$. Similarly, the zeroth coefficient $(L_1G(\bar U))_0$ involves all the coefficients of $G(\bar U)$, and the contribution of those with $|n|>N_{\tt tail}$ is bounded using \eqref{eq : tail estimates}.

\subsubsection{Computation of the \boldmath$Z_2$\unboldmath~bound}\label{ssec : Z2 forcing}
We are now in a position to present the computation of the bound $Z_2$. 
For this specific example it does not depend on $r$, but the bound $Z_2$ will depend on $r$ for the examples of Section~\ref{sec : autonomous SDDE} and Section~\ref{sec:threshold_SDDE}. Therefore, we choose to keep the $r$ in the notation even here.
%
%
%
\begin{lemma}\label{lem : Z2 forcing}
    Let $r >0$ and let $\bar{x} = (\bar{\beta},\bu)$ be as in \eqref{def : approx solution forcing}. Then, let $\hat{Z}_2(r)$ be defined as follows
   \begin{equation}\label{def : Z2 hat}
    \begin{aligned}
   \hat{Z}_{2}(r) \bydef & ~~  |\varepsilon_1 \varepsilon_2|( |\varepsilon_2| \|\bar{u}''\|_{C^0}  + 2).
\end{aligned}
\end{equation}
Defining $Z_2(r)$ as 
\begin{align*}
    Z_2(r) \bydef \max\{\|\mathbb{A}\mathbb{D}^\dagger\|_{B(H_0)}, ~ \|\mathbb{D}\mathbb{A}\mathbb{D}^\dagger\|_{B(H_0)}\} \hat{Z}_2(r)
\end{align*}
we have that $\|\mathbb{A}(D\mathbb{F}(\bar{x} +x) - D\mathbb{F}(\bar{x}))\|_{B(H)} \leq Z_2(r)r$ for all $x \in B_r(0) \subset H$.
\end{lemma}
\begin{proof}
First, we prove that 
\begin{align}\label{eq : inequality g and hat Z2}
    \|\partial_u g(\bu + u) - \partial_u {g}(\bu) \Vert_{B(C^1,C^0)}
    &\leq  \hat{Z}_{2}(r) r 
\end{align}
for all $u \in B_r(0) \subset C^1_{per}([-\pi,\pi])$. Given $v \in C^1_{per}([-\pi,\pi])$,  $g$ is differentiable at $v$ and we have
\begin{equation}\label{eq : derivative of g toy model}
    \begin{aligned}
    \partial_u g(v)h =  \partial_{2} \mathbb{G}(\psi(v), v)h  + \partial_1 \mathbb{G}(\psi(v),v) D\psi(v)h &= - h - \varepsilon_1 D\psi(v)h \\
    & =  - h - \varepsilon_1 \left(- \varepsilon_2   v'(t-1-\varepsilon_2v(t)) h + h(t-1-\varepsilon_2v(t))\right)
\end{aligned}
\end{equation}
where we used \eqref{eq : derivative of psi}.
Let $u \in B_r(0) \subset C^1_{per}([-\pi,\pi])$ and $h \in B_1(0) \subset C^1_{per}([-\pi,\pi])$. Using the above, we have that 
\begin{align*}
    \|\partial_u g(\bu + u)h - \partial_u {g}(\bu)h\Vert_{C^0} &\leq  |\varepsilon_1 \varepsilon_2| \|((\bar{u}+u)'(t-1-\varepsilon_2(\bar{u}+u)(t)) - \bar{u}'(t-1-\varepsilon_2\bar{u}(t)) )h\|_{C^0}\\
    & \quad +|\varepsilon_1| \|h(t-1-\varepsilon_2(\bar{u}+u)(t)) - h(t-1-\varepsilon_2\bar{u}(t))\|_{C^0}.
\end{align*}
Using the mean value inequality, we get
\begin{align*}
    \|h(t-1-\varepsilon_2(\bar{u}+u)(t)) - h(t-1-\varepsilon_2\bar{u}(t))\|_{C^0} \leq |\varepsilon_2|\|h'\|_{C^0} \|{u}\|_{C^0}.
\end{align*}
On the other hand, we have
\begin{align*}
    &\|((\bar{u}+u)'(t-1-\varepsilon_2(\bar{u}+u)(t)) - \bar{u}'(t-1-\varepsilon_2\bar{u}(t)) )h\|_{C^0}\\
    &\qquad\qquad\qquad\leq \|(\bar{u}'(t-1-\varepsilon_2(\bar{u}+u)(t)) - \bar{u}'(t-1-\varepsilon_2\bar{u}(t)) )h\|_{C^0} +  \|u'(t-1-\varepsilon_2(\bar{u}+u)(t))h\|_{C^0}\\
    &\qquad\qquad\qquad\leq |\varepsilon_2| \|\bar{u}''\|_{C^0} \|u\|_{C^0} \|h\|_{C^0} + \|u'\|_{C^0} \|h\|_{C^0}.
\end{align*}
Using that $\|h\|_{C^1} \leq 1$
and that $\|u\|_{C^1} \leq r$, we get
\begin{align*}
    \|\partial_u g(\bu + u)h - \partial_u {g}(\bu)h\Vert_{C^0} &\leq |\varepsilon_1 \varepsilon_2|( |\varepsilon_2| \|\bar{u}''\|_{C^0} r + 2r) = \hat{Z}_{2}(r) r.
\end{align*}
Overall, we have proven \eqref{eq : inequality g and hat Z2}.
Let us now compute the upper bound $Z_2$. Let $x = (\beta,u) \in B_r(0) \subset H$. In particular, $u \in B_r(0)\subset C^1_{per}([-\pi,\pi]) $. Recalling \eqref{eq : decomposition DF with DQ}, we have
\begin{align}\label{eq : difference of DF}
 D\mathbb{F}(\bx + x) - D\mathbb{F}(\bx)   = \mathbb{D}^\dagger \begin{pmatrix} 
0 & \mathbb{M}
\vspace{.1cm} \\ 
\displaystyle 
0 &   -I
\end{pmatrix} \begin{pmatrix}
 0&   0\\
 0&   \partial_u g(\bu +u) - \partial_u g(\bu)
\end{pmatrix}.
\end{align}
Picking an arbitrary $z \in H$ and recalling~\eqref{eq : identity norm H with H0 H1}, we get
\begin{align*}
    \|\mathbb{A}(D\mathbb{F}(\bx + x)-D\mathbb{F}(\bx))z\|_H = \max\{\|\mathbb{A}(D\mathbb{F}(\bx + x)-D\mathbb{F}(\bx))z\|_{H_0}, \|\mathbb{A}(D\mathbb{F}(\bx + x)-D\mathbb{F}(\bx))z\|_{H_1}\}.
\end{align*}
We estimate each of the terms on the right-hand side of the above identity.
First, straightforward computations yield
\begin{align}\label{eq : identity norm M}
    \left\|\begin{pmatrix} 
0 & \mathbb{M}
\vspace{.1cm} \\ 
\displaystyle 
0 &   -I
\end{pmatrix}\right\|_{B(H_0)} = 1.
\end{align}
Using \eqref{eq : difference of DF} combined with \eqref{eq : identity norm M}, we have that 
\begin{align*}
\|\mathbb{A}(D\mathbb{F}(\bx + x)-D\mathbb{F}(\bx))\|_{B(H,H_0)} &\leq  \|\mathbb{A}\mathbb{D}^\dagger\|_{B(H_0)}
\left\| 
\begin{pmatrix} 
0 & \mathbb{M}
\vspace{.1cm} \\ 
\displaystyle 
0 &   -I
\end{pmatrix} 
\begin{pmatrix}
 0&   0\\
 0&   \partial_u g(\bu +u) - \partial_u g(\bu)
\end{pmatrix} \right\|_{B(H,H_0)}\\
&\leq \|\mathbb{A}\mathbb{D}^\dagger\|_{B(H_0)} \|\partial_u g(\bu +u) - \partial_u g(\bu) \|_{B(C^1,C^0)}.
\end{align*}
On the other hand, using \eqref{eq : identity norm H1 and D}, we have
\begin{align*}
    \|\mathbb{A}(D\mathbb{F}(\bx + x)-D\mathbb{F}(\bx))z\|_{H_1} &=   \|\mathbb{D}\mathbb{A}(D\mathbb{F}(\bx + x)-D\mathbb{F}(\bx))z\|_{H_0}  \\
    &\leq  \|\mathbb{D}\mathbb{A}\mathbb{D}^\dagger\|_{B(H_0)}\|\partial_u g(\bu +u) - \partial_u g(\bu)\|_{B(C^1,C^0)}\|z\|_H.
\end{align*}
We conclude the proof using \eqref{eq : inequality g and hat Z2}.
\end{proof}
The quantity $\hat{Z}_2(r)$ is computable, and therefore getting a computable $Z_2(r)$ bound only requires computable upper bounds for the operator norms of $\mathbb{A}\mathbb{D}^\dagger$ and $\mathbb{D}\mathbb{A}\mathbb{D}^\dagger$, which can be achieved using Lemma~\ref{lem : from infinite to finite} and Lemma~\ref{lem : sampling operator norm}.
\subsubsection{Truncation error estimates}

Before presenting the computation of the bound $Z_1$, we first provide several Fourier analysis estimates allowing us to control the truncation operator $\pi^{>N}$. We first introduce, for $N\in\N\cup\{0\}$,
\begin{align}\label{def : Kolmogorov constant expression}
    \theta_N = \frac{1}{\pi} \int_{0}^{\pi} |\kappa_N(t)| dt \quad  \text{ where } \kappa_N(t) \bydef 2\sum_{k=N+1}^\infty \frac{\sin(kt)}{k} =  \pi - t - 2 \sum_{k=1}^N \frac{\sin(kt)}{k}, \quad t \in (0,2\pi).
\end{align}
\begin{lemma}
\label{lem:piNC0C1}
For all $h \in C^1_{per}$ and $N\in\N\cup\{0\}$,
\begin{align}
\label{def : Kolmogorov constant}
    \|\pi^{>N} h\|_{C^0} \leq \theta_N \|h'\|_{C^0},
\end{align}
and the constant $\theta_N$ is optimal. 
\end{lemma}
\begin{proof}
    This was first proven by Kolmogorov in~\cite{Kolmogorov1935}, and follows from the observation that $\pi^{>N} h$ can be written as a convolution against $h'$, since convolution corresponds to multiplication of Fourier coefficients and $\kappa_N$ has Fourier coefficients $\frac{1}{in}$ for $|n| > N$ and $0$ otherwise:
\begin{equation}
\label{eq:truncationkernel}
    (\pi^{>N} h)(t) = (\kappa_N\ast h')(t) = \frac{1}{2\pi}\int_{-\pi}^{\pi}\kappa_N(t-s) h'(s) ds. \quad\qedsymbol
\end{equation}
\let\qed\relax 
\end{proof}
Kolmogorov also proved that $\theta_N$ behaves asymptotically like $\theta_N = \frac{4}{\pi^2} \frac{\ln(N)}{N} + \mathcal{O}(\frac{1}{N})$. We derive below computable upper bounds for $\theta_N$ that can be used for a fixed $N$. Before doing so, we collect additional truncation estimates involving $\theta_N$ that will prove useful later on.
\begin{lemma}
\label{lem:piNL1norm}
For all $N\in\N$, with $\theta_N$ as in~\eqref{def : Kolmogorov constant expression}, we have
\begin{align}
\label{eq:piNL1}
\|\pi^{>N}\mathbb{L}_1 \|_{B(C^0)}\leq \theta_N.
\end{align}
Moreover, for all $h \in C^0_{per}$,
\begin{align}
\label{eq:piNL1piNcst}
    \pi^{\leq N}\mathbb{L}_1 \pi^{>N} h = \mathbb{M} \mathbb{L}_1 \pi^{>N} h = -(\pi^{>N} \mathbb{L}_1 h)(\pi),
\end{align}
which shows that $\pi^{\leq N}\mathbb{L}_1 \pi^{>N}$ ranges over constant functions only, and, combined with~\eqref{eq:piNL1}, that
\begin{align}
\label{eq:piNL1piNnorm}
    \|\pi^{\leq N}\mathbb{L}_1 \pi^{>N} \|_{B(C^0)}  \leq \theta_N.
\end{align}
\end{lemma}

\begin{proof}
Let $h\in C^0_{per}$. According to~\eqref{eq:truncationkernel} and~\eqref{eq : id dt versus L}, 
\begin{align*}
    \pi^{>N}\mathbb{L}_1 h &= \kappa_N \ast \left( \partial_t \mathbb{L}_1 h\right) = \kappa_N \ast \left( (I-\mathbb{M}) h\right) = \kappa_N \ast h,
\end{align*}
as $\kappa_N$ has mean $0$, which yields \eqref{eq:piNL1}. 
Next, recall that $\mathbb{L}_1 h \in C^1_{per}$ with $(\mathbb{L}_1 h)' = h - \mathbb{M}h$, so that an integration by parts gives $(\mathcal{F}(\mathbb{L}_1 h))_n = \frac{(\mathcal{F}(h))_n}{in}$ for all $n \neq 0$. Applying this to $h$ and to $\pi^{>N} h$, whose Fourier coefficients vanish for $|n| \leq N$, we see that $\mathbb{L}_1\pi^{>N} h - \mathbb{M}\mathbb{L}_1\pi^{>N} h$ and $\pi^{>N}\mathbb{L}_1 h$ are continuous functions with the same Fourier coefficients, namely $\frac{(\mathcal{F}(h))_n}{in}$ for $|n| > N$ and $0$ otherwise. Hence
\begin{align}
\label{eq:identityLM}
    \mathbb{L}_1\pi^{>N} h = \mathbb{M}\mathbb{L}_1\pi^{>N} h + \pi^{>N}\mathbb{L}_1 h.
\end{align}
Applying $\pi^{\leq N}$ to~\eqref{eq:identityLM} gives the first equality in~\eqref{eq:piNL1piNcst}, and evaluating~\eqref{eq:identityLM} at $\pi$, using that $(\mathbb{L}_1\pi^{>N} h)(\pi) = 0$, gives the second one. Finally,~\eqref{eq:piNL1piNnorm} follows from~\eqref{eq:piNL1piNcst} and~\eqref{eq:piNL1}.
\end{proof}

In the remainder of this section, we present a procedure for computing a rigorous and tight enclosure of $\theta_N$ with high accuracy, for any fixed $N$. Our method is based on the following central proposition, establishing some properties of the function $\kappa_N.$
\begin{proposition}\label{prop : computation Kolmogorov constant}
    Let $N \in \mathbb{N}$ and let $\kappa_N$ be defined as in \eqref{def : Kolmogorov constant expression}. Then, we have the following properties
    \begin{itemize}
        \item $\kappa_N'(t) =  -\frac{\sin\left((N+\frac{1}{2})t\right)}{\sin(t/2)}$.
        \item $\kappa_N$ is strictly monotone on $(t_{j-1}, t_{j})$ for $j = 1, \dots, N+1$, where $t_j = \frac{2\pi j }{2N+1}$ for $j = 0, \dots, N$ and $t_{N+1} = \pi$.
        \item Let $\mathcal{K}_N(t) = \pi t - \frac{t^2}{2}  + 2 \sum_{k=1}^N \frac{\cos(kt)}{k^2}$, then $\mathcal{K}_N' = \kappa_N$ on $(0,2\pi)$.
    \end{itemize}
    Moreover, $\kappa_N$ possesses a unique zero $\tau_j$ on each interval $(t_{j-1}, t_j)$ for all $j = 1, \dots, N$, and no other zero in $(0,\pi)$. Finally, we have 
    \begin{align}\label{eq : value of theta}
        \theta_N = \frac{2}{\pi} \sum_{j=1}^{N} (-1)^{j-1} \mathcal{K}_N(\tau_j)
    - \frac{2}{\pi} \sum_{k=1}^{N} \frac{1}{k^2}
    + (-1)^N \left( \frac{\pi}{2}
      + \frac{2}{\pi} \sum_{k=1}^{N} \frac{(-1)^k}{k^2} \right).
    \end{align}
\end{proposition}

\begin{proof}
    The three properties of $\kappa_N$ listed in the bullet points follow from direct computations. For $i = 1, \dots, N$, let $I_i \bydef \int_{t_i}^{t_{i+1}} \frac{|\sin((N+\frac{1}{2})t)|}{\sin(t/2)}dt > 0$. The sequence $(I_i)$ is decreasing, since $t \mapsto \frac{1}{\sin(t/2)}$ is decreasing on $(0,\pi)$ and each interval $(t_i,t_{i+1})$ with $i < N$ covers exactly one period $\frac{2\pi}{2N+1}$ of $|\sin((N+\frac{1}{2})t)|$. The last interval $(t_N,\pi)$ covers only half a period. Since $\kappa_N(\pi) = 0$ and $\sin((N+\frac{1}{2})t)$ has the sign of $(-1)^i$ on $(t_i,t_{i+1})$, we get
    $
        (-1)^j\kappa_N(t_j) = (-1)^{j+1}\int_{t_j}^{\pi}\kappa_N'(t)dt = I_j - I_{j+1} + I_{j+2} - \dots + (-1)^{N-j} I_N > 0
    $
    for all $j = 1, \dots, N$. Since moreover $\kappa_N(t) \to \pi$ as $t \to 0^+$ and $\kappa_N(\pi) = 0$, the monotonicity of $\kappa_N$ provides a unique zero $\tau_j$ in each interval $(t_{j-1}, t_{j})$, $j = 1, \dots, N$, and no zero in $(t_N,\pi)$. Hence, denoting $\tau_0 = 0$ and 
    $\tau_{N+1} = \pi$, $\kappa_N$ has the sign of $(-1)^j$ on $(\tau_j,\tau_{j+1})$ for $j = 0, \dots, N$ and, using a telescopic sum argument, 
    \begin{align*}
        \pi\theta_N = \int_{0}^\pi |\kappa_N(t)| dt =  \sum_{j=0}^{N} (-1)^j \big( \mathcal{K}_N(\tau_{j+1}) - \mathcal{K}_N(\tau_j) \big)
   = -\mathcal{K}_N(0) + (-1)^N \mathcal{K}_N(\pi) + 2\sum_{j=1}^{N} (-1)^{j-1} \mathcal{K}_N(\tau_j).
    \end{align*}
 Substituting
$\mathcal{K}_N(0) = 2\sum_{k=1}^N k^{-2}$ and
$\mathcal{K}_N(\pi) = \frac{\pi^2}{2} + 2\sum_{k=1}^N (-1)^k k^{-2}$ and dividing by $\pi$
gives the claim.
\end{proof}
Using the previous proposition, one can get a rigorous enclosure of $\theta_N$ by rigorously enclosing the zeros $\tau_j$ of the function $\kappa_N$. Using the expression for $\kappa_N$ given in the right-hand side of~\eqref{def : Kolmogorov constant expression}, enclosing its roots is straightforward using rigorous numerics (cf. \cite{code_proof}).

Finally, we note that $\theta_N \leq 1$ for all $N \geq 1$.  For $N \geq 2$, this follows from the Cauchy--Schwarz inequality and Parseval's identity, which give $\theta_N \leq \big(2\sum_{k>N}k^{-2}\big)^{1/2} < \sqrt{2/N}$, while for $N = 1$, Proposition~\ref{prop : computation Kolmogorov constant} yields $\theta_1 < 0.84$. These estimates only serve to check the assumption in Lemma~\ref{lem : Z1 forcing} below. In the bounds of our computer-assisted proofs, we use the much sharper enclosure of $\theta_N$ obtained from Proposition~\ref{prop : computation Kolmogorov constant}.

\subsubsection{Computation of the \boldmath$Z_1$\unboldmath~bound}

We now derive, in several steps, a computable bound $Z_1$ controlling $\|I - \mathbb{A}D\mathbb{F}(\bar{x})\|_{B(H)}$. This is the most crucial estimate for the proof, which guarantees (if $Z_1<1$) that the derivative of the operator $x\mapsto x-\mathbb{A}\mathbb{F}(x)$ at $\bx$ has norm less than $1$. The first step is Lemma~\ref{lem : Z1 forcing}, where we express $Z_1$ in terms of operator norms of derivatives of $g$. Quantitative estimates for these derivatives are then obtained in Lemma~\ref{lem : Z1_12 forcing}. Finally, Lemma~\ref{lem : from infinite to finite} reduces the operator norms of infinite-dimensional operators to operator norms of finite-dimensional ones, and the latter can be efficiently and rigorously estimated using Lemma~\ref{lem : sampling operator norm}.
%
%

%
%
\begin{lemma}\label{lem : Z1 forcing}
   Assume that $\theta_N \leq 1$. Let $Z_{1,0}, Z_{1,1}$ and $Z_{1,2}$ be bounds satisfying
    \begin{equation}\label{eq : components Z1}
        \begin{aligned}
            Z_{1,0} &\geq \max\left\{ \|\pi^{\leq N}(I - \mathbb{A}D\mathbb{F}(\bar{x})) \pi^{\leq N} \|_{B(H,H_0)} ,~  \|\pi^{\leq N}\mathbb{D}(I - \mathbb{A}D\mathbb{F}(\bar{x}))  \pi^{\leq N}\|_{B(H,H_0)}  \right\} \\
            Z_{1,1} &\geq \max\{ \|\mathbb A^{\leq N} \mathbb{D}^\dagger\|_{B(H_0)}, ~ \|\mathbb{D}\mathbb A^{\leq N} \mathbb{D}^\dagger\|_{B(H_0)} \} \|\partial_u g(\bar{u})\Vert_{B(C^0)} \theta_N\\
            Z_{1,2} &\geq  \sup_{\|h\|_{C^1}=1} \|\partial_t\left[ \partial_u g(\bu)  h\right]\Vert_{C^0} \theta_N.
        \end{aligned}
    \end{equation}
    Then, defining $Z_1$ as 
    \begin{align}
        Z_1 = Z_{1,0} + Z_{1,1}  + Z_{1,2}, 
    \end{align}
    we have that $\|I - \mathbb{A}D\mathbb{F}(\bar{x})\|_{B(H)} \leq Z_1.$
\end{lemma}

\begin{proof}
    Let $x = (\beta, u) \in B_1(0) \subset H$. Notice the following decomposition using the projection operators defined in Section~\ref{sec : approximate objects and projections}
    \begin{align}\label{eq : Z1 decomposition proof}
        \|(I - \mathbb{A}D\mathbb{F}(\bar{x}))x\|_H &\leq \|\pi^{\leq N}\left(I - \mathbb{A}D\mathbb{F}(\bar{x})\right)\pi^{\leq N}x\|_H  + \|\pi^{\leq N} \left( I- \mathbb{A}D\mathbb{F}(\bar{x})\right)\pi^{> N}x\|_H \\
        &\qquad\qquad
        + \|\pi^{> N}(I - \mathbb{A}D\mathbb{F}(\bar{x}))x\|_H,  
    \end{align}
    where we used that $I = \pi^{\leq N} + \pi^{>N}$ and the triangle inequality. We now estimate the three terms on the right-hand side separately. We are going to show that the first one is bounded by $Z_{1,0}$, the second one by $Z_{1,1}$ and the last one by $Z_{1,2}$. To do so, we heavily use the identity \eqref{eq : identity norm H with H0 H1} to decompose the norm of $H$. We first estimate all the bad terms, namely the $H_1$-seminorms.

First, recall the decomposition \eqref{eq:general_A} of $\mathbb{A} = \mathbb{A}^{\leq N} + \pi^{>N}$.   
Then, using \eqref{eq : identity norm H1 and D} and the fact that $\pi^{\leq N} \pi^{>N} = 0$, we have
\begin{align*}
   \|\pi^{\leq N}(I - \mathbb{A}D\mathbb{F}(\bar{x})) \pi^{>N}x\|_{H_1} &=  \|\mathbb{D}\pi^{\leq N}(I - \mathbb{A}D\mathbb{F}(\bar{x})) \pi^{>N}x\|_{H_0}\\
   &= \|\mathbb{D} \mathbb{A}^{\leq N} D\mathbb{F}(\bar{x}) \pi^{>N}x\|_{H_0}.
\end{align*}
Next, notice that $\mathbb{A}^{\leq N} \mathbb{I} \pi^{>N} = \mathbb{A}^{\leq N}  \pi^{>N} =0$. Consequently, combining \eqref{eq : decomposition DF with DQ} and \eqref{def : Kolmogorov constant}, we have 
\begin{align}\label{eq : Z1 step 0}
   \|\pi^{\leq N}(I - \mathbb{A}D\mathbb{F}(\bar{x})) \pi^{>N}x\|_{H_1} 
    &\leq \|\mathbb{D} \mathbb{A}^{\leq N}\mathbb{D}^\dagger\|_{B(H_0)}\|  D\mathbb{Q}(\bar{x})\|_{B(H_0)} \theta_N \|u'\Vert_{C^0}.
\end{align}
Now,  using \eqref{eq : identity norm M}, we get
\begin{equation}\label{eq : norm of DQ with Dg}
    \| D\mathbb{Q}(\bar{x})\|_{B(H_0)} \leq \|\partial_u g(\bar{u})\|_{B(C^0)},
\end{equation}
and combining the previous estimates,
we obtain that 
\begin{align*}
   \|\pi^{\leq N}(I - \mathbb{A}D\mathbb{F}(\bar{x})) \pi^{>N}x\|_{H_1} 
    &\leq Z_{1,1},
\end{align*}
since $x \in B_1(0).$ Now, using \eqref{eq : decomposition DF with DQ} and \eqref{eq:general_A}, observe that 
\begin{align*}
    \pi^{> N}(I - \mathbb{A}D\mathbb{F}(\bar{x}))  = \pi^{> N} - \pi^{>N}\left(\mathbb{I} + \mathbb{D}^\dagger D\mathbb{Q}(\bar{x})\right) =   - \pi^{>N} \mathbb{D}^\dagger D\mathbb{Q}(\bar{x}), 
\end{align*}
and thus
\begin{align*}
  \|\pi^{> N}(I - \mathbb{A}D\mathbb{F}(\bar{x})) x\|_{H_1}  =  \|\mathbb{D}\pi^{> N}(I - \mathbb{A}D\mathbb{F}(\bar{x})) x\|_{H_0} = \|\mathbb{D}\pi^{> N}\mathbb{D}^\dagger D\mathbb{Q}(\bx) x\|_{H_0}.
\end{align*}
According to the first identity in~\eqref{eq : id dt versus L}, $\mathbb{D}\pi^{> 0}\mathbb{D}^\dagger = \pi^{> 0}$, which readily yields $\mathbb{D}\pi^{> N}\mathbb{D}^\dagger = \pi^{> N}$, and thus
\begin{align*}
  \|\pi^{> N}(I - \mathbb{A}D\mathbb{F}(\bar{x})) x\|_{H_1}  = \|\pi^{>N} D\mathbb{Q}(\bx) x\|_{H_0} =   \|\pi^{>N}  \partial_u g(\bu) u\Vert_{C^0}.
\end{align*}
Then, using \eqref{def : Kolmogorov constant} we have 
\begin{align*}
    \|\mathbb{D}\pi^{> N}(I - \mathbb{A}D\mathbb{F}(\bar{x})) x\|_{H_0}  \leq \theta_N \sup_{ \|h\|_{C^1}=1}\|\partial_t[\partial_u g(\bu) h]\Vert_{C^0}  \leq Z_{1,2}.
\end{align*}
Since the first term in \eqref{eq : Z1 decomposition proof} is bounded by $Z_{1,0}$ by definition, we obtain
\begin{align*}
    \|(I - \mathbb{A}D\mathbb{F}(\bar{x})) x\|_{H_1}  =
    \|\mathbb{D}(I - \mathbb{A}D\mathbb{F}(\bar{x})) x\|_{H_0} 
    \leq Z_{1,0}  +  Z_{1,1} + Z_{1,2}.
\end{align*}
Then, using very similar arguments, 
\begin{align*}
    \|\pi^{\leq N}(I - \mathbb{A}D\mathbb{F}(\bar{x})) \pi^{>N}x\|_{H_0}
    \leq  \|  \mathbb{A}^{\leq N}\mathbb{D}^\dagger\|_{B(H_0)}\| D\mathbb{Q}(\bar{x})\|_{B(H_0)} \theta_N \leq Z_{1,1}.
\end{align*}
Moreover, the first identity of~\eqref{eq : id dt versus L} shows that $\partial_t^2 \mathbb{L}_1 = \partial_t$, and applying~\eqref{def : Kolmogorov constant} twice (which is possible since $\pi^{>N} = (\pi^{>N})^2$ commutes with $\partial_t$) and using $\theta_N \leq 1$ in the last step, we get
\begin{align*}
    \|\pi^{> N}(I - \mathbb{A}D\mathbb{F}(\bar{x})) x\|_{H_0} &= \|\pi^{> N}\mathbb{D}^\dagger D\mathbb{Q}(\bx) x\|_{H_0}  \\
    &=   \| \pi^{>N} \mathbb{L}_1 \partial_u g(\bu)  u\Vert_{C^0}\\
    &\leq \theta_N^2 \| \partial_t^2  \mathbb{L}_1 \partial_u g(\bu)  u\Vert_{C^0} \\
    &=  \theta_N^2 \|   \partial_t[ \partial_u g(\bu)  u] \Vert_{C^0} \\
    & \leq Z_{1,2},
\end{align*}
and thus
\begin{align*}
    \|(I - \mathbb{A}D\mathbb{F}(\bar{x})) x\|_{H_0}  
    \leq Z_{1,0} +  Z_{1,1} + Z_{1,2},
\end{align*}
which finishes the proof.
\end{proof}

We made progress toward an explicit $Z_1$ bound with Lemma~\ref{lem : Z1 forcing}, but it is not yet obvious that the upper bounds for $Z_{1,0}$, $Z_{1,1}$ and $Z_{1,2}$ given there are actually computable. We first derive explicit bounds for the derivatives of $g$ appearing in $Z_{1,1}$ and $Z_{1,2}$.

\begin{lemma}
\label{lem : Z1_12 forcing}
    Let $Z_{1,1}$ and $Z_{1,2}$ be defined as 
    \begin{align*}
        Z_{1,1} &\bydef \max\{ \|  \mathbb{A}^{\leq N}\mathbb{D}^\dagger\|_{B(H_0)}, ~ \| \mathbb{D} \mathbb{A}^{\leq N}\mathbb{D}^\dagger\|_{B(H_0)} \}\left(1 + |\varepsilon_1|(1+|\varepsilon_2| \|\bu'\|_{C^0})\right) \theta_N\\
        Z_{1,2} &\bydef \left( 1 + {|\varepsilon_1|}(1 + |\varepsilon_2| \|\bar{u}'\Vert_{C^0}) \left(1 + |\varepsilon_2|\|\bar{u}''\Vert_{C^0} \right) +     |\varepsilon_1\varepsilon_2|\|\bar{u}'\Vert_{C^0}\right) \theta_N.
    \end{align*}
    Then $Z_{1,1}$ and $Z_{1,2}$ satisfy \eqref{eq : components Z1}.
\end{lemma}

\begin{proof}
 We first focus on $Z_{1,1}$. Using Lemma \ref{lem : Z1 forcing}, we simply need to compute $ \|\partial_u g(\bar{u})\Vert_{B(C^0)}$. To this end, we use \eqref{eq : derivative of g toy model} and directly get
\begin{align*}
    \|\partial_u g(\bar{u})\Vert_{B(C^0)} &\leq   \sup_{\|u\Vert_{C^0}=1}  \| u + \varepsilon_1 \left(- \varepsilon_2   \bar{u}'(t-1-\varepsilon_2\bar{u}(t)) u + u(t-1-\varepsilon_2\bar{u}(t))\right)\Vert_{C^0} \\ &\leq 1 + |\varepsilon_1|(1+|\varepsilon_2| \|\bu'\|_{C^0}).
\end{align*}
This provides the expression for $Z_{1,1}$. For the computation of $Z_{1,2}$, we need to estimate the supremum of $\|\partial_t\left[ \partial_u g(\bu)  u\right]\Vert_{C^0}$ over $\|u\|_{C^1}=1$. Our estimation is a direct consequence of the following identity
\begin{align*}
    \partial_t\left[ \partial_u g(\bu)  u\right] = & -u' - \varepsilon_1(1-\varepsilon_2 \bar{u}') u'(t-1-\varepsilon_2\bu) \\ &\qquad \qquad + \varepsilon_1\varepsilon_2 (1-\varepsilon_2 \bu')\bar{u}''(t-1-\varepsilon_2\bu)u + \varepsilon_1\varepsilon_2\bar{u}'(t-1-\varepsilon_2\bu)u'. \qedhere
\end{align*}
\end{proof}

Thanks to the above lemmas, we were able to derive upper bounds $Z_1$ (cf. Lemmas \ref{lem : Z1 forcing} and \ref{lem : Z1_12 forcing}) and $Z_2$ (cf. Lemma \ref{lem : Z2 forcing}) for applying Theorem \ref{thm:radPolyBanach}. However, these bounds still involve operator norms of operators that are potentially infinite-dimensional, which is not yet fully tractable. In order to resolve this, we introduce the following result, allowing us to decompose each of these operator norms as the norm of a finite Galerkin projection plus an explicit bound on the remainder. The finite-dimensional part will then be treated in detail in Section \ref{sec:C0_operator_norms}.

\begin{lemma}\label{lem : from infinite to finite}
Let $\mathbb{B} : H \to H_0$  be a bounded linear operator and let $e  \bydef (0,1) \in H$, where $1$ has to be understood as the constant function $1$. Then the following estimates hold
\begin{equation}\label{eq : estimates operators norms}
    \begin{aligned}
        \|\pi^{\leq N} \mathbb{B} \pi^{\leq N} \mathbb{D}^\dagger\|_{B(H_0)} &\leq \| \pi^{\leq N} \mathbb{B} \pi^{\leq N} \mathbb{D}^\dagger \pi^{\leq N}\|_{B(H_0)} + \theta_N \| \pi^{\leq N} \mathbb{B} e\|_{H_0}\\
        \|\pi^{\leq N} \mathbb{B} \pi^{\leq N} \|_{B(H,H_0)} &\leq \| \pi^{\leq N} \mathbb{B} \pi^{\leq N} \mathbb{D}^\dagger\pi^{\leq N}  \|_{B(H_0)} + (1+\theta_N)\| \pi^{\leq N} \mathbb{B} e \|_{H_0}.
    \end{aligned}
\end{equation}
Moreover, letting $\mathbb{J}(\beta,u) \bydef (0,u)$ for all $(\beta,u) \in H_0$, we have
\begin{align*}
     \| \mathbb{A}  \mathbb{D}^\dagger\|_{B(H_0)} & \leq \|\mathbb{A}^{\leq N}\mathbb{D}^\dagger\pi^{\leq N}\|_{B(H_0)} + \theta_N \|\mathbb{A}^{\leq N}e\|_{H_0} + \theta_N  \\
        \| \mathbb{D}\mathbb{A}  \mathbb{D}^\dagger\|_{B(H_0)} & \leq \|\mathbb{D}\mathbb{A}^{\leq N}\mathbb{D}^\dagger\pi^{\leq N} - \pi^{\leq N}\mathbb{J}\|_{B(H_0)} + \theta_N \|\mathbb{D}\mathbb{A}^{\leq N}e\|_{H_0} + 1.
\end{align*}
\end{lemma}
\begin{proof}
Recalling \eqref{eq : L1 operator functions}, we have that $\mathbb{L}_1 u' = u - u(-\pi)$ for all $u\in C^1_{per}([-\pi,\pi])$. This implies that  
\begin{equation*}
    \pi^{\leq N} x = \mathbb{D}^\dagger \mathbb{D} \pi^{\leq N} x + (0, (\pi^{\leq N}u)(-\pi))
\end{equation*}
for all $x = (\beta, u) \in H$. Moreover, using \eqref{def : Kolmogorov constant}, notice that 
\begin{align*}
    (0, (\pi^{\leq N}u)(-\pi)) =  (\pi^{\leq N}u)(-\pi) e ~ \text{ and } ~  |(\pi^{\leq N}u)(-\pi)| \leq \|u\|_{C^0} + \|\pi^{>N}u\|_{C^0} \leq  (1+\theta_N) \|u\|_{C^1}.
\end{align*}
Using the above, we obtain that 
\begin{align*}
     \|\pi^{\leq N} \mathbb{B} \pi^{\leq N} \|_{B(H,H_0)} &\leq  \|\pi^{\leq N} \mathbb{B} \pi^{\leq N} \mathbb{D}^\dagger \mathbb{D}  \pi^{\leq N} \|_{B(H,H_0)} + (1+\theta_N) \|\pi^{\leq N} \mathbb{B} e \|_{H_0}\\
     &\leq \|\pi^{\leq N} \mathbb{B} \pi^{\leq N}\mathbb{D}^\dagger   \pi^{\leq N} \|_{B(H_0)} + (1+\theta_N) \|\pi^{\leq N} \mathbb{B} e \|_{H_0},
\end{align*}
since $\mathbb{D}\pi^{\leq N} x  = \pi^{\leq N}\mathbb{D} x$ for all $x \in H$ and  $\|\mathbb{D}\|_{B(H,H_0)} = 1$. This proves the second inequality of~\eqref{eq : estimates operators norms}. 
The first inequality of \eqref{eq : estimates operators norms} follows from $\pi^{\leq N}\mathbb{D}^\dagger = \pi^{\leq N}\mathbb{D}^\dagger\pi^{\leq N} + \pi^{\leq N}\mathbb{D}^\dagger\pi^{>N}$, since by Lemma \ref{lem:piNL1norm}, $\pi^{\leq N}\mathbb{D}^\dagger\pi^{>N}(\beta,v) = (0, \pi^{\leq N}\mathbb{L}_1\pi^{>N}v)$ is a multiple of $e$ with coefficient bounded by $\theta_N\|v\|_{C^0}$.

Next, recalling that $\mathbb{A} = \mathbb{A}^{\leq N} + \pi^{>N}$, we have
\begin{align}\label{eq : decomposition A and D}
    \mathbb{D}\mathbb{A}\mathbb{D}^\dagger &= \mathbb{D}\mathbb{A}^{\leq N}\mathbb{D}^\dagger\pi^{\leq N} + \mathbb{D}\mathbb{A}^{\leq N}\mathbb{D}^\dagger\pi^{> N} + \mathbb{J}\pi^{>N}.
\end{align}
Then, we write $\mathbb{J}\pi^{>N} = \mathbb{J} - \pi^{\leq N}\mathbb{J}$. Since, as above, $\pi^{\leq N}\mathbb{D}^\dagger\pi^{>N}$ maps $(\beta,v)$ to a multiple of $e$ with coefficient bounded by $\theta_N\|v\|_{C^0}$, and $\mathbb{A}^{\leq N} = \mathbb{A}^{\leq N}\pi^{\leq N}$, we obtain the estimate
\begin{align*}
\|\mathbb{D}\mathbb{A}\mathbb{D}^\dagger\|_{B(H_0)} &\leq \|\mathbb{D}\mathbb{A}^{\leq N}\mathbb{D}^\dagger\pi^{\leq N} - \pi^{\leq N}\mathbb{J}\|_{B(H_0)} + \theta_N \|\mathbb{D}\mathbb{A}^{\leq N}e\|_{H_0} + 1.
\end{align*}
Concerning $\|\mathbb{A} \mathbb{D}^\dagger\|_{B(H_0)}$, we use similar ideas and the decomposition
\begin{align*}
    \mathbb{A} \mathbb{D}^\dagger = \mathbb{A}^{\leq N} \mathbb{D}^\dagger \pi^{\leq N}  + \mathbb{A}^{\leq N} \mathbb{D}^\dagger \pi^{> N}  +  \pi^{>N} \mathbb{D}^\dagger,
\end{align*}
where the last term is bounded using~\eqref{eq:piNL1}. 
This concludes the proof.
\end{proof}

\begin{remark}
    In \eqref{eq : decomposition A and D}, one could estimate each term separately, but note that $\Vert \mathbb{J}\pi^{>N} \Vert_{B(H_0)} = \Vert \pi^{>N} \Vert_{B(C^0)}$ grows like the Lebesgue constant $\Lambda_N \bydef \Vert \pi^{\leq N} \Vert_{B(C^0)}$, hence like $\log(N)$ as $N \to \infty$ \cite{Fejer1910}. 
\end{remark}

The estimates~\eqref{eq : estimates operators norms} can be used to control the operator norms appearing in $Z_{1,0}$ (Lemma~\ref{lem : Z1 forcing}) and $Z_{1,1}$ (Lemma~\ref{lem : Z1_12 forcing}), whereas the second part of Lemma~\ref{lem : from infinite to finite} allows us to control the operator norms appearing in $Z_2$ (Lemma~\ref{lem : Z2 forcing}). Our only remaining task for making these estimates fully computable is to provide a way of explicitly bounding $H_0$ operator norms of finite-dimensional operators, ideally in a sharp way, which is accomplished in Section~\ref{sec:C0_operator_norms}. Before that, we first give an explicit description of the finite-dimensional operators involved.

\subsubsection{Explicit Galerkin truncation of the Jacobian in Fourier coefficients}

The bound $Z_{1,0}$ obtained in Lemma~\ref{lem : Z1 forcing} requires some extra care since we need to be able to exactly represent the Galerkin projection of the Jacobian $\pi^{\leq N}D\mathbb{F}(\bar{x})\pi^{\leq N}$. To this end, we use its Fourier series representation $\Pi^{\leq N} DF(\bar{X}) \Pi^{\leq N}.$ Recalling the decomposition of $D\mathbb{F}(\bar{x})$ in \eqref{eq : decomposition DF with DQ}, we need to compute the Fourier coefficients representation of the operator $\partial_u g(\bar u)$. We break down the latter into three pieces.

Indeed, we define the following linear operators on $C^1_{per}([-\pi,\pi])$ 
\begin{align*}
     (\mathbb{L}_2 u)(t) \bydef \bar{u}'(t-1 - \varepsilon_2\bar{u}(t))u(t) \qquad \text{and} \qquad (\mathbb{L}_3u)(t) \bydef u(t-1-\varepsilon_2\bar{u}(t)),
\end{align*}
so that 
\begin{align}
    \partial_u g(\bu)u = - u + \varepsilon_1\varepsilon_2 \mathbb{L}_2 u - \varepsilon_1 \mathbb{L}_3 u.
\end{align}
Our goal is then to derive the Fourier coefficients expression for the operators $\mathbb{L}_2$ and $\mathbb{L}_3.$

First, notice that $\mathbb{L}_2$ is simply a multiplication operator by the function $v(t) \bydef \bar{u}'(t-1-\varepsilon_2\bar{u}(t))$.
Letting $V$ be the Fourier coefficients representation of the function $v$,  we have 
\begin{align*}
    (L_2(U))_n = (V*U)_n  ~~ \text{ for all } n \in \mathbb{Z}.
\end{align*}
Consequently, $L_2$ is a Toeplitz matrix given by the coefficients $V$.
On the other hand, the expression for $L_3$ is given as follows 
\begin{align}\label{def : operator L3}
    ({L}_3)_{n,k} =  \frac{1}{2\pi } \int_{-\pi}^{\pi}  e^{ik(x-1 - \varepsilon_2 \bar{u}(x))} e^{-inx} dx = \mathcal{F}\left(e^{ik(x-1 - \varepsilon_2 \bar{u}(x))}\right)_n.
\end{align}
 Using the above expression, we need to perform $2N+1$ rigorous FFTs for computing the Fourier coefficients of index $|n| \leq N$ of $x \mapsto e^{ik(x-1 - \varepsilon_2 \bar{u}(x))}$ for all $k \in \{-N, \dots, N\}$. This is achieved using the analysis of Section~\ref{sec:fourier_coeff_delayed_term}.
\subsubsection{Computation of finite-dimensional operator norms} \label{sec:C0_operator_norms}

Using the previous section, we are able to derive a rigorous enclosure of the Galerkin projection matrix $\Pi^{\leq N} DF(\bar{X}) \Pi^{\leq N}$, and consequently of $\Pi^{\leq N} - A^{\leq N} \Pi^{\leq N} DF(\bar{X}) \Pi^{\leq N}$. However, in order to compute the bound $Z_{1,0}$ of Lemma~\ref{lem : Z1 forcing}, we need to compute operator norms on $H_0$ involving $\pi^{\leq N} - \mathbb{A}^{\leq N} \pi^{\leq N} D\mathbb{F}(\bar{x}) \pi^{\leq N}$, namely its composition with $\mathbb{D}^\dagger\pi^{\leq N}$ on the right, and with $\mathbb{D}$ on the left (cf. \eqref{eq : components Z1} and Lemma~\ref{lem : from infinite to finite}).
 Similarly, the bounds $Z_{1,1}$ and $Z_2$ of Lemmas~\ref{lem : Z1_12 forcing} and \ref{lem : Z2 forcing} involve the operator norms of $\mathbb{A}^{\leq N}\mathbb{D}^\dagger$, $\mathbb{D}\mathbb{A}^{\leq N}\mathbb{D}^\dagger$, $\mathbb{A}\mathbb{D}^\dagger$ and $\mathbb{D}\mathbb{A}\mathbb{D}^\dagger$, which are simplified thanks to Lemma \ref{lem : from infinite to finite} to operator norms of finite-dimensional operators again. 
Ignoring the $\R^m$ component of $H_0$, which is dealt with separately, we are given a linear operator $\mathbb{B} : C^0_{per}([-\pi,\pi]) \to C^0_{per}([-\pi,\pi])$ such that $\mathbb{B} = \pi^{\leq N}\mathbb{B}\pi^{\leq N}$. Our goal is to compute a sharp upper bound for $\|\mathbb{B}\|_{B(C^0)}$ using the Fourier coefficients representation $B$ of $\mathbb{B}$.

Since $\mathbb{B} = \pi^{\leq N}\mathbb{B}\pi^{\leq N}$, the coefficients $B_{nm}$ of $B$ vanish if $|n| > N$ or $|m| > N$, and $\mathbb{B}$ has the following kernel representation
\begin{equation}\label{eq : kernel formulation finite dim op}
    (\mathbb{B}u)(x) = \int_{-\pi}^{\pi} K(x,y) u(y) dy \quad  \text{ with } K(x,y) \bydef \frac{1}{2\pi}\sum_{n,m=-N}^N B_{nm} e^{i(nx-my)}.
\end{equation}
Since the kernel $K$ is continuous, we have that 
\begin{align}
\label{eq:opnorm_estimate}
    \|\mathbb{B}\Vert_{B(C^0)} = \max_{x \in [-\pi, \pi]} \int_{-\pi}^\pi |K(x,y)| dy.
\end{align}

Let $t_k \bydef \frac{2\pi k}{N_{\tt FFT}}$, $k \in \{0, \dots, N_{\tt FFT}-1\}$, be the grid points of Section~\ref{sec:tail_coefficients}. Note that an upper bound for \eqref{eq:opnorm_estimate} can be obtained as in Section~\ref{sec:tail_coefficients}, by enclosing $K$ on each product $[t_k,t_{k+1}] \times [t_j,t_{j+1}]$ using an interval-valued FFT combined with interval arithmetic (cf.~\cite{HaroSandinVidal2025,MR4930548}).  
However, the overestimation of such an enclosure is controlled by the absolute values of the coefficients $B_{nm}$ and does not benefit from cancellations in $K$, so the intervals must be small compared to $\frac{1}{N}$, which requires $N_{\tt FFT}$ much larger than $N$ and enclosing $K$ on $N_{\tt FFT}^2$ products of intervals. For the large matrices needed in Section~\ref{sec : CAP threshold}, this becomes too expensive.

Instead, we only use the values of $K$ on the grid. Since $K$ is a trigonometric polynomial of order $N$ in each variable, we need inequalities of Marcinkiewicz--Zygmund type (cf. \cite{MarcinkiewiczZygmund1937}), which bound the norms of a trigonometric polynomial $q$ of order $N$ by their discrete counterparts on the grid:
\begin{equation}\label{eq : MZ}
    \|q\|_{C^0} \leq C_\infty \max_{k \leq N_{\tt FFT}-1} |q(t_k)| \quad \text{ and } \quad \|q\|_{L^1} \leq C_1 \frac{2\pi}{N_{\tt FFT}}\sum_{k=0}^{N_{\tt FFT}-1} |q(t_k)|.
\end{equation}
Indeed, given $u$ with $\|u\|_{C^0} \leq 1$, applying the first inequality to the trigonometric polynomial $q = \mathbb{B}u$ of order $N$ in $x$, and the second one to the trigonometric polynomial $q = K(t_k,\cdot)$ of order $N$ in $y$, we get
\begin{equation}\label{eq : MZ bound}
    \|\mathbb{B}\|_{B(C^0)} \leq C_\infty C_1 \frac{2\pi}{N_{\tt FFT}} \max_{k \leq N_{\tt FFT}-1} \sum_{j=0}^{N_{\tt FFT}-1}|K(t_k,t_j)|.
\end{equation}

In the remainder of this discussion, $q$ denotes an arbitrary trigonometric polynomial of order $N$. A first choice of constants follows from Bernstein's inequalities $\|q'\|_{C^0} \leq N\|q\|_{C^0}$ and $\|q'\|_{L^1} \leq N\|q\|_{L^1}$ (cf. \cite{Zygmund2002}). Since every point of $[-\pi,\pi]$ is at distance at most $\frac{\pi}{N_{\tt FFT}}$ from the grid, the first one yields $\|q\|_{C^0} \leq \max_k |q(t_k)| + \frac{\pi N}{N_{\tt FFT}}\|q\|_{C^0}$, and the second one bounds the error of the Riemann sum of $|q|$ on the grid by $\frac{\pi N}{N_{\tt FFT}}\|q\|_{L^1}$. Hence \eqref{eq : MZ} holds with $C_\infty = C_1 = \left(1-\pi N/N_{\tt FFT}\right)^{-1}$ if $N_{\tt FFT} > \pi N$. These constants are close to $1$ only if $N_{\tt FFT}$ is much larger than $N$, and both can be improved.

For $C_\infty$, let $x^*$ be such that $|q(x^*)| = \|q\|_{C^0}$, let $\phi$ be the argument of $q(x^*)$ and let $p \bydef \mathrm{Re}(e^{-i\phi}q)$. Then, $p$ is a real trigonometric polynomial of order $N$ such that $p \leq |q|$, reaching its maximum $\|p\|_{C^0} = \|q\|_{C^0}$ at $x^*$. In particular $p'(x^*) = 0$, and applying Bernstein's inequality twice, that is $\|p''\|_{C^0} \leq N^2\|p\|_{C^0}$, yields $p(x) \geq \|p\|_{C^0}\left(1 - \frac{1}{2}N^2(x-x^*)^2\right)$. Since $x^*$ is at distance at most $\pi / N_{\tt FFT}$ from the grid and $p \leq |q|$, we get that the first inequality in \eqref{eq : MZ} holds with $C_\infty = \left(1-\pi^2N^2/(2N_{\tt FFT}^2)\right)^{-1}$, provided that $N_{\tt FFT} > \pi N/\sqrt{2}$.

The constant $C_\infty$ can be further improved slightly, at no additional cost, by replacing Bernstein's inequality with the Bernstein--Szeg\H{o} inequality $p'(x)^2 + N^2 p(x)^2 \leq N^2 \|p\|_{C^0}^2$, which holds for every real trigonometric polynomial $p$ of order $N$ (cf. \cite{Szego1928, BorweinErdelyi1995}). Writing $p = \|p\|_{C^0}\cos(\psi)$ with $\psi \bydef \arccos(p/\|p\|_{C^0}) \in [0,\pi]$, it gives $|\psi'| = |p'|/\sqrt{\|p\|_{C^0}^2 - p^2} \leq N$ wherever $|p| < \|p\|_{C^0}$, that is, except at finitely many points (unless $p$ is constant). Since $\psi$ is continuous and $\psi(x^*) = 0$, we get $\psi(x) \leq N|x-x^*|$, hence $p(x) \geq \|p\|_{C^0} \cos\left(N(x-x^*)\right)$ for all $|x-x^*| \leq \frac{\pi}{N}$, as $\cos$ is decreasing on $[0,\pi]$. Consequently, if $N_{\tt FFT} > 2N$, so that $\cos(\pi N/N_{\tt FFT}) > 0$, then the first inequality in \eqref{eq : MZ} holds with $C_\infty = 1/\cos(\pi N/N_{\tt FFT})$, a bound that goes back to~\cite{EhlichZeller1964}.

For $C_1$, we reconstruct $q$ from its values on the grid. If $V$ is a trigonometric polynomial such that $(\mathcal{F}(V))_n = 1$ for $|n| \leq N$ and $(\mathcal{F}(V))_n = 0$ for $|n| \geq N_{\tt FFT}-N$, then
\begin{equation*}
    q(y) = \frac{1}{N_{\tt FFT}}\sum_{j=0}^{N_{\tt FFT}-1} q(t_j) V(y-t_j) \quad \text{ for all } y \in [-\pi,\pi],
\end{equation*}
as one checks for $q(y) = e^{imy}$, $|m| \leq N$, using that $\sum_{j=0}^{N_{\tt FFT}-1} e^{int_j} = 0$ unless $n$ is a multiple of $N_{\tt FFT}$. Taking absolute values and integrating over $y$, we get that the second inequality in \eqref{eq : MZ} holds with $C_1 = \|V\|_{L^1}/(2\pi)$. To keep $\|V\|_{L^1}$ small, we let the Fourier coefficients of $V$ decrease linearly from $1$ to $0$ between $|n| = N$ and $|n| = N_{\tt FFT}-N$, which gives the de la Vall\'ee Poussin kernel (cf. \cite{Zygmund2002, DaiPrymakTemlyakovTikhonov2019}):
\begin{equation}\label{eq : de la Vallee Poussin}
    (\mathcal{F}(V))_n = \begin{cases}
    1, & |n| \leq N, \\
    \frac{N_{\tt FFT}-N-|n|}{N_{\tt FFT}-2N}, & N \leq |n| \leq N_{\tt FFT}-N, \\
    0, & |n| \geq N_{\tt FFT}-N.
    \end{cases}
\end{equation}
Combining the above, we obtain the following result.
\begin{lemma}\label{lem : sampling operator norm}
    Let $\mathbb{B}$ and $K$ be as in \eqref{eq : kernel formulation finite dim op}, let $N_{\tt FFT} > 2N$, let $t_k$, $k \in \{0, \dots, N_{\tt FFT} - 1\}$, be the grid points defined above and let $V$ be given by \eqref{eq : de la Vallee Poussin}. Then, we have that
    \begin{align}\label{eq : sampling operator norm}
        \|\mathbb{B}\Vert_{B(C^0)} \leq \frac{\|V\|_{L^1}}{N_{\tt FFT}\cos\left(\frac{\pi N}{N_{\tt FFT}}\right)} \max_{k \leq N_{\tt FFT}-1} \sum_{j = 0}^{N_{\tt FFT}-1} |K(t_k,t_j)|.
    \end{align}
\end{lemma}

\begin{proof}
    Since $N_{\tt FFT} > 2N$, the first inequality in \eqref{eq : MZ} holds with $C_\infty = 1/\cos(\pi N/N_{\tt FFT})$, and since $V$ given by \eqref{eq : de la Vallee Poussin} satisfies $(\mathcal{F}(V))_n = 1$ for $|n| \leq N$ and $(\mathcal{F}(V))_n = 0$ for $|n| \geq N_{\tt FFT}-N$, the second inequality in \eqref{eq : MZ} holds with $C_1 = \|V\|_{L^1}/(2\pi)$. Hence, \eqref{eq : sampling operator norm} follows from \eqref{eq : MZ bound}.
\end{proof}

In practice, the values $K(t_k,t_j)$ in \eqref{eq : sampling operator norm} are enclosed rigorously by applying the FFT, with interval arithmetic, to the rows and columns of $B$ (cf. \eqref{eq : kernel formulation finite dim op}). 
Moreover, the enclosure of $\|V\|_{L^1}$ is also obtained using the FFT on a fine grid. The right-hand side of \eqref{eq : sampling operator norm} is the factor $C_\infty C_1 = \frac{\|V\|_{L^1}}{2\pi\cos(\pi N/N_{\tt FFT})}$ times a discretization of \eqref{eq:opnorm_estimate}, which converges to $\|\mathbb{B}\|_{B(C^0)}$ as $N_{\tt FFT} \to \infty$. Hence, this factor measures the overestimation of $\|\mathbb{B}\|_{B(C^0)}$. It tends to $1$ as $\frac{N}{N_{\tt FFT}}$ tends to $0$. Indeed, comparing Fourier coefficients shows that $V$ is the following combination of Fej\'er kernels:
\begin{equation*}
    V = \frac{(N_{\tt FFT}-N)F_{N_{\tt FFT}-N} - NF_N}{N_{\tt FFT}-2N}, \quad \text{ where } F_m(t) \bydef \sum_{|n|<m}\left(1-\frac{|n|}{m}\right)e^{int}.
\end{equation*}
Since $F_m \geq 0$ and $\int_{-\pi}^{\pi}F_m(t)dt = 2\pi$, the triangle inequality gives $\frac{\|V\|_{L^1}}{2\pi} \leq \frac{(N_{\tt FFT}-N)+N}{N_{\tt FFT}-2N} = \frac{N_{\tt FFT}}{N_{\tt FFT}-2N}$. For the values of $N$ and $N_{\tt FFT}$ used in the applications in this paper, the factor $C_\infty C_1$ lies between $1.06$ and $1.46$.

\subsection{Computer-assisted proof of existence}\label{sec : CAP forcing}
We now combine the previous analysis in order to prove the existence of a periodic solution to \eqref{eq : toy model} in a vicinity of the approximate solution represented in Figure~\ref{fig : toy model}.

\begin{theorem}\label{th : CAP forcing}
    Let $\varepsilon_1 = 0.5$, $\varepsilon_2 = 2$ and let $\bx = (\bar{\beta},\bu)$ be the approximate solution represented in Figure~\ref{fig : toy model}. Then there exists an infinitely differentiable $2\pi$-periodic solution $\tilde{u}$ to \eqref{eq : toy model} such that
    \begin{align*}
        \|\tilde{u} - \bu\|_{C^1} \leq r_0 \bydef 2.06 \times 10^{-7}.
    \end{align*}
\end{theorem}

\begin{proof}
    In practice,  $\bu$ is chosen as a trigonometric polynomial of order $150$ and we choose the truncation size $N = 660$. Next, we construct $\mathbb{A} = \mathbb{A}^{\leq N} + \pi^{>N}$ as described in Section \ref{sec : approximate objects and projections}. Then, choosing $N_{\tt FFT} = 2^{12}$ and $\bar{\nu} = 1.1$, we compute the bounds $Y$, $Z_1$ and $Z_2$ of Lemmas \ref{lem : Y bound toy model}, \ref{lem : Z1 forcing} and \ref{lem : Z2 forcing}, respectively, using Lemmas \ref{lem : Z1_12 forcing}, \ref{lem : from infinite to finite} and \ref{lem : sampling operator norm}, Proposition \ref{prop : computation Kolmogorov constant} for $\theta_N$, and interval arithmetic \cite{Ru99a}. We obtain
    \begin{align*}
        Y \leq 1.21 \times 10^{-7}, \quad Z_1 \leq 0.410 \quad \text{ and } \quad Z_2 \leq 158.9.
    \end{align*}
    Since $\hat{Z}_2(r)$ in Lemma \ref{lem : Z2 forcing} does not depend on $r$, \eqref{eq:Z2_radPolyBanach} holds for any $r_* > 0$. Moreover, $Y + Z_1 r_0 + \frac{1}{2}Z_2 r_0^2 \leq r_0$ and $Z_1 + Z_2 r_0 < 1$. Since $Z_1<1$, $\mathbb{A}D\mathbb{F}(\bx)$ is invertible, so that $\mathbb{A}$ is surjective. Since $\pi^{\le N}\mathbb{A} = \mathbb{A}^{\le N}$, this implies that $\mathbb{A}^{\le N}$ is surjective, hence bijective, on the finite-dimensional space $H^{\le N}$, and therefore $\mathbb{A} = \mathbb{A}^{\le N} + \pi^{>N}$ is injective. Therefore, Theorem \ref{thm:radPolyBanach} provides a zero $\tilde{x} = (\tilde{\beta},\tilde{u})$ of $\mathbb{F}$ such that $\|\tilde{x} - \bx\|_H \leq r_0$. We conclude using Lemma \ref{lem : zeros of F are solutions} and $\|\tilde{u} - \bu\|_{C^1} \leq \|\tilde{x} - \bx\|_H$. We note that the truncation size $N$ is much larger than the order of $\bu$. This is because the bounds $Z_{1,1}$ and $Z_{1,2}$ are proportional to $\theta_N$, which only decays like $\frac{4}{\pi^2}\frac{\ln N}{N}$, and which is the optimal constant in Lemma~\ref{lem:piNC0C1}. The code performing the computer-assisted proofs is available at \cite{code_proof}.
\end{proof}

We point out that the Fourier coefficients of $\bu$ do not exhibit a geometric decay: after a fast initial decrease, their decay slows down markedly (see Figure~\ref{fig : toy model}, right). This suggests that the periodic solution $\tilde{u}$ is not analytic, in contrast with the threshold model of Section~\ref{sec:threshold_SDDE} (compare with Figure~\ref{fig : threshold}, right; see also \cite{MR3229655} for examples of nonanalytic solutions of delay-differential equations). Theorem~\ref{th : CAP forcing} nevertheless applies, since our approach is formulated in the space $C^1$ and does not require the analyticity of the solution, but only that of the approximate solution $\bu$, which is a trigonometric polynomial. In practice, however, as the order of $\bu$ increases, the constant $C$ in the tail estimate \eqref{eq : tail estimates} grows rapidly, which makes the proof increasingly difficult. This is a consequence of the apparent nonanalyticity of the solution: on one of the lines $\mathrm{Im}(z) = \pm\bar\rho$, the $n$-th mode of $\bu$ is amplified by a factor $\bar\nu^{|n|}$, which is no longer compensated by the decay of the coefficients. We come back to this issue in Section~\ref{sec:conclusion}.

\section{An autonomous SDDE with two state-dependent delays} \label{sec : autonomous SDDE} 

We now turn to an autonomous problem, namely the scalar SDDE with two state-dependent delays
\begin{equation}\label{eq : autonomous SDDE}
    u'(t) = - \gamma u(t) - \kappa_1 u (\alpha_1(t,u(t))) - \kappa_2 u(\alpha_2(t,u(t))), \text{ where } \alpha_i(t,u(t)) = t - a_i - c_i u(t)
\end{equation}
introduced in \cite{HumphriesDeMasi2012} and studied in detail in \cite{humphries_2017_resonance}. Here $\gamma>0$ is a decay rate, $\kappa_1,\kappa_2\ge 0$ are feedback strengths, and the delays $a_i + c_i u(t)$, with $a_i,c_i>0$, depend linearly on the state. For $\kappa_2=0$, \eqref{eq : autonomous SDDE} reduces to the single-delay equation introduced by Mallet-Paret and Nussbaum as a prototypical SDDE (see \cite{humphries_2017_resonance} and the references therein). Moreover, for $c_1=c_2=0$, \eqref{eq : autonomous SDDE} is linear, so that its nonlinear dynamics is entirely generated by the state-dependence of the delays. Following \cite{humphries_2017_resonance}, we fix
\begin{equation*}
    \gamma = 4.75, ~ a_1 = 1.3, ~ a_2 = 6, ~ c_1 = c_2 = 1, ~ \kappa_2 = 3,
\end{equation*}
and consider $\kappa_1$ as a bifurcation parameter. Since $\gamma>\kappa_2$, the delays cannot become advanced along solutions \cite{humphries_2017_resonance}. As $\kappa_1$ increases, the zero solution loses stability in a Hopf bifurcation at $\kappa_1\approx 3.206$, and the bifurcating periodic orbits remain stable until a torus bifurcation at $\kappa_1\approx 3.656$, beyond which invariant tori, resonance tongues and more intricate dynamics arise \cite{humphries_2017_resonance}. In Theorem~\ref{th : CAP autonomous}, we prove the existence of a periodic orbit on this branch, for $\kappa_1 = 3.4$.

Compared with Section~\ref{sec:forced_toy_model}, the main new difficulty is that periodic solutions of \eqref{eq : autonomous SDDE} do not have an a priori prescribed period. The period therefore becomes an unknown, which is solved for simultaneously with the periodic orbit, and a phase condition is added to remove the invariance under time translations.

Given $\omega > 0$, we look for $2\pi\omega$-periodic solutions to \eqref{eq : autonomous SDDE}. Using the change of variable $t \to \frac{1}{\omega} t$, we rescale \eqref{eq : autonomous SDDE} and redefine $\alpha_i$ as 
\begin{align}\label{eq : time rescaling}
    u'(t) = -  \omega \gamma u(t) - \omega \kappa_1 u (\alpha_1(\omega,t,u(t))) - \omega \kappa_2 u(\alpha_2(\omega,t,u(t))) \text{ where } \alpha_i(\omega, t,u(t)) = t - \frac{a_i}{\omega}  - \frac{c_i}{\omega} u(t).
\end{align}
Moreover, we use the rescaling $u \to \omega u$ to get
\begin{align} \label{eq : autonomous SDDE rescaled}
    u'(t) = -  \omega \gamma u(t) - \omega \kappa_1 u (\alpha_1(\omega,t,u(t))) - \omega \kappa_2 u(\alpha_2(\omega,t,u(t))), 
\end{align}
where we now redefine $\alpha_i$ as 
\[
\alpha_i(\omega, t,u(t)) = t - \frac{a_i}{\omega}  - c_i u(t),
\]
and where we note that this rescaling allows us to decouple the state-dependency and the period. Using notation similar to that of \eqref{eq : toy equation with g}, we introduce $\mathbb{G} : \R^3 \to \R$ as 
\begin{align}\label{def : G autonomous example}
    \mathbb{G}( x, y, z) = - \gamma z -  \kappa_1 x -  \kappa_2 y.
\end{align}
Then, we obtain that \eqref{eq : autonomous SDDE rescaled} is equivalent to
\begin{align}
    u'(t) = \omega \mathbb{G}(u (\alpha_1(\omega,t,u(t))),u (\alpha_2(\omega,t,u(t))),u(t)).
\end{align}
Moreover, in order to simplify notation, as in \eqref{eq : toy equation with g}, we denote by $g : (0,\infty) \times C^1_{per}([-\pi,\pi]) \to C^1_{per}([-\pi,\pi])$ the following map
\begin{align}\label{eq : g autonomous}
    g(\omega,u)(t) = \omega \mathbb{G}(u (\alpha_1(\omega,t,u(t))),u (\alpha_2(\omega,t,u(t))),u(t)).
\end{align}
Now, since we need to solve for the period, we  add a phase condition allowing us to fix the natural translation invariance of the periodic orbit. For this purpose, assuming that the periodic solution oscillates around a point $u_0$, we impose the following condition
\begin{align}
    (\pi^{\leq N}u)(0) = u_0.
\end{align}
Note that in practice, one can choose a different condition which might be better suited or better conditioned for the problem at hand.

Now, in order to follow the notations of \eqref{eq:the_map}, we let $\beta_1 = \omega$ and let $\beta_2 \in \R$ be the value of $u$ at $-\pi$,  as introduced in Lemma \ref{lem : zeros of F are solutions}. Then, in the rest of the section, we denote $x = (\beta, u) \in \R^2 \times C^1_{per}([-\pi,\pi])$ where $\beta = (\beta_1,\beta_2)$.
Let $\xi : (0,\infty) \times \R \times C^1_{per}([-\pi,\pi]) \to \R^2$ be defined as 
\begin{align*}
    \xi(x) = \begin{pmatrix}
        \xi_1(x)\\
        \xi_2(x) 
    \end{pmatrix}
    = \begin{pmatrix}
        (\pi^{\leq N}u)(0) - u_0 \\
        \mathbb{M} g(\beta_1,u) 
    \end{pmatrix},
\end{align*}
and let $f : (0,\infty) \times \R \times C^1_{per}([-\pi,\pi]) \to C^1_{per}([-\pi,\pi])$ be defined as 
\begin{align*}
   f(x) =  u - \beta_2 - \mathbb{L}_1g(\beta_1,u).
\end{align*}
We consider $H = \R^2 \times C^1_{per}([-\pi,\pi])$, $H_0 = \R^2 \times C^0_{per}([-\pi,\pi])$ and the following map $\mathbb{F}$, with values in $H$,
\begin{align}\label{def : zero finding F autonomous}
    \mathbb{F}(x) \bydef \begin{pmatrix}
        \xi(x)\\
        f(x) 
    \end{pmatrix} = \begin{pmatrix}
        (\pi^{\leq N}u)(0) - u_0\\
        \mathbb{M} g(\beta_1,u) \\
        u - \beta_2 - \mathbb{L}_1g(\beta_1,u)
    \end{pmatrix}
\end{align}
for all $x = (\beta_1, \beta_2, u) \in H$ with $\beta_1 > 0$. We prove in the following that zeros of $\mathbb{F}$ in $H$ provide smooth periodic solutions to \eqref{eq : autonomous SDDE}.
\begin{lemma}\label{lem : zeros of F autonomous}
    Let $x = (\beta_1, \beta_2, u) \in H$ be a zero of~$\mathbb{F}$, given in \eqref{def : zero finding F autonomous}. Then, $v : t \mapsto \beta_1 u\left(\frac{t}{\beta_1}\right)$ is an infinitely differentiable $2\pi\beta_1$-periodic solution to \eqref{eq : autonomous SDDE}. Moreover, $v\left(-\pi\beta_1\right) = \beta_1 \beta_2$.
\end{lemma}

\begin{proof}
    The proof follows from Lemma \ref{lem : zeros of F are solutions}, combined with the time and amplitude rescalings~\eqref{eq : time rescaling} and~\eqref{eq : autonomous SDDE rescaled}.
\end{proof}

\subsection{Computation of bounds}

This section heavily relies on the analysis derived in Section \ref{sec:bounds_toy_model}. We again assume that we have access to an approximate solution $\bar{x} \in H$ such that
\begin{align}\label{eq : approx solution sec 3}
    \bar{x} = (\bar{\beta}, \bar{u}) = (\bar{\beta}_1, \bar{\beta}_2, \bar{u}) \in H \text{ and } \bar{U} \bydef \mathcal{F}(\bar{u})  \text{ with } \bar{u} = \pi^{\leq N} \bar{u}.
\end{align}
We also denote $\bar{X} = (\bar{\beta}, \bar{U})$, the Fourier coefficients representation.
The approximate solution $\bar{u}$ obtained for $\kappa_1 = 3.4$ is represented in Figure~\ref{fig : autonomous}, together with its Fourier coefficients.
\begin{figure}[ht]
\centering
\includegraphics[width=1\textwidth]{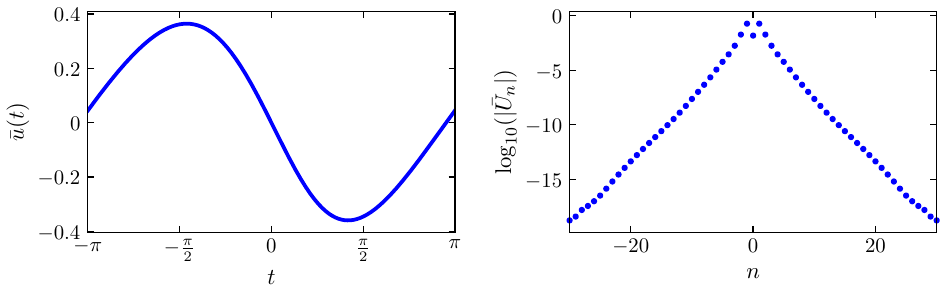}
\caption{Left: approximate solution $\bar{u}$ of \eqref{eq : autonomous SDDE rescaled} for $\gamma = 4.75$, $a_1 = 1.3$, $a_2 = 6$, $c_1 = c_2 = 1$, $\kappa_1 = 3.4$ and $\kappa_2 = 3$, given as a trigonometric polynomial of order $30$. In particular, $\bar{\beta}_1 \approx 0.4104$, so that the corresponding periodic solution of \eqref{eq : autonomous SDDE} has period $2\pi\bar{\beta}_1 \approx 2.578$. Right: $\log_{10}$ of the absolute value of its Fourier coefficients $\bar{U}_n$.}
\label{fig : autonomous}
\end{figure}
Moreover, as in Section \ref{sec : approximate objects and projections}, we construct an approximate inverse $\mathbb{A} = \mathbb{A}^{\leq N} + \pi^{>N}$ for the Jacobian $D\mathbb{F}(\bar{x})$ using its Fourier coefficients representation $A = A^{\leq N} + \Pi^{>N}.$

Using \eqref{eq : approx solution sec 3}, the computation of the bound $Y$ follows the same analysis as the one presented in Section \ref{sec:bounds_toy_model}. Specifically, $g(\bar{\beta}_1, \bar{u})$, defined in \eqref{eq : g autonomous}, possesses the structure described in Section \ref{sec:fourier_coeff_delayed_term}, hence we can rigorously enclose its Fourier coefficients.  The computation of the bound $Y$ then follows from Lemma \ref{lem : Y bound toy model}.

\subsubsection{Computation of \boldmath$Z_1$\unboldmath}

In Lemma \ref{lem : Z1 forcing}, we derived a general formula for the computation of the bound $Z_1.$ In particular, the aforementioned lemma requires the computation of the partial derivatives of $g$ evaluated at  $\bar{x}$. This is what we achieve in the following.

Writing $\alpha_j(\beta_1,u)(t) \bydef \alpha_j(\beta_1,t,u(t))$, let $\bar{v}_i$ ($i \in \{1,2,3,4\}$) be defined as 
\begin{align}
    \bar{v}_1 = \bar{u}(\alpha_1(\bar{\beta}_1,\bar{u})), ~ 
    \bar{v}_2 = \bar{u}(\alpha_2(\bar{\beta}_1,\bar{u})), ~
    \bar{v}_3 = \bar{u}'(\alpha_1(\bar{\beta}_1,\bar{u})) \text{ and }
    \bar{v}_4 = \bar{u}'(\alpha_2(\bar{\beta}_1,\bar{u})).
\end{align}

Using \eqref{eq : g autonomous}, we have
\begin{equation}\label{eq : partials of g autonomous}
    \begin{aligned}
   \partial_{\beta_1}g(\bar{x}) &= 
        \mathbb{G}(\bar{v}_1,\bar{v}_2, \bar{u}) -\frac{\kappa_1 a_1}{\bar{\beta}_1} \bar{v}_3    -\frac{\kappa_2 a_2}{\bar{\beta}_1} \bar{v}_4 \\
        \partial_{\beta_2}g(\bar{x}) &= 
        0\\
        \partial_{u} g(\bar{x}) h &= - \bar{\beta}_1 \gamma h -  \bar{\beta}_1 \kappa_1\left( h(\alpha_1(\bar{\beta}_1,\bar{u})) - c_1 \bar{v}_3 h \right) -\bar{\beta}_1 \kappa_2\left( h(\alpha_2(\bar{\beta}_1,\bar{u})) - c_2 \bar{v}_4 h \right)
\end{aligned}
\end{equation}
for all $h \in C^1_{per}([-\pi,\pi]).$ Now, let us look at the structure of the Fr\'echet derivative:
\begin{align}
    D\mathbb{F}(\bx) = \begin{pmatrix} 
0 & 0 & (\pi^{\leq N} \cdot )(0)\\
\mathbb{M}\partial_{\beta_1} {g}(\bx) & 0 & \mathbb{M}\partial_u {g}(\bx)
\vspace{.1cm} \\ 
\displaystyle 
-\mathbb{L}_1\partial_{\beta_1} {g}(\bx) & -1  & I -   \mathbb{L}_1\partial_u g(\bx) 
\end{pmatrix}.
\end{align}
In our analysis, we use the following decomposition
\begin{align}\label{eq : decomposition DF with DQ sec 3}
    D\mathbb{F}(\bx) = \mathbb{L} + \mathbb{D}^\dagger D\mathbb{Q}(\bx)
\end{align}
where 
\begin{align*}
\mathbb{Q}(x) = \begin{pmatrix}
-\beta_1 + (\pi^{\leq N}u)(0) - u_0\\
    -\beta_2 + \mathbb{M} g(x)\\
    - g(x)
\end{pmatrix}, ~~    D\mathbb{Q}(\bx) = \begin{pmatrix} 
-1 & 0 & (\pi^{\leq N} \cdot )(0)\\
\mathbb{M}\partial_{\beta_1} {g}(\bx) & -1 & \mathbb{M}\partial_u {g}(\bx)
\vspace{.1cm} \\ 
\displaystyle 
-\partial_{\beta_1} {g}(\bx) & 0 &  -   \partial_u g(\bx) 
\end{pmatrix} \text{ and } \mathbb{L} = \begin{pmatrix} 
1 & 0 & 0 \\ 
0 & 1 & 0\\
\displaystyle 
0 & -1 &  I
\end{pmatrix}.
\end{align*}
The above explicit formulas allow us to compute the finite-dimensional bounds involved in $Z_{1,0}$ (cf. Lemmas \ref{lem : Z1 forcing} and \ref{lem : from infinite to finite}). In contrast with Section \ref{sec:forced_toy_model}, the map $g$ now also depends on $\beta_1$, so that the third component of $D\mathbb{Q}(\bar{x})x$ reads $-\partial_{\beta_1}g(\bar{x})h_{0,1} - \partial_u g(\bar{x})h$ for $x = (h_{0,1},h_{0,2},h) \in H$. Since $\partial_{\beta_1} g(\bar{x})$ is a fixed function, its contribution is estimated directly using $\|\mathcal{F}^{-1}(U)\|_{C^0} \leq \|U\|_{\ell^1}$, rather than through \eqref{def : Kolmogorov constant}. Repeating verbatim the proof of Lemma \ref{lem : Z1 forcing}, the conclusion of that lemma therefore holds with $Z_{1,2}$ satisfying
\begin{equation}\label{eq : Z12 autonomous}
    Z_{1,2} \geq \left\|\Pi^{>N} \mathcal{F}\left(\partial_{\beta_1} g(\bar{x})\right)\right\|_{\ell^1} + \sup_{\|h\|_{C^1}=1} \|\partial_t[\partial_u g(\bar{x})h]\Vert_{C^0} \, \theta_N,
\end{equation}
the bound $Z_{1,1}$ being unchanged since $\pi^{>N}x$ has a vanishing $\R^2$ component and the phase condition vanishes on $\pi^{>N}x$. The Fourier coefficients of $\partial_{\beta_1} g(\bar{x})$ are enclosed using the analysis of Section \ref{sec:fourier_coeff_delayed_term}, since $\bar{v}_1, \dots, \bar{v}_4$ are all of the form \eqref{eq:composition_term}. Consequently, we require explicit upper bounds for $\|\partial_u g(\bar{x})\Vert_{B(C^0)}$ and $\sup_{\|h\|_{C^1}=1}  \|\partial_t[\partial_u g(\bar{x})h]\Vert_{C^0}$. This is achieved in the next lemma.
\begin{lemma}\label{lem : Z1 autonomous}
Denoting $\bar{v}_{2+j} = \bar{u}'(\alpha_j(\bar{\beta}_1,\bar{u}))$ for $j \in \{1,2\}$, we have the following inequalities
\begin{align*}
     \|\partial_u g(\bar{x})\Vert_{B(C^0)} &\leq |\bar{\beta}_1\gamma| + \sum_{j=1}^2 |\bar{\beta}_1 \kappa_j|\left(1 + |c_j|\|\bar{v}_{2+j}\|_{C^0}\right),\\
      \sup_{\|h\|_{C^1}=1}  \|\partial_t[\partial_u g(\bar{x})h]\Vert_{C^0} &\leq |\bar{\beta}_1\gamma| + \sum_{j=1}^2 |\bar{\beta}_1 \kappa_j|\left(\|1 - c_j\bar{u}'\|_{C^0} + |c_j|\left(\|\bar{v}_{2+j}\|_{C^0} + \|\bar{v}_{2+j}'\|_{C^0}\right)\right).
    \end{align*}
\end{lemma}

\begin{proof}
    Throughout the proof we use that $\partial_t\alpha_j(\bar{\beta}_1,\bar{u}) = 1 - c_j\bar{u}'$, so that
    \begin{equation}\label{eq : derivatives of bar vi}
        \bar{v}_{2+j}' = (1-c_j\bar{u}')\,\bar{u}''(\alpha_j(\bar{\beta}_1,\bar{u})), \qquad j \in \{1,2\}.
    \end{equation}
    The first inequality is a direct consequence of \eqref{eq : partials of g autonomous}, using that $\|h(\alpha_j(\bar{\beta}_1,\bar{u}))\|_{C^0} \leq \|h\|_{C^0}$ for all $h \in C^0_{per}([-\pi,\pi])$. Next, differentiating $\partial_u g(\bar{x})h$ given in \eqref{eq : partials of g autonomous} with respect to $t$, we obtain
    \begin{align*}
        \partial_t[\partial_u g(\bar{x})h] = -\bar{\beta}_1\gamma h' - \sum_{j=1}^2 \bar{\beta}_1\kappa_j \left( (1-c_j\bar{u}')\,h'(\alpha_j(\bar{\beta}_1,\bar{u})) - c_j\left(\bar{v}_{2+j}'h + \bar{v}_{2+j}h'\right) \right),
    \end{align*}
    and the second inequality follows using $\|h\|_{C^0} \leq 1$, $\|h'\|_{C^0} \leq 1$ and \eqref{eq : derivatives of bar vi}.
\end{proof}
Using the above lemma combined with Lemma \ref{lem : Z1 forcing} and \eqref{eq : Z12 autonomous}, we obtain explicit formulas for the computation of the bound $Z_1$.

\subsubsection{Computation of \boldmath$Z_2$\unboldmath} Now we turn to the computation of the bound $Z_2$. We follow a strategy similar to the one developed in Lemma \ref{lem : Z2 forcing}, and provide a complete proof in the following $Z_2$ result. Note that, in contrast with Section \ref{ssec : Z2 forcing}, the derivative of $\mathbb{F}$ now also varies in the direction $\beta_1$, which is why two separate contributions appear in \eqref{def : Z2 hat autonomous}.
\begin{lemma}\label{lem : Z2 autonomous}
Let $\bar{x} = (\bar{\beta},\bar{u})$ be as in \eqref{eq : approx solution sec 3} and let $0 < r < |\bar{\beta}_1|$. For $j \in \{1,2\}$, let
\begin{equation}\label{eq : Z2 autonomous constants}
    \delta_j \bydef \frac{|a_j|}{|\bar{\beta}_1|\left(|\bar{\beta}_1| - r\right)} \qquad \text{ and } \qquad \ell_j \bydef \delta_j + |c_j|,
\end{equation}
and let $\hat{\mathcal{Z}}_2(r)$ be a bound satisfying $\hat{\mathcal{Z}}_2(r) \geq \hat{\mathcal{Z}}_2^{\beta}(r) + \hat{\mathcal{Z}}_2^{u}(r)$, where
\begin{equation}\label{def : Z2 hat autonomous}
\begin{aligned}
    \hat{\mathcal{Z}}_2^{\beta}(r) &\bydef |\gamma| + \sum_{j=1}^{2} |\kappa_j| \Big( 1 + \ell_j \|\bar{u}'\Vert_{C^0} + \delta_j \left(\|\bar{u}'\Vert_{C^0} + r\right) + |\bar{\beta}_1| \delta_j \left(1 + \ell_j \|\bar{u}''\Vert_{C^0}\right) \Big),\\
    \hat{\mathcal{Z}}_2^{u}(r) &\bydef |\gamma| + \sum_{j=1}^{2} |\kappa_j| \Big( 1 + |\bar{\beta}_1| \ell_j + |c_j| \big( \|\bar{u}'\Vert_{C^0} + r + |\bar{\beta}_1| \left(1 + \ell_j \|\bar{u}''\Vert_{C^0}\right) \big) \Big).
\end{aligned}
\end{equation}
Then, defining
\[
Z_2(r) \bydef \max\{\|\mathbb{A}\mathbb{D}^\dagger\|_{B(H_0)}, ~ \|\mathbb{D}\mathbb{A}\mathbb{D}^\dagger\|_{B(H_0)}\} \hat{\mathcal{Z}}_2(r),
\]
we have that $\|\mathbb{A}(D\mathbb{F}(\bar{x} + x) - D\mathbb{F}(\bar{x}))\|_{B(H)} \leq Z_2(r)r$ for all $x \in B_r(0) \subset H$.
\end{lemma}

\begin{proof}
In order to simplify the presentation of our analysis, let us introduce the maps $g_1, g_2$ and $g_3$ defined as follows
\begin{equation}\label{def : decomposition of g with gi}
\begin{aligned}
    g_1(\beta_1,u) &\bydef -\gamma \beta_1 u,\\
    g_{1+j}(\beta_1,u) &\bydef -\kappa_j \beta_1 u(\alpha_j(\beta_1,u)), \quad j \in \{1,2\}.
\end{aligned}
\end{equation}
Recalling \eqref{def : G autonomous example} and \eqref{eq : g autonomous}, we have that $g(\beta_1,u) = g_1(\beta_1,u) + g_2(\beta_1,u) + g_3(\beta_1,u)$. Moreover, since
\[
    \partial_{\beta_1}\alpha_j(\beta_1,u) = \frac{a_j}{\beta_1^2} \qquad \text{ and } \qquad \partial_{u}\alpha_j(\beta_1,u)h = - c_j h,
\]
the chain rule yields, for all $h \in C^1_{per}([-\pi,\pi])$ and all $j \in \{1,2\}$,
\begin{equation}\label{eq : partials of gi autonomous}
\begin{aligned}
    \partial_{\beta_1} g_1(\beta_1,u) &= -\gamma u, \qquad \qquad \partial_{u} g_1(\beta_1,u)h = -\gamma \beta_1 h,\\
    \partial_{\beta_1} g_{1+j}(\beta_1,u) &= -\kappa_j u(\alpha_j(\beta_1,u)) - \frac{\kappa_j a_j}{\beta_1} u'(\alpha_j(\beta_1,u)),\\
    \partial_{u} g_{1+j}(\beta_1,u)h &= -\kappa_j \beta_1 \Big( h(\alpha_j(\beta_1,u)) - c_j u'(\alpha_j(\beta_1,u)) h \Big),
\end{aligned}
\end{equation}
which is consistent with \eqref{eq : partials of g autonomous} when evaluated at $\bar{x}$.

Now, let $x = (b_1,b_2,u) \in B_r(0) \subset H$, so that $|b_1| \leq r$, $\|u\|_{C^0} \leq r$ and $\|u'\|_{C^0} \leq r$. Throughout the proof we denote
\[
    \tilde{\beta}_1 \bydef \bar{\beta}_1 + b_1 \qquad \text{ and } \qquad v \bydef \bar{u} + u,
\]
and we notice that $|\tilde{\beta}_1| \geq |\bar{\beta}_1| - r > 0$ and $\|v'\|_{C^0} \leq \|\bar{u}'\|_{C^0} + r$. Finally, we denote
\begin{align*}
    \Delta_\beta &\bydef \partial_{\beta_1} g(\tilde{\beta}_1, v) - \partial_{\beta_1} g(\bar{\beta}_1, \bar{u}) \in C^0_{per}([-\pi,\pi]),\\
    \Delta_u &\bydef \partial_{u} g(\tilde{\beta}_1, v) - \partial_{u} g(\bar{\beta}_1, \bar{u}) \in B\big(C^1_{per}([-\pi,\pi]), C^0_{per}([-\pi,\pi])\big).
\end{align*}

\medskip
\noindent \textbf{Step 1: reduction to $\Delta_\beta$ and $\Delta_u$.} The operator $\mathbb{L}$ of \eqref{eq : decomposition DF with DQ sec 3} does not depend on the point at which the derivative is evaluated, and $\partial_{\beta_2}g = 0$ by \eqref{eq : partials of g autonomous}. Consequently,
\begin{equation}\label{eq : Z2 autonomous difference of DF}
    D\mathbb{F}(\bar{x}+x) - D\mathbb{F}(\bar{x}) = \mathbb{D}^\dagger\left(D\mathbb{Q}(\bar{x}+x) - D\mathbb{Q}(\bar{x})\right)
    = \mathbb{D}^\dagger \begin{pmatrix}
        0 & 0 & 0\\
        \mathbb{M}\Delta_\beta & 0 & \mathbb{M}\Delta_u\\
        -\Delta_\beta & 0 & -\Delta_u
    \end{pmatrix}.
\end{equation}
Let $z = (z_1,z_2,h) \in H$ with $\|z\|_H \leq 1$, so that $|z_1| \leq 1$ and $\|h\|_{C^1} \leq 1$. Since $|\mathbb{M}w| \leq \|w\|_{C^0}$ for all $w \in C^0_{per}([-\pi,\pi])$, we obtain from \eqref{eq : Z2 autonomous difference of DF} that
\[
    \left\|\left(D\mathbb{Q}(\bar{x}+x) - D\mathbb{Q}(\bar{x})\right)z\right\|_{H_0} = \left\|\Delta_\beta z_1 + \Delta_u h\right\|_{C^0} \leq \|\Delta_\beta\|_{C^0} + \|\Delta_u\|_{B(C^1,C^0)},
\]
and therefore
\begin{equation}\label{eq : Z2 autonomous norm DQ}
    \left\|D\mathbb{Q}(\bar{x}+x) - D\mathbb{Q}(\bar{x})\right\|_{B(H,H_0)} \leq \|\Delta_\beta\|_{C^0} + \|\Delta_u\|_{B(C^1,C^0)}.
\end{equation}
Denoting $\mathbb{K} \bydef D\mathbb{Q}(\bar{x}+x) - D\mathbb{Q}(\bar{x})$ and using \eqref{eq : identity norm H with H0 H1} followed by \eqref{eq : identity norm H1 and D}, we get for all $z \in H$ that
\begin{align*}
    \|\mathbb{A}(D\mathbb{F}(\bar{x}+x) - D\mathbb{F}(\bar{x}))z\|_{H} &= \max\left\{\|\mathbb{A}\mathbb{D}^\dagger \mathbb{K}z\|_{H_0}, ~ \|\mathbb{D}\mathbb{A}\mathbb{D}^\dagger \mathbb{K}z\|_{H_0}\right\}\\
    &\leq \max\{\|\mathbb{A}\mathbb{D}^\dagger\|_{B(H_0)}, ~ \|\mathbb{D}\mathbb{A}\mathbb{D}^\dagger\|_{B(H_0)}\} \|\mathbb{K}\|_{B(H,H_0)} \|z\|_H.
\end{align*}
Combining the above with \eqref{eq : Z2 autonomous norm DQ}, the proof reduces to establishing that
\begin{equation}\label{eq : Z2 autonomous goal}
    \|\Delta_\beta\|_{C^0} \leq \hat{\mathcal{Z}}_2^{\beta}(r) r \qquad \text{ and } \qquad \|\Delta_u\|_{B(C^1,C^0)} \leq \hat{\mathcal{Z}}_2^{u}(r) r.
\end{equation}

\medskip
\noindent \textbf{Step 2: increment of the delays.} Recalling that $\alpha_j(\beta_1,u)(t) = t - \frac{a_j}{\beta_1} - c_j u(t)$, we have for $j \in \{1,2\}$ that
\begin{equation}\label{eq : Z2 autonomous delay increment}
    \left\|\alpha_j(\tilde{\beta}_1,v) - \alpha_j(\bar{\beta}_1,\bar{u})\right\|_{C^0} \leq |a_j| \left|\frac{1}{\tilde{\beta}_1} - \frac{1}{\bar{\beta}_1}\right| + |c_j| \|u\|_{C^0} \leq \left(\frac{|a_j|}{|\bar{\beta}_1|(|\bar{\beta}_1|-r)} + |c_j|\right) r = \ell_j r,
\end{equation}
where we used that $|\tilde{\beta}_1 - \bar{\beta}_1| = |b_1| \leq r$, that $|\tilde{\beta}_1| \geq |\bar{\beta}_1| - r$ and that $\|u\|_{C^0} \leq r$. Note that $\bar{u} = \pi^{\leq N}\bar{u}$ is a trigonometric polynomial by \eqref{eq : approx solution sec 3}, hence $\bar{u} \in C^\infty_{per}([-\pi,\pi])$ and the quantities $\|\bar{u}'\|_{C^0}$ and $\|\bar{u}''\|_{C^0}$ are computable. In particular, the mean value inequality combined with \eqref{eq : Z2 autonomous delay increment} yields
\begin{equation}\label{eq : Z2 autonomous mean value}
    \left\|\bar{u}(\alpha_j(\tilde{\beta}_1,v)) - \bar{u}(\alpha_j(\bar{\beta}_1,\bar{u}))\right\|_{C^0} \leq \|\bar{u}'\|_{C^0} \ell_j r
    \quad \text{ and } \quad
    \left\|\bar{u}'(\alpha_j(\tilde{\beta}_1,v)) - \bar{u}'(\alpha_j(\bar{\beta}_1,\bar{u}))\right\|_{C^0} \leq \|\bar{u}''\|_{C^0} \ell_j r.
\end{equation}
Let us emphasize that in \eqref{eq : Z2 autonomous mean value} the increments in $\beta_1$ and in $u$ are performed simultaneously. Splitting them into two successive steps would require comparing $v' = \bar{u}' + u'$ at two distinct delayed arguments, hence would require a control on $\|u''\|_{C^0}$, which is not available since $x \in H$ only.

\medskip
\noindent \textbf{Step 3: estimate of $\Delta_\beta$.} Using \eqref{def : decomposition of g with gi} and \eqref{eq : partials of gi autonomous}, we have
\begin{equation}\label{eq : Z2 autonomous step 0}
    \|\Delta_\beta\|_{C^0} \leq \sum_{i=1}^{3} \left\|\partial_{\beta_1} g_i(\tilde{\beta}_1,v) - \partial_{\beta_1} g_i(\bar{\beta}_1,\bar{u})\right\|_{C^0}.
\end{equation}
The term $i=1$ is immediate, as
\[
    \left\|\partial_{\beta_1} g_1(\tilde{\beta}_1,v) - \partial_{\beta_1} g_1(\bar{\beta}_1,\bar{u})\right\|_{C^0} = |\gamma| \|u\|_{C^0} \leq |\gamma| r.
\]
Now we turn our attention to $g_2$ and $g_3$, which we treat simultaneously by letting $j \in \{1,2\}$. Using \eqref{eq : partials of gi autonomous}, we obtain
\begin{equation}\label{eq : Z2 autonomous step 1}
\begin{aligned}
    \left\|\partial_{\beta_1} g_{1+j}(\tilde{\beta}_1,v) - \partial_{\beta_1} g_{1+j}(\bar{\beta}_1,\bar{u})\right\|_{C^0}
    \leq\ & |\kappa_j| \left\| v(\alpha_j(\tilde{\beta}_1,v)) - \bar{u}(\alpha_j(\bar{\beta}_1,\bar{u})) \right\|_{C^0}\\
    & + |\kappa_j a_j| \left\| \frac{v'(\alpha_j(\tilde{\beta}_1,v))}{\tilde{\beta}_1} - \frac{\bar{u}'(\alpha_j(\bar{\beta}_1,\bar{u}))}{\bar{\beta}_1} \right\|_{C^0}.
\end{aligned}
\end{equation}
We estimate separately the two terms of the right-hand side of \eqref{eq : Z2 autonomous step 1}. Using that $\|u\|_{C^0} \leq r$ and \eqref{eq : Z2 autonomous mean value}, the first one satisfies
\begin{equation}\label{eq : Z2 autonomous step 2}
    \left\| v(\alpha_j(\tilde{\beta}_1,v)) - \bar{u}(\alpha_j(\bar{\beta}_1,\bar{u})) \right\|_{C^0} \leq \|u\|_{C^0} + \left\| \bar{u}(\alpha_j(\tilde{\beta}_1,v)) - \bar{u}(\alpha_j(\bar{\beta}_1,\bar{u})) \right\|_{C^0} \leq \left(1 + \ell_j \|\bar{u}'\|_{C^0} \right) r.
\end{equation}
As for the second one, using that $\|u'\|_{C^0} \leq r$, $\|v'\|_{C^0} \leq \|\bar{u}'\|_{C^0} + r$, $|\tilde{\beta}_1| \geq |\bar{\beta}_1| - r$ and \eqref{eq : Z2 autonomous mean value}, we get
\begin{equation}\label{eq : Z2 autonomous step 3}
\begin{aligned}
    \left\| \frac{v'(\alpha_j(\tilde{\beta}_1,v))}{\tilde{\beta}_1} - \frac{\bar{u}'(\alpha_j(\bar{\beta}_1,\bar{u}))}{\bar{\beta}_1} \right\|_{C^0}
    &\leq \left|\frac{1}{\tilde{\beta}_1} - \frac{1}{\bar{\beta}_1}\right| \|v'\|_{C^0} + \frac{1}{|\bar{\beta}_1|} \left\| v'(\alpha_j(\tilde{\beta}_1,v)) - \bar{u}'(\alpha_j(\bar{\beta}_1,\bar{u})) \right\|_{C^0}\\
    &\leq \frac{\|\bar{u}'\|_{C^0} + r}{|\bar{\beta}_1|(|\bar{\beta}_1|-r)} r + \frac{1 + \ell_j\|\bar{u}''\|_{C^0}}{|\bar{\beta}_1|-r} r.
\end{aligned}
\end{equation}
Since $\frac{|a_j|}{|\bar{\beta}_1|(|\bar{\beta}_1|-r)} = \delta_j$ and $\frac{|a_j|}{|\bar{\beta}_1|-r} = |\bar{\beta}_1|\delta_j$ by \eqref{eq : Z2 autonomous constants}, combining \eqref{eq : Z2 autonomous step 1}, \eqref{eq : Z2 autonomous step 2} and \eqref{eq : Z2 autonomous step 3} yields
\[
    \left\|\partial_{\beta_1} g_{1+j}(\tilde{\beta}_1,v) - \partial_{\beta_1} g_{1+j}(\bar{\beta}_1,\bar{u})\right\|_{C^0} \leq |\kappa_j| \Big( 1 + \ell_j\|\bar{u}'\|_{C^0} + \delta_j\left(\|\bar{u}'\|_{C^0} + r\right) + |\bar{\beta}_1|\delta_j\left(1 + \ell_j\|\bar{u}''\|_{C^0}\right) \Big) r.
\]
Summing the above over $j \in \{1,2\}$, adding the contribution of $g_1$ and recalling \eqref{def : Z2 hat autonomous} and \eqref{eq : Z2 autonomous step 0}, we have proven the first inequality of \eqref{eq : Z2 autonomous goal}.

\medskip
\noindent \textbf{Step 4: estimate of $\Delta_u$.} Let $h \in C^1_{per}([-\pi,\pi])$ be such that $\|h\|_{C^1} = \max\{\|h\|_{C^0}, \|h'\|_{C^0}\} \leq 1$. Using \eqref{def : decomposition of g with gi} and \eqref{eq : partials of gi autonomous} again, we have
\begin{equation}\label{eq : Z2 autonomous step 4}
    \|\Delta_u h\|_{C^0} \leq \sum_{i=1}^{3} \left\|\partial_{u} g_i(\tilde{\beta}_1,v)h - \partial_{u} g_i(\bar{\beta}_1,\bar{u})h\right\|_{C^0},
\end{equation}
and the term $i=1$ satisfies
\[
    \left\|\partial_{u} g_1(\tilde{\beta}_1,v)h - \partial_{u} g_1(\bar{\beta}_1,\bar{u})h\right\|_{C^0} = |\gamma| |b_1| \|h\|_{C^0} \leq |\gamma| r.
\]
Again, we treat $g_2$ and $g_3$ simultaneously by letting $j \in \{1,2\}$. From \eqref{eq : partials of gi autonomous} we get
\begin{equation}\label{eq : Z2 autonomous step 5}
\begin{aligned}
    \left\|\partial_{u} g_{1+j}(\tilde{\beta}_1,v)h - \partial_{u} g_{1+j}(\bar{\beta}_1,\bar{u})h\right\|_{C^0}
    \leq\ & |\kappa_j| \left\| \tilde{\beta}_1 h(\alpha_j(\tilde{\beta}_1,v)) - \bar{\beta}_1 h(\alpha_j(\bar{\beta}_1,\bar{u})) \right\|_{C^0}\\
    & + |\kappa_j c_j| \left\| \tilde{\beta}_1 v'(\alpha_j(\tilde{\beta}_1,v)) - \bar{\beta}_1 \bar{u}'(\alpha_j(\bar{\beta}_1,\bar{u})) \right\|_{C^0}.
\end{aligned}
\end{equation}
Writing $\tilde{\beta}_1 = \bar{\beta}_1 + b_1$, the mean value inequality combined with \eqref{eq : Z2 autonomous delay increment} provides
\[
    \left\| \tilde{\beta}_1 h(\alpha_j(\tilde{\beta}_1,v)) - \bar{\beta}_1 h(\alpha_j(\bar{\beta}_1,\bar{u})) \right\|_{C^0}
    \leq |b_1| \|h\|_{C^0} + |\bar{\beta}_1| \|h'\|_{C^0} \left\|\alpha_j(\tilde{\beta}_1,v) - \alpha_j(\bar{\beta}_1,\bar{u})\right\|_{C^0}
    \leq \left(1 + |\bar{\beta}_1| \ell_j \right) r,
\]
whereas \eqref{eq : Z2 autonomous mean value} and $\|u'\|_{C^0} \leq r$ provide
\begin{align*}
    \left\| \tilde{\beta}_1 v'(\alpha_j(\tilde{\beta}_1,v)) - \bar{\beta}_1 \bar{u}'(\alpha_j(\bar{\beta}_1,\bar{u})) \right\|_{C^0}
    &\leq |b_1| \|v'\|_{C^0} + |\bar{\beta}_1| \left( \|u'\|_{C^0} + \|\bar{u}''\|_{C^0} \ell_j r \right)\\
    &\leq \Big( \|\bar{u}'\|_{C^0} + r + |\bar{\beta}_1| \left(1 + \ell_j \|\bar{u}''\|_{C^0}\right) \Big) r.
\end{align*}
Plugging the two previous inequalities in \eqref{eq : Z2 autonomous step 5}, we obtain
\[
    \left\|\partial_{u} g_{1+j}(\tilde{\beta}_1,v)h - \partial_{u} g_{1+j}(\bar{\beta}_1,\bar{u})h\right\|_{C^0} \leq |\kappa_j| \Big( 1 + |\bar{\beta}_1| \ell_j + |c_j| \big( \|\bar{u}'\|_{C^0} + r + |\bar{\beta}_1| \left(1 + \ell_j \|\bar{u}''\|_{C^0}\right) \big) \Big) r.
\]
Summing over $j \in \{1,2\}$, adding the contribution of $g_1$, recalling \eqref{def : Z2 hat autonomous} and \eqref{eq : Z2 autonomous step 4} and taking the supremum over $\|h\|_{C^1} \leq 1$, we have proven the second inequality of \eqref{eq : Z2 autonomous goal}. This concludes the proof.
\end{proof}

\subsection{Computer-assisted proof of existence}\label{sec : CAP autonomous}

\begin{theorem}\label{th : CAP autonomous}
    Let $\gamma = 4.75$, $a_1 = 1.3$, $a_2 = 6$, $c_1 = c_2 = 1$, $\kappa_1 = 3.4$ and $\kappa_2 = 3$, and let $\bx = (\bar{\beta}_1,\bar{\beta}_2,\bu)$ be the approximate solution represented in Figure~\ref{fig : autonomous}. Then, there exist $\tilde{\beta}_1 > 0$ and $\tilde{u} \in C^1_{per}([-\pi,\pi])$ such that
    \begin{align*}
        |\tilde{\beta}_1 - \bar{\beta}_1| \leq r_0 \quad \text{ and } \quad \|\tilde{u} - \bu\|_{C^1} \leq r_0, \quad \text{ where } r_0 \bydef 9.7 \times 10^{-11},
    \end{align*}
    and $v : t \mapsto \tilde{\beta}_1\tilde{u}\left(\frac{t}{\tilde{\beta}_1}\right)$ is an infinitely differentiable $2\pi\tilde{\beta}_1$-periodic solution to \eqref{eq : autonomous SDDE}.
\end{theorem}

\begin{proof}
     $\bu$ is chosen as a trigonometric polynomial of order $30$, we take $u_0 = 0$ in the phase condition and we consider the truncation order $N = 800$. Next,  we construct $\mathbb{A}$ as in the proof of Theorem \ref{th : CAP forcing}. Then, choosing $N_{\tt FFT} = 2^{12}$, $\bar{\nu} = 1.1$ and $r_* = 10^{-3}$, we compute the bounds of Lemmas \ref{lem : Y bound toy model}, \ref{lem : Z1 forcing} (using Lemma \ref{lem : Z1 autonomous} and \eqref{eq : Z12 autonomous}) and \ref{lem : Z2 autonomous}, as in the proof of Theorem \ref{th : CAP forcing}, and we obtain
    \begin{align*}
        Y \leq 1.6 \times 10^{-11}, \quad Z_1 \leq 0.834 \quad \text{ and } \quad Z_2(r_*) \leq 3.3 \times 10^{4}.
    \end{align*}
    The closed ball of radius $r_*$ centered at $\bx$, which is the only region where $\mathbb{F}$ is used in Theorem \ref{thm:radPolyBanach}, is contained in $\{\beta_1 > 0\}$ since $r_* < \bar{\beta}_1$. Since $\hat{\mathcal{Z}}_2^{\beta}$ and $\hat{\mathcal{Z}}_2^{u}$ are non-decreasing in $r$, we may take $\hat{\mathcal{Z}}_2(r) = \hat{\mathcal{Z}}_2^{\beta}(r_*) + \hat{\mathcal{Z}}_2^{u}(r_*)$ for all $r \leq r_*$, so that Lemma \ref{lem : Z2 autonomous} provides \eqref{eq:Z2_radPolyBanach} with $Z_2 = Z_2(r_*)$. Moreover, $r_0 \leq r_*$, $Y + Z_1 r_0 + \frac{1}{2}Z_2 r_0^2 \leq r_0$ and $Z_1 + Z_2 r_0 < 1$, and $\mathbb{A}$ is injective as in the proof of Theorem \ref{th : CAP forcing}, so that Theorem \ref{thm:radPolyBanach} provides a zero $\tilde{x} = (\tilde{\beta}_1,\tilde{\beta}_2,\tilde{u})$ of $\mathbb{F}$ such that $\|\tilde{x} - \bx\|_H \leq r_0$. In particular, $\tilde{\beta}_1 \geq \bar{\beta}_1 - r_0 > 0$, and we conclude using Lemma \ref{lem : zeros of F autonomous}.
\end{proof}

\section{A gene regulatory model with threshold delay} \label{sec:threshold_SDDE}

In this section, we consider the threshold SDDE studied in \cite{MR4959313}, obtained as a quasi-steady state reduction of a state-dependent delay extension of the Goodwin operon model \cite{Goodwin1963,GedeonOperon2022}. Here $w$ is the effector concentration, $q$ the production flux of messenger RNA, and $v$ the transcription/translation velocity, so that the delay $\tau(t)$ is the time needed to process a strand of fixed length $a$. Rather than reducing the model to a constant-delay equation \cite{Smith1993}, the authors of \cite{MR4959313} study it directly in its threshold form, combining analysis and numerical continuation \cite{WangHumphries2022}, which reveals a rich variety of bifurcations depending on the monotonicity of $q$ and $v$. In particular, when $q$ is constant and $v$ is decreasing, Hopf bifurcations give rise to stable periodic orbits; Theorem~\ref{th : CAP threshold} below validates one of the orbits computed numerically in \cite{MR4959313}, for the same parameter values. The model reads
\begin{equation}\label{eq : appli original}
    \begin{aligned}
    w'(t) &= b e^{-\mu \tau(t)} \frac{v(w(t))}{v(w(t - \tau(t)))} q(w(t-\tau(t))) - \gamma w(t) \\
    a &= \int_{- \tau(t)}^0 v(w(s+t))ds,
\end{aligned}
\end{equation}
where the Hill functions $v,q$ are defined by
\begin{equation}
q(x) = \frac{q^{-} \theta_q^n + q^{+} x^n}{\theta_q^n + x^n} 
\quad \text{and} \quad
v(x) = \frac{v^{-} \theta_v^m + v^{+} x^m}{\theta_v^m + x^m}.
\end{equation}
Periodic solutions of threshold SDDEs with analytic right-hand side and positive analytic velocity are analytic \cite{Krisztin2014} (see also \cite{MR3229655}, where a change of time reduces the problem to a constant delay, for which \cite{MR322315} applies). Since $v>0$, this applies to the pair $(w,\tau)$, which solves \eqref{eq : appli augemented system} below with $\eta=0$. Consequently, by the Paley--Wiener theorem, the Fourier coefficients of periodic solutions of \eqref{eq : appli original} decay geometrically. This is not necessarily the case for the examples of Sections~\ref{sec:forced_toy_model} and~\ref{sec : autonomous SDDE}, for which analyticity of periodic solutions is not known a priori (see \cite{MR3229655} for $C^\infty$ but nonanalytic periodic solutions of delay-differential equations with analytic nonlinearity and analytic time-dependent delay).

The derivative of $q$, which appears in some of the estimates needed for the proof, is given by
\begin{align*}
    q'(x) = \frac{nq^+ x^{n-1}}{\theta_q^n + x^n} - n x^{n-1} \frac{q^{-} \theta_q^n + q^{+} x^n}{(\theta_q^n + x^n)^2} &= \frac{nq^+x^{n-1}(\theta_q^n + x^n) - n x^{n-1}(q^{-} \theta_q^n + q^{+} x^n)}{(\theta_q^n + x^n)^2}\\
   & = \frac{nq^+x^{n-1}\theta_q^n  - n x^{n-1}q^{-} \theta_q^n }{(\theta_q^n + x^n)^2}.
\end{align*}

Our approach is not directly applicable to~\eqref{eq : appli original}; therefore, we first reformulate it into a system of SDDEs, with an extra integral constraint no longer depending on time.  Taking a derivative in the second equation of~\eqref{eq : appli original} and using the Leibniz integral rule for differentiation, we equivalently get
\begin{align*}
    w'(t) &= b e^{-\mu \tau(t)} \frac{v(w(t))}{v(w(t - \tau(t)))} q(w(t-\tau(t))) - \gamma w(t) \\
    \tau'(t) &= 1 - \frac{v(w(t))}{v(w(t - \tau(t)))}\\
    a &= \int_{-\tau(0)}^0 v(w(s))ds.
\end{align*}

Since the above system is overdetermined, we add an unfolding parameter $\eta\in\R$ (a classical idea in the numerical continuation of periodic orbits in conservative systems, see \cite{MunozAlmaraz2003}, which has also been used in computer-assisted proofs, see e.g.~\cite{MR3545977,calleja2021torus})
and solve for
\begin{equation}\label{eq : appli augemented system}
\begin{aligned}
    w'(t) &= b e^{-\mu \tau(t)} \frac{v(w(t))}{v(w(t - \tau(t)))} q(w(t-\tau(t))) - \gamma w(t) \\
    \tau'(t) &= 1 - \frac{v(w(t))}{v(w(t - \tau(t)))} + \eta\\
    a &= \int_{-\tau(0)}^0 v(w(s))ds.
\end{aligned}
\end{equation}
In fact, we prove that having a periodic solution to \eqref{eq : appli augemented system} with $w > 0$ necessarily implies that $\eta = 0$, so that $(w,\tau)$ is a periodic solution to \eqref{eq : appli original}.

\begin{lemma}\label{lem : eta is zero}
If $(w,\tau,\eta) \in C^1_{per}([0,\theta])\times C^1_{per}([0,\theta])\times \R$ solves \eqref{eq : appli augemented system} and $v > 0$, then $\eta=0$.
\end{lemma}

\begin{proof}
Rearranging the second differential equation of \eqref{eq : appli augemented system}, we get
$$\frac{d}{dt}\left(\int_{t-\tau(t)}^t v(w(s))ds\right)=\eta v(w(t-\tau(t))).$$
Note that $t\mapsto\int_{t-\tau(t)}^{t}v(w(s))ds$ is $\theta$-periodic and $C^1$, since both $w$ and $\tau$ are in $C^1_{per}([0,\theta])$. Integrate from $0$ to $\theta$ both sides of the above equation to get $0=\eta \int_0^\theta v(w(t-\tau(t)))dt$, which implies $\eta=0$, since $v > 0$.
\end{proof}

Starting from~\eqref{eq : appli augemented system}, we introduce notation and further reformulate the system, in order to obtain a suitable $\mathbb{F}=0$ problem to study periodic solutions of~\eqref{eq : appli original}.
First, let $u = (u_1,u_2) \bydef (w,\tau)$. Then, we have that 
\begin{align}
    u'(t) = \mathcal{G}(u_1(t-u_2(t)), u(t)) + \eta e_2 \text{ with }  e_2 \bydef \begin{pmatrix}
        0\\
        1
    \end{pmatrix}\\
     \int_{-u_2(0)}^0 v(u_1(s))ds - a = 0,
\end{align}
where $\mathcal{G} : \R \times \R^2  \to \R^2$ is given as 
\begin{align}
    \mathcal{G}(y,z) \bydef \begin{pmatrix}
        b e^{-\mu z_2} \frac{v(z_1)}{v(y)} q(y) - \gamma z_1\\
        1 - \frac{v(z_1)}{v(y)} 
    \end{pmatrix}.
\end{align}
Following the notations of Section \ref{sec : autonomous SDDE}, we look for $\theta$-periodic solutions, where $\theta$ will have to be determined, and introduce $\omega \bydef \frac{\theta}{2\pi}$. Using the change of variable $t \to \frac{1}{\omega} t$, the above problem becomes
\begin{align}
    u'(t) = {\omega}\mathcal{G}(u_1(t-\frac{u_2(t)}{\omega} ), u(t)) + \omega\eta e_2 \text{ with }  e_2 \bydef \begin{pmatrix}
        0\\
        1
    \end{pmatrix} \text{ for } t \in [0, 2\pi],\\
   \omega \int_{-\frac{u_2(0)}{\omega} }^0 v(u_1(s))ds -  a = 0.
\end{align}
Then, we change $u_2$ to $\omega u_2$ and get
\begin{align}
    u' &=  \mathcal{G}(u_1(t-{u_2} ), u, \omega) + \eta e_2 \text{ with }  \mathcal{G}(y, z, \omega) \bydef \begin{pmatrix}
       \omega b e^{-\mu \omega z_2} \frac{v(z_1)}{v(y)} q(y) -  \omega \gamma z_1\\
        1 - \frac{v(z_1)}{v(y)} 
    \end{pmatrix},\\
    &\omega\int_{-{u_2(0)} }^0 v( u_1(s))ds -  a = 0.
\end{align}
To simplify, we denote 
\begin{align}
    g(\omega,u) \bydef \mathcal{G}(u_1(t-{u_2} ), u, \omega).
\end{align}

\begin{remark}
    In this section, $u : \R \to \R^2$ is vector-valued and the space $C^k_{per}([-\pi,\pi])$, with its associated norm $\|\cdot\|_{C^k}$, has to be understood in the sense of \eqref{def : C0} with $\R^m$ replaced by $\R^2$.  
\end{remark}
Then, following the notations of  Section \ref{sec:forced_toy_model}, we introduce $f :  \R^4 \times C^1_{per}([-\pi,\pi]) \to  C^1_{per}([-\pi,\pi])$ as 
\begin{equation*}
f(x) = f(\beta,u) \bydef 
\displaystyle 
u - \begin{pmatrix}
    \beta_3\\
    \beta_4
\end{pmatrix} - \mathbb{L}_1g(\beta_2, u) - \mathbb{L}_1\beta_1 e_2,
\end{equation*}
with $\beta_1 = \eta$ and $\beta_2 = \omega$, and $\mathbb{L}_1 (\varphi_1,\varphi_2)$ stands for $(\mathbb{L}_1 \varphi_1,\mathbb{L}_1 \varphi_2)$. Now, notice that
\begin{equation*}
    \mathbb{M} \eta e_2 = \frac{1}{2\pi}\int_{-\pi}^\pi \eta e_2  ds = \eta e_2
\end{equation*}
and 
\begin{equation*}
    \left(\mathbb{L}_1 \eta e_2\right)(t)  = \int_{-\pi}^t (\eta e_2 - \mathbb{M}\eta e_2) ds = 0,
\end{equation*}
since $\eta e_2$ is a constant. This implies that 
\begin{equation} \label{eq:the_map appli}
f(x) = 
\displaystyle 
u - \begin{pmatrix}
    \beta_3\\
    \beta_4
\end{pmatrix} - \mathbb{L}_1g(\beta_2, u).
\end{equation}
Now, we introduce $\mathbb{F}$ to incorporate the integral constraint and a phase condition to solve for $\omega$, that is, for the period $\theta = 2\pi\omega$. Indeed, let $\mathbb{F} : \R^4 \times C^1_{per}([-\pi,\pi]) \to \R^4 \times C^1_{per}([-\pi,\pi])$ be defined as 
\begin{align}
    \mathbb{F}(x) \bydef \begin{pmatrix}
        {\xi}(x) \\
        f(x)
    \end{pmatrix}
\end{align}
where we define $\xi : \R^4 \times C^1_{per}([-\pi,\pi]) \to \R^4$ as 
\begin{align}
    {\xi}(x) \bydef \begin{pmatrix}
         \xi_1(\beta,u)\\
         \xi_2(u) \\
         \xi_3(\beta,u)
    \end{pmatrix} = \begin{pmatrix}
 \beta_2 \int_{- u_2(0)}^0 v( u_1(s))ds -a\\      
       (\pi^{\leq N} u_1)(0) - 1\\
         \mathbb{M}g(\beta_2, u) + \beta_1 e_2 
    \end{pmatrix}    
\end{align}
where we recall that $\mathbb{M}$ is the mean value operator defined in \eqref{eq : mean_integral operator} and here is applied component-wise. We prove in the following that zeros of $\mathbb{F}$ provide periodic solutions to \eqref{eq : appli original}.
\begin{lemma}\label{lem : zeros of F threshold}
    Let $x = (\beta, u) \in \R^4 \times C^1_{per}([-\pi,\pi])$ be a zero of $\mathbb{F}$ such that $\beta_2 > 0$ and $v > 0$. Then $\beta_1 = 0$ and $(w,\tau) : s \mapsto \left(u_1\left(\frac{s}{\beta_2}\right), \beta_2 u_2\left(\frac{s}{\beta_2}\right)\right)$ is a $2\pi\beta_2$-periodic solution to \eqref{eq : appli original}.
\end{lemma}

\begin{proof}
    Since $f(x) = 0$ and $\xi_3(x) = 0$, we have $u' = g(\beta_2,u) - \mathbb{M}g(\beta_2,u) = g(\beta_2,u) + \beta_1 e_2$, as in the proof of Lemma~\ref{lem : zeros of F are solutions}. Undoing the above amplitude and time rescalings, and using $\xi_1(x) = 0$, we obtain that $(w,\tau,\beta_1)$ is a $2\pi\beta_2$-periodic solution to \eqref{eq : appli augemented system}. Since $v>0$, Lemma~\ref{lem : eta is zero} yields $\beta_1 = 0$. Finally, the proof of Lemma~\ref{lem : eta is zero} shows that $t \mapsto \int_{t-\tau(t)}^{t} v(w(s))ds$ is then constant, and it is equal to $a$ at $t=0$.
\end{proof}

\begin{remark}
    In Section \ref{sec:forced_toy_model}, we introduced $H = \R^m \times C^1_{per}([-\pi,\pi])$ with $C^1_{per}([-\pi,\pi])$ containing vector-valued functions in $\R^m$. In the specific case of this section, the analysis requires extra parameters, namely $\omega$ (equivalently, the period $\theta = 2\pi\omega$) and the unfolding parameter $\eta$. This leads to a larger parameter space $\R^4$, instead of $\R^2$ from the original definition of Section \ref{sec:forced_toy_model}. However, the analysis and estimates derived in Section \ref{sec:forced_toy_model} can still be applied, modulo a rank 2 perturbation.
\end{remark}
Let $H_0 \bydef \R^4 \times C^0_{per}([-\pi,\pi])$ and let $H \bydef \R^4 \times C^1_{per}([-\pi,\pi])$ be endowed with the norm
\begin{align}
    \|x\|_H = \max\{|\beta|_\infty, \|u\|_{C^0},  \|u'\|_{C^0}\}.
\end{align}
Let $\bx$ be the approximate solution represented in Figure~\ref{fig : threshold}, and let $\mathbb{A} : H \to H$ be an approximate inverse for $D\mathbb{F}(\bx)$ such that 
\begin{align}
    \mathbb{A} = \mathbb{A}^{\leq N} + \pi^{>N}.
\end{align}
Finally, we represent in Figure~\ref{fig : threshold} the approximate solution $\bx = (\bar{\beta},\bu)$, with $\bar{u} = (\bw, \btau)$, obtained for the parameters of Figure 5.7 in \cite{MR4959313}, together with the Fourier coefficients of $\bw$ and $\btau$.
\begin{figure}[ht]
\centering
\includegraphics[width=1\textwidth]{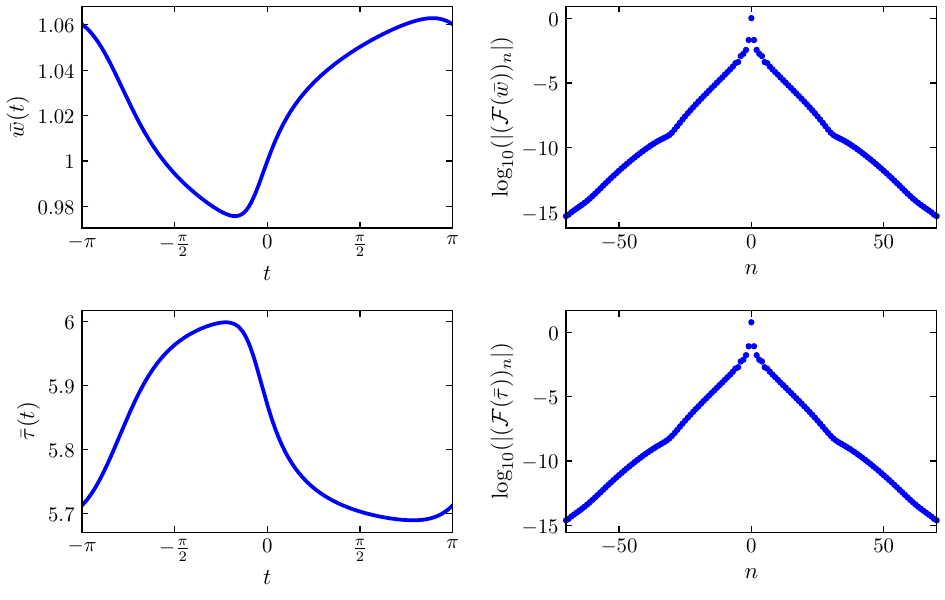}
\caption{Approximate solution $\bar{u} = (\bw, \btau)$ for the parameters of Figure 5.7 in \cite{MR4959313}, that is $a = 3$, $b = 3$, $\mu = 0.3$, $\gamma = 0.13$, $m = 56$, $v^- = 0.5$, $v^+ = 0.2$, $\theta_v = 1$ and $q^- = q^+ = 1$. Left: $\bw$ (top) and the rescaled delay $\btau$ (bottom), so that $\tau(s) \approx \bar{\beta}_2\btau(s/\bar{\beta}_2)$, given as trigonometric polynomials of order $70$. Right: $\log_{10}$ of the absolute value of their Fourier coefficients. The corresponding periodic solution of \eqref{eq : appli original} has period $2\pi\bar{\beta}_2 \approx 11.22$.}
\label{fig : threshold}
\end{figure}

\subsection{Computation of bounds}

\subsubsection{\boldmath$Y$\unboldmath~bound}

The computation of the bound $Y$ requires, with $\bar{\omega} \bydef \bar{\beta}_2$, computing a rigorous enclosure of the Fourier coefficients of 
\begin{align*}
g(\bar \omega, \bu) &= \mathcal{G}(\bw(t-\bar \tau(t)), \bu, \bar \omega) \\
& = 
\begin{pmatrix}
\bar \omega b e^{-\mu \bar\omega \bar \tau} \frac{v(\bw)}{v(\bw(t-\bar \tau(t)))} q(\bw(t-\bar \tau(t))) -  \bar \omega \gamma \bw
\\
1 - \frac{v(\bw)}{v(\bw(t-\bar \tau(t)))} 
\end{pmatrix} = \begin{pmatrix}
    \bar{g}_1\\
    \bar{g}_2
\end{pmatrix}.
\end{align*}
In this case, we cannot follow identically the strategy developed in Section \ref{sec:fourier_coeff_delayed_term} since $q$ and $v$ are not trigonometric polynomials. However, note that if we can prove that $\bar{g}_1$ and $\bar{g}_2$ are analytic in a strip $S_{\bar{\rho}} = \left\{z \in \mathbb{C} : |\mathrm{Im}(z)| \leq \bar{\rho}\right\}$, with $\bar{\rho} >0$, then the computation of the tail estimates \eqref{eq : tail estimates} and the aliasing error \eqref{eq:aliasing_control} will be identical. Consequently, we need to derive a strategy for verifying that the given functions $\bar{g}_1$ and $\bar{g}_2$, defined thanks to the trigonometric polynomial $\bar{u}$, are analytic on $S_{\bar{\rho}}$, for a well-chosen $\bar{\rho}$.

\paragraph{Verification of the analyticity domain}

Recalling the definition of the Hill functions $v$ and $q$, the terms appearing in $\bar{g}_1$ and $\bar{g}_2$ can be written as
\begin{equation}\label{eq : rational terms threshold}
\begin{aligned}
    v(\bw(z)) &= \frac{v^- \theta_v^m + v^+ \bw(z)^m}{\theta_v^m + \bw(z)^m}, \qquad 
    \frac{1}{v(\bw(z-\btau(z)))} = \frac{\theta_v^m + \bw(z-\btau(z))^m}{v^- \theta_v^m + v^+ \bw(z-\btau(z))^m} \\
    q(\bw(z-\btau(z))) &= \frac{q^- \theta_q^n + q^+ \bw(z-\btau(z))^n}{\theta_q^n + \bw(z-\btau(z))^n}.
\end{aligned}
\end{equation}
Since $\bw$ and $\btau$ are trigonometric polynomials, the maps $z \mapsto \bw(z)$, $z \mapsto \bw(z-\btau(z))$ and $z \mapsto e^{-\mu \bar{\omega} \btau(z)}$ are entire. Consequently, $\bar{g}_1$ and $\bar{g}_2$ are analytic on $S_{\bar{\rho}}$ as soon as the three denominators in \eqref{eq : rational terms threshold} do not vanish on $S_{\bar{\rho}}$. Now, the roots of the polynomials $\theta_v^m + y^m$, $v^- \theta_v^m + v^+ y^m$ and $\theta_q^n + y^n$ are explicitly given by
\begin{align}\label{eq : roots denominators threshold}
    \theta_v e^{i \pi \frac{2k+1}{m}}, \quad \theta_v \left(\frac{v^-}{v^+}\right)^{\frac{1}{m}} e^{i \pi \frac{2k+1}{m}}, \quad k = 0, \dots, m-1, \quad \text{ and } \quad \theta_q e^{i \pi \frac{2k+1}{n}}, \quad k = 0, \dots, n-1,
\end{align}
respectively. In particular, since all the parameters are positive, none of these roots is a positive real number. Therefore, we need to verify that $\bw(S_{\bar{\rho}})$ does not contain any root of $\theta_v^m + y^m$, and that $\left\{\bw(z-\btau(z)), ~ z \in S_{\bar{\rho}}\right\}$ does not contain any root of $v^- \theta_v^m + v^+ y^m$ and $\theta_q^n + y^n$. Following \cite[Section 3.5]{vandenberg2026spectral}, this is achieved thanks to the following lemma, which only requires controlling the image of the boundary of the strip.

\begin{lemma}\label{lem : avoiding roots}
Let $h : \C \to \C$ be a $2\pi$-periodic entire function, let $p \in \C$ and let $\partial S_{\bar{\rho}} \bydef \left\{z \in \C, ~ |\mathrm{Im}(z)| = \bar{\rho}\right\}$. Assume that there exists $\phi \in \R$ such that
\begin{align}\label{eq : half plane condition}
    \mathrm{Re}\left(e^{-i\phi}\left(h(z) - p\right)\right) > 0 \quad \text{ for all } z \in \partial S_{\bar{\rho}}.
\end{align}
Then $h(z) \neq p$ for all $z \in S_{\bar{\rho}}$.
\end{lemma}

\begin{proof}
The map $z \mapsto \mathrm{Re}\left(e^{-i\phi}\left(h(z) - p\right)\right)$ is harmonic on $\C$ and $2\pi$-periodic. In particular, it reaches its minimum on $S_{\bar{\rho}}$ at some point $z_0 \in S_{\bar{\rho}}$. If $z_0 \in \partial S_{\bar{\rho}}$, then this minimum is positive by \eqref{eq : half plane condition}. Otherwise, $z_0$ is an interior point of $S_{\bar{\rho}}$ and the minimum principle for harmonic functions implies that the map is constant on $S_{\bar{\rho}}$, hence positive by \eqref{eq : half plane condition} again. In both cases, we obtain that $\mathrm{Re}\left(e^{-i\phi}\left(h(z) - p\right)\right) > 0$ for all $z \in S_{\bar{\rho}}$, which implies that $h(z) \neq p$.
\end{proof}

We apply Lemma \ref{lem : avoiding roots} to $h = \bw$, with $p$ each root of $\theta_v^m + y^m$, and to $h = \bw(\cdot - \btau(\cdot))$, with $p$ each root of $v^- \theta_v^m + v^+ y^m$ and $\theta_q^n + y^n$. The verification of \eqref{eq : half plane condition} is achieved thanks to rigorous numerics. Specifically, since $\partial S_{\bar{\rho}}$ consists of the two lines $\mathrm{Im}(z) = \pm \bar{\rho}$, we proceed as in Section~\ref{sec:tail_coefficients}. Namely, for each $k=0,\dots,N_{\tt FFT}-1$, the interval vector $\boldsymbol{V}_k$ of Section~\ref{sec:tail_coefficients} (with $\bu = \bw$) encloses $\bw(t - i\bar{\rho} - \btau(t - i \bar{\rho}))$ for all $t \in t_k + \boldsymbol{\delta}$, whereas an enclosure of $\bw(t - i\bar{\rho})$ for all $t \in t_k + \boldsymbol{\delta}$ is obtained exactly as $\boldsymbol{\tau}_k$ in \eqref{eq:definition_Uk}, replacing $\btau$ by $\bw$. Repeating these computations on the line $\mathrm{Im}(z) = \bar{\rho}$, we obtain an enclosure of $h(\partial S_{\bar{\rho}})$ by a finite union of interval vectors. Note that these enclosures are needed anyway for the computation of the constant $C$ in \eqref{eq : tail estimates}. Then, given a root $p$, we first verify that $p$ does not belong to any of these interval vectors, and we compute, for each of them, an interval containing the argument of its elements translated by $-p$. If all these intervals are contained in a common interval of the form $[\phi - \phi_0, \phi + \phi_0]$ with $\phi_0 < \frac{\pi}{2}$, then \eqref{eq : half plane condition} holds since
\begin{align*}
    \mathrm{Re}\left(e^{-i\phi}\left(h(z) - p\right)\right) = |h(z) - p| \cos\left( \arg(h(z)-p) - \phi \right) \quad \text{ for all } z \in \partial S_{\bar{\rho}}.
\end{align*}
The whole process is carried out rigorously using interval arithmetic.

Note that \eqref{eq : half plane condition} is only a sufficient condition to obtain that $p\notin h(S_{\bar{\rho}})$. However, since $h$ is real-valued on the real axis, $h(\partial S_{\bar{\rho}})$ collapses onto the real segment $h(\R) \subset \bw(\R)$ as $\bar{\rho} \to 0$. Hence, as $\bw$ is positive and none of the roots in \eqref{eq : roots denominators threshold} is a positive real number, \eqref{eq : half plane condition} is satisfied for $\bar{\rho}$ small enough. In practice, $\bar{\rho}$ is chosen a priori, and the above verification is performed before computing the tail estimates \eqref{eq : tail estimates}. 

\begin{remark}
The functions $\bar{v}_1, \dots, \bar{v}_4$ introduced in the computation of the bound $Z_1$ below, whose Fourier coefficients also need to be enclosed, are rational functions of $\bw$ and $\bw(\cdot - \btau)$, with coefficients polynomial in $\btau$, $\bw'(\cdot - \btau)$ and $e^{-\mu \bar{\omega} \btau}$, and with denominators given by products of powers of the three denominators in \eqref{eq : rational terms threshold}. Consequently, the above verification also provides their analyticity on $S_{\bar{\rho}}$ (for $\bar{v}_1$ and $\bar{v}_3$, which contain $\frac{v'(y)}{v(y)^2}$ with $y = \bw(\cdot - \btau)$, after the simplification $\frac{v'(y)}{v(y)^2} = \frac{m\theta_v^m(v^+-v^-)y^{m-1}}{(v^-\theta_v^m+v^+y^m)^2}$), as well as that of $v(\bw)$ and $v'(\bw)$.
\end{remark}

\subsubsection{\boldmath$Z_1$\unboldmath~bound}

We first introduce notation for the computation of the Jacobian.
We introduce the following decomposition
\begin{align*}
    D\mathbb{F}(\bx) = \mathbb{I} + \mathbb{D}^{\dagger} D\mathbb{Q}(\bx)
\end{align*}
where 
\begin{align}\label{eq : decomposition of DF}
    \mathbb{D}^\dagger(\beta,u) = (\beta, \mathbb{L}_1 u), ~~ \mathbb{I}(\beta,u) = (0, (u_1-\beta_3, u_2-\beta_4)) ~ \text{ and } ~  \mathbb{Q}(\beta,u) = \begin{pmatrix}
        \xi(x)\\
        -g(\beta_2,u)
    \end{pmatrix}.
\end{align}
In particular, we have that 
\begin{align*}
    D\mathbb{Q}(\bar{x}) = \begin{pmatrix}
    \partial_{\beta}\xi(\bar{x})& \partial_u \xi(\bar{x})\\
     -\partial_{\beta }g(\bar{x}) & -\partial_u g(\bar{x})
    \end{pmatrix}.
\end{align*}

\paragraph{Computation of \boldmath$Dg(\bar{x})$\unboldmath}

In fact we have that 
\begin{align}
   \partial_u g(\bar{\beta}_2,\bu) = \partial_1 \mathcal{G}(\bu_1(t-\bu_2),\bu,\bb_2) \partial_{u} \psi_1(\bu) +    \partial_2 \mathcal{G}(\bu_1(t-\bu_2),\bu,\bb_2).
\end{align}
Here, $\psi_1(u) \bydef u_1(t-u_2)$. Therefore,
\begin{align}
    \partial_u \psi_1(\bu) u = u_1(t-\bu_2) - \bu_1'(t-\bu_2)u_2.
\end{align}
This implies that 
\begin{align}\label{eq : computation d g u}
     \partial_u g(\bar{\beta}_2,\bu)u &= \bar{v}_1(t) u_1(t-\bar{u}_2) + \bar{v}_2(t)u_1(t) + \bar{v}_3u_2(t)\\
     &= \bar{v}_1(t) (\mathbb{L}_3(\bar{u}_2) u_1)(t) + \bar{v}_2(t)u_1(t) + \bar{v}_3u_2(t)
\end{align}
where 
\begin{align*}
    \bar{v}_1 &= \partial_1 \mathcal{G}(\bu_1(t-\bu_2),\bu,\bb_2)\\
    \bar{v}_2 &= \partial_2 \mathcal{G}(\bu_1(t-\bu_2),\bu,\bb_2)_{col(1)}\\
    \bar{v}_3 &= -\bar{v}_1\bar{u}'_1(t-\bar{u}_2) + \partial_2 \mathcal{G}(\bu_1(t-\bu_2),\bu,\bb_2)_{col(2)}\\
    \mathbb{L}_3(\bar{u}_2) u_1 &= u_1(t-\bar{u}_2).
\end{align*}
In particular the Fourier series representation of $\mathbb{L}_3(\bar{u}_2)$ is given by \eqref{def : operator L3}, with $1+\varepsilon_2\bar{u}$ replaced by $\bar{u}_2$. Then, we get
\begin{align*}
    \partial_1 \mathcal{G}(y,z,\omega) = \begin{pmatrix}
       \omega b e^{-\mu \omega z_2}  \frac{ v(z_1)}{v(y)} q'(y)  - \omega b e^{-\mu \omega z_2}  \frac{v'(y) v(z_1)}{v(y)^2} q(y)\\
         \frac{v'(y) v(z_1)}{v(y)^2}
    \end{pmatrix}
\end{align*}
and
\begin{align*}
    \partial_2 \mathcal{G}(y,z,\omega) = \begin{pmatrix}
         \omega b e^{-\mu \omega z_2}  \frac{ v'(z_1)}{v(y)} q(y) - \omega \gamma  & - \mu  \omega^2 b e^{-\mu \omega z_2}  \frac{ v(z_1)}{v(y)} q(y)\\
         -\frac{ v'(z_1)}{v(y)} & 0
    \end{pmatrix}.
\end{align*}

Then we have that 
\begin{align}\label{def : bar v4}
    \partial_\beta g(\bar{x}) \beta = \partial_{\beta_2} g(\bar{x})\beta_2  =  \bar{v}_4 \beta_2, ~~ \text{ where } \bar{v}_4 \bydef \partial_3 \mathcal{G}(\bu_1(t-\bu_2),\bu,\bb_2).
\end{align}
In fact, we have that 
\begin{align*}
    \partial_3 \mathcal{G}(y,z,\omega) = \begin{pmatrix}
        be^{-\mu \omega z_2}\frac{v(z_1)}{v(y)}q(y)\left(1 - \omega \mu z_2\right) - \gamma z_1\\
        0
    \end{pmatrix}.
\end{align*}

\paragraph{Computation of \boldmath$D\xi(\bar{x})$\unboldmath}

Now, we have
\begin{equation}\label{eq : computation d xi u}
    \begin{aligned}
    \partial_{u} \xi_1(\bx) u &= \bar{\beta}_2\int^0_{-\bu_2(0)} v'(\bu_1(s)) u_1(s) ds +  \bar{\beta}_2 v(\bu_1(-\bu_2(0)))u_2(0)\\
    \partial_u \xi_2(\bar{x}) u &= (\pi^{\leq N}u_1)(0)\\
    \partial_u \xi_3(\bx)u &= \mathbb{M}\partial_u g(\bar{x})u
    \end{aligned}
\end{equation}

Given $h \in C^1_{per}([-\pi,\pi])$, we have that $\int^0_{-\bu_2(0)} h(s) ds = \sum_{n \in \mathbb{Z}} (\mathcal{F}(h))_n \vartheta_n$, where
\begin{align*}
    \vartheta_n \bydef \int_{-\bar{u}_2(0)}^{0} e^{ins}ds  =  \frac{1- e^{-in\bar{u}_2(0)}}{i n} ~ \text{ for } n \neq 0, \quad \text{ and } \quad \vartheta_0 = \bu_2(0).
\end{align*}
Applying this identity to $h = v'(\bu_1)u_1$, it remains to compute the Fourier coefficients of $v'(\bu_1)u_1$, that is, the multiplication operator by $v'(\bu_1)$ in Fourier coefficients.

On the other hand, we have 
\begin{align*}
    \partial_\beta \xi_1(\bar{x})\beta &=  \beta_2 \int_{- \bar{u}_2(0)}^0 v( \bar{u}_1(s))ds\\
    \partial_\beta \xi_2(\bar{x}) &= 0\\
    \partial_{\beta} \xi_3(\bar{x})\beta &= \mathbb{M} \partial_{\beta_2}g(\bar{x})\beta_2 + \beta_1 e_2
\end{align*}

For the parameters of Figure~\ref{fig : threshold}, estimating $\pi^{\leq N}(I - \mathbb{A}D\mathbb{F}(\bx))\pi^{>N}$ as in Lemma \ref{lem : Z1 forcing}, that is, bounding separately $\|\mathbb{A}^{\leq N}\mathbb{D}^\dagger\|_{B(H_0)}$ and $\|D\mathbb{Q}(\bx)\pi^{>N}\|_{B(H,H_0)}$, is too pessimistic. Instead, given $N_2 > N$, we split $\pi^{>N} = \pi^{(N,N_2]} + \pi^{>N_2}$, where $\pi^{(N,N_2]} \bydef \pi^{\leq N_2} - \pi^{\leq N}$. The operator $\mathbb{A}^{\leq N}D\mathbb{F}(\bx)\pi^{(N,N_2]}$ is computed explicitly, and for the modes $|k| > N_2$, we use that the delayed term $u_1(t-\bu_2(t))$ only transfers them to the modes $|n| \leq N$ with an exponentially small factor. This is the purpose of the next lemma.

\begin{lemma}\label{lem : frequency gap}
    Let $\bar{\rho} > 0$, let $\bar{\nu} \bydef e^{\bar{\rho}}$ and let $\phi$ be a $2\pi$-periodic function, analytic on $S_{\bar{\rho}}$. Moreover, let $C_\phi \geq \frac{1}{2\pi}\int_{-\pi}^{\pi}|\phi(t \pm i\bar{\rho})|dt$ for both signs, and let $\lambda > 0$ be such that
    \begin{align}\label{eq : condition lambda}
        \mathrm{Im}(\bu_2(t+i\bar{\rho})) \leq (1-\lambda)\bar{\rho} \quad \text{ for all } t \in [-\pi,\pi].
    \end{align}
    Then $\left|\left(\mathcal{F}\left(\phi e^{ik(\cdot - \bu_2)}\right)\right)_n\right| \leq C_\phi \bar{\nu}^{|n|-\lambda|k|}$ for all $n,k \in \Z$.
\end{lemma}

\begin{proof}
    Let $n,k \in \Z$ with $k \geq 0$. Since $\bu_2$ is a trigonometric polynomial, $\phi e^{ik(\cdot - \bu_2)}$ is analytic on $S_{\bar{\rho}}$, and shifting the line of integration as in Section \ref{sec:tail_coefficients}, we get
    \begin{align*}
        \left(\mathcal{F}\left(\phi e^{ik(\cdot - \bu_2)}\right)\right)_n = \frac{1}{2\pi}\int_{-\pi}^{\pi} \phi(t+i\bar{\rho})\, e^{ik(t+i\bar{\rho} - \bu_2(t+i\bar{\rho}))}\, e^{-in(t+i\bar{\rho})} dt.
    \end{align*}
    Now, \eqref{eq : condition lambda} yields $|e^{ik(t+i\bar{\rho} - \bu_2(t+i\bar{\rho}))}| = e^{-k(\bar{\rho} - \mathrm{Im}(\bu_2(t+i\bar{\rho})))} \leq \bar{\nu}^{-\lambda k}$, and $|e^{-in(t+i\bar{\rho})}| = e^{n\bar{\rho}} \leq \bar{\nu}^{|n|}$, which proves the result for $k \geq 0$. If $k < 0$, we proceed identically with the line $\mathrm{Im}(z) = -\bar{\rho}$, using that $\mathrm{Im}(\bu_2(t-i\bar{\rho})) = -\mathrm{Im}(\bu_2(t+i\bar{\rho}))$ since $\bu_2$ is real-valued.
\end{proof}

\begin{lemma}\label{lem : Z1 appli}
    Let $N_2 > N$, let $\bar{\rho} > 0$ be such that $\bar{v}_1, \bar{v}_2, \bar{v}_3$ are analytic on $S_{\bar{\rho}}$ and  let $\bar{\nu} \bydef e^{\bar{\rho}}$. Now, let $\lambda > 0$ satisfy \eqref{eq : condition lambda} and let $C_1,C_2,C_3$ be such that $\max_{i = 1,2}\frac{1}{2\pi}\int_{-\pi}^{\pi}|(\bar v_j)_i(t \pm i\bar{\rho})|dt \leq C_j$ for both signs, $j \in \{1,2,3\}$. Moreover, let
    \begin{align*}
        \mathcal{E} &\bydef \frac{4\bar{\nu}^{N+1}}{(\bar{\nu}-1)(N_2+1)}\left( \frac{C_1\bar{\nu}^{-\lambda(N_2+1)}}{1-\bar{\nu}^{-\lambda}} + \frac{(C_2+C_3)\bar{\nu}^{-N_2}}{\bar{\nu}-1} \right),\\
        \mathcal{X} &\bydef |\bar{\beta}_2|\left( \left( 2\|v'(\bu_1)\|_{C^0} + |\bu_2(0)|\, \|v''(\bu_1)\|_{C^0} \|\bu_1'\|_{C^0} \right)\theta_{N_2}^2 + |v(\bu_1(-\bu_2(0)))|\,\theta_{N_2} \right),
    \end{align*}
    and let $\mathcal{A} \geq \max\{\|\mathbb{A}^{\leq N}y\|_{H_0}, \|\mathbb{D}\mathbb{A}^{\leq N}y\|_{H_0}\}$ for all $y = (\beta,c) \in \R^4 \times C^0_{per}([-\pi,\pi])$ such that $c = (c_1,c_2)$ is constant, $\beta_2 = 0$ and $|\beta_1| + |\beta_3| + |\beta_4| + |c_1| + |c_2| \leq 1$. Finally, let $Z_{1,0}$ be as in \eqref{eq : components Z1} and let $Z_{1,1}$ and $Z_{1,2}$ be bounds satisfying
    \begin{equation}\label{eq : components Z1 appli}
        \begin{aligned}
            Z_{1,1} &\geq \max\left\{ \| \mathbb{A}^{\leq N}D\mathbb{F}(\bx)\pi^{(N,N_2]}\mathbb{D}^\dagger\|_{B(H_0)}, ~ \| \mathbb{D} \mathbb{A}^{\leq N}D\mathbb{F}(\bx)\pi^{(N,N_2]}\mathbb{D}^\dagger\|_{B(H_0)} \right\}\\
            & \qquad + \max\left\{ \| \mathbb{A}^{\leq N}\mathbb{D}^\dagger\|_{B(H_0)}, ~ \| \mathbb{D} \mathbb{A}^{\leq N}\mathbb{D}^\dagger\|_{B(H_0)} \right\}\mathcal{E}\\ & \qquad + \mathcal{A}\left( \mathcal{X} + 2\mathcal{E} + 2\left(\|\bar{v}_1\|_{C^0} + \|\bar{v}_2\|_{C^0} + \|\bar{v}_3\|_{C^0}\right)\theta_N\theta_{N_2} \right)\\
            Z_{1,2} & \geq \|\pi^{>N} \bar{v}_4\Vert_{C^0} + \left(\|(1-\bar{u}_2')\bar{v}_1\|_{C^0} + \|\bar{v}_2\|_{C^0} + \|\bar{v}_3\|_{C^0} + \|\bar{v}_1'\|_{C^0} + \|\bar{v}_2'\|_{C^0} + \|\bar{v}_3'\|_{C^0}\right) \theta_N.
        \end{aligned}
    \end{equation}
    Then, defining $Z_1 = Z_{1,0}  + Z_{1,1} + Z_{1,2}$, we have that $\|I - \mathbb{A} D\mathbb{F}(\bar{x})\|_{B(H)} \leq Z_1.$
\end{lemma}

\begin{proof}
Recall that $H_0 = \R^4 \times C^0_{per}([-\pi,\pi])$ is equipped with the norm $\|(\beta,u)\|_{H_0} = \max\{|\beta|_\infty, \|u\|_{C^0}\}$, let $\mathbb{D}(\beta,u) = (\beta,u')$ and let $\|\cdot\|_{H_1}$ be defined as in Section \ref{sec:zero_finding_problem_toy}, so that \eqref{eq : identity norm H with H0 H1} and \eqref{eq : identity norm H1 and D} hold in the present setting as well. Using \eqref{eq : decomposition of DF}, we still have $\mathbb{I}\pi^{>N} = \pi^{>N}$ and $\mathbb{A}^{\leq N}\mathbb{I}\pi^{>N} = 0$, so that the proof of Lemma \ref{lem : Z1 forcing} applies verbatim, except for the estimates of $\pi^{\leq N}(I - \mathbb{A}D\mathbb{F}(\bx))\pi^{>N} = -\mathbb{A}^{\leq N}D\mathbb{F}(\bx)\pi^{>N}$ and of the model-dependent quantity $\|\pi^{>N} D\mathbb{Q}(\bar{x})x\|_{H_0}$, appearing in $Z_{1,2}$. More specifically, it remains to prove that
\begin{align}\label{eq : Z11 appli goal}
    \|\mathbb{A}^{\leq N}D\mathbb{F}(\bx)\pi^{>N}x\|_{H_0} \leq Z_{1,1} ~ \text{ and } ~ \|\mathbb{D}\mathbb{A}^{\leq N}D\mathbb{F}(\bx)\pi^{>N}x\|_{H_0} \leq Z_{1,1} \quad \text{ for all } x \in B_1(0) \subset H,
\end{align}
and to estimate $\|\pi^{>N} D\mathbb{Q}(\bar{x})x\|_{H_0}$. Moreover, as in Section \ref{sec : autonomous SDDE}, the map $g$ now also depends on a parameter, and the term $\|\pi^{>N}\bar{v}_4\|_{C^0}$ in $Z_{1,2}$ arises exactly as in \eqref{eq : Z12 autonomous}.

Let $x = (\alpha, h) \in B_1(0) \subset H$. We only prove the first inequality in \eqref{eq : Z11 appli goal}, the second one being obtained identically, replacing $\mathbb{A}^{\leq N}$ by $\mathbb{D}\mathbb{A}^{\leq N}$. First, since $\mathbb{D}^\dagger\mathbb{D}x = (\alpha, h - h(-\pi))$, we have $\pi^{(N,N_2]}x = \pi^{(N,N_2]}\mathbb{D}^\dagger\mathbb{D}x$, and $\|\mathbb{D}x\|_{H_0} \leq \|x\|_H \leq 1$, so that
\begin{align*}
    \|\mathbb{A}^{\leq N}D\mathbb{F}(\bx)\pi^{(N,N_2]}x\|_{H_0} \leq \|\mathbb{A}^{\leq N}D\mathbb{F}(\bx)\pi^{(N,N_2]}\mathbb{D}^\dagger\|_{B(H_0)}.
\end{align*}
Now, let $u \bydef \pi^{>N_2}h$ and $f \bydef \partial_u g(\bx)u$. Since $\mathbb{A}^{\leq N}\mathbb{I}\pi^{>N_2} = 0$ and $\partial_u\xi_2(\bx)\pi^{>N_2} = 0$, we have, using \eqref{eq : decomposition of DF},
\begin{align*}
    \mathbb{A}^{\leq N}D\mathbb{F}(\bx)\pi^{>N_2}x = \mathbb{A}^{\leq N}\mathbb{D}^\dagger D\mathbb{Q}(\bx)\pi^{>N_2}x = \mathbb{A}^{\leq N}\left(\partial_u\xi_1(\bx)u, 0, \mathbb{M}f, -\pi^{\leq N}\mathbb{L}_1 f\right).
\end{align*}
Moreover, $\pi^{\leq N}\mathbb{L}_1 f = \mathbb{L}_1\pi^{\leq N}f + c$, where $c \bydef (\mathcal{F}(\mathbb{L}_1\pi^{>N}f))_0 \in \R^2$. Therefore, we obtain that
\begin{align*}
    \mathbb{A}^{\leq N}D\mathbb{F}(\bx)\pi^{>N_2}x = \mathbb{A}^{\leq N}y - \mathbb{A}^{\leq N}\mathbb{D}^\dagger(0,\pi^{\leq N}f), \quad \text{ where } y \bydef (\partial_u\xi_1(\bx)u, 0, \mathbb{M}f, -c),
\end{align*}
and consequently
\begin{align}\label{eq : Z11 appli far}
    \|\mathbb{A}^{\leq N}D\mathbb{F}(\bx)\pi^{>N_2}x\|_{H_0} \leq \mathcal{A}\left( |\partial_u\xi_1(\bx)u| + |\mathbb{M}f|_1 + |c|_1 \right) + \|\mathbb{A}^{\leq N}\mathbb{D}^\dagger\|_{B(H_0)}\|\pi^{\leq N}f\|_{C^0}.
\end{align}
We now estimate each term of \eqref{eq : Z11 appli far}. Since $h \in C^1_{per}([-\pi,\pi])$, its Fourier series converges absolutely and $|(\mathcal{F}(h_i))_k| \leq \frac{\|h_i'\|_{C^0}}{|k|} \leq \frac{1}{|k|}$ for all $k \neq 0$. Therefore, using \eqref{eq : computation d g u}, we have
\begin{align*}
    (\mathcal{F}(f))_n = \sum_{|k| > N_2} \Big[ (\mathcal{F}(h_1))_k \left( \left(\mathcal{F}\left(\bar{v}_1 e^{ik(\cdot - \bu_2)}\right)\right)_n + (\mathcal{F}(\bar{v}_2))_{n-k} \right) + (\mathcal{F}(h_2))_k (\mathcal{F}(\bar{v}_3))_{n-k} \Big]
\end{align*}
for all $n \in \Z$. Using Lemma \ref{lem : frequency gap} for the first term and \eqref{eq : tail estimates} for the other two, we get, component-wise,
\begin{align*}
    |(\mathcal{F}(f))_n| \leq \sum_{|k| > N_2} \frac{\bar{\nu}^{|n|}}{|k|}\left( C_1\bar{\nu}^{-\lambda|k|} + (C_2+C_3)\bar{\nu}^{-|k|} \right) \leq \frac{2\bar{\nu}^{|n|}}{N_2+1}\left( \frac{C_1\bar{\nu}^{-\lambda(N_2+1)}}{1-\bar{\nu}^{-\lambda}} + \frac{(C_2+C_3)\bar{\nu}^{-N_2}}{\bar{\nu}-1} \right).
\end{align*}
Since $\sum_{|n| \leq N}\bar{\nu}^{|n|} \leq \frac{2\bar{\nu}^{N+1}}{\bar{\nu}-1}$, this implies that $\sum_{|n| \leq N}|(\mathcal{F}(f_i))_n| \leq \mathcal{E}$ for $i \in \{1,2\}$. In particular, $\|\pi^{\leq N}f\|_{C^0} \leq \mathcal{E}$ and $|\mathbb{M}f|_1 = |(\mathcal{F}(f_1))_0| + |(\mathcal{F}(f_2))_0| \leq 2\mathcal{E}$. Then, \eqref{eq : computation d g u} and \eqref{def : Kolmogorov constant} provide $\|f\|_{C^0} \leq \left(\|\bar{v}_1\|_{C^0} + \|\bar{v}_2\|_{C^0} + \|\bar{v}_3\|_{C^0}\right)\|u\|_{C^0} \leq \left(\|\bar{v}_1\|_{C^0} + \|\bar{v}_2\|_{C^0} + \|\bar{v}_3\|_{C^0}\right)\theta_{N_2}$, and since $|c_i| = |(\pi^{>N}\mathbb{L}_1 f_i)(\pi)| \leq \theta_N\|f_i\|_{C^0}$ by~\eqref{eq:piNL1piNcst} and~\eqref{eq:piNL1}, we get $|c|_1 \leq 2\left(\|\bar{v}_1\|_{C^0} + \|\bar{v}_2\|_{C^0} + \|\bar{v}_3\|_{C^0}\right)\theta_N\theta_{N_2}$. Finally, let $p \bydef \pi^{>N_2}\mathbb{L}_1 h_1$. Using \eqref{eq : id dt versus L}, we have $p' = u_1$, and using \eqref{def : Kolmogorov constant} twice, $\|p\|_{C^0} \leq \theta_{N_2}\|u_1\|_{C^0} \leq \theta_{N_2}^2$. Then, an integration by parts in \eqref{eq : computation d xi u} provides
\begin{align*}
    |\partial_u\xi_1(\bx)u| &= |\bar{\beta}_2|\left| \left[ v'(\bu_1)p \right]_{-\bu_2(0)}^0 - \int_{-\bu_2(0)}^0 v''(\bu_1(s))\bu_1'(s)p(s)ds + v(\bu_1(-\bu_2(0)))u_2(0) \right| \leq \mathcal{X},
\end{align*}
where we used again~\eqref{def : Kolmogorov constant} through $|u_2(0)| \leq \|\pi^{>N_2}h_2\|_{C^0} \leq \theta_{N_2}$. Combining the above estimates with \eqref{eq : Z11 appli far}, and using that $\pi^{>N} = \pi^{(N,N_2]} + \pi^{>N_2}$, we obtain \eqref{eq : Z11 appli goal}.

We now turn to $\|\pi^{>N} D\mathbb{Q}(\bar{x})x\|_{H_0}$. Since $\pi^{>N}$ vanishes on the $\R^4$ component of $D\mathbb{Q}(\bar{x})x$, and since $\partial_\beta g(\bar{x})\alpha = \bar{v}_4\alpha_2$ by \eqref{def : bar v4}, we get using \eqref{def : Kolmogorov constant} that
\begin{align}\label{eq : Z1 step 0 b}
    \|\pi^{>N} D\mathbb{Q}(\bar{x})x\|_{H_0} = \|\pi^{>N}\left(\bar{v}_4\alpha_2 + \partial_u g(\bx)h\right)\Vert_{C^0}  \leq \|\pi^{>N} \bar{v}_4\Vert_{C^0}  + \sup_{\|h\|_{C^1}=1} \|\partial_t[\partial_u g(\bx)h]\Vert_{C^0} \theta_N.
\end{align}
It remains to compute $\sup_{\|h\|_{C^1}=1}\|\partial_t\left[\partial_u g(\bx)  h\right]\Vert_{C^0}$. Using \eqref{eq : computation d g u}, we have
\begin{align*}
    \partial_t\left[\partial_u g(\bx)  h\right] &=\partial_t\left[ \bar{v}_1h_1(t-\bar{u}_2) + \bar{v}_2 h_1 + \bar{v}_3 h_2 \right] \\
    & = \bar{v}_1'h_1(t-\bar{u}_2) + \bar{v}_2' h_1 + \bar{v}_3' h_2 + \bar{v}_1 (1 - \bu_2') h_1'(t-\bu_2) + \bar{v}_2 h_1' + \bar{v}_3 h_2'.
\end{align*}
This implies that
\begin{align*}
  \sup_{\|h\|_{C^1} = 1}  \| \partial_t\left[\partial_u g(\bx)  h\right]\|_{C^0} & \leq  \|(1-\bar{u}_2')\bar{v}_1\|_{C^0} + \|\bar{v}_2\|_{C^0} + \|\bar{v}_3\|_{C^0} + \|\bar{v}_1'\|_{C^0} + \|\bar{v}_2'\|_{C^0} + \|\bar{v}_3'\|_{C^0},
\end{align*}
and we conclude the proof combining the above with \eqref{eq : Z1 step 0 b}.
\end{proof}

Note that the operator $\mathbb{B} \bydef \mathbb{A}^{\leq N}D\mathbb{F}(\bx)\pi^{(N,N_2]}\mathbb{D}^\dagger$ satisfies $\mathbb{B} = \pi^{\leq N_2}\mathbb{B}\pi^{\leq N_2}$, and so does $\mathbb{D}\mathbb{B}$. Consequently, the first term of $Z_{1,1}$ is computed using Lemma \ref{lem : sampling operator norm} with $N$ replaced by $N_2$, applied to each of the four scalar blocks of the function part (the $B(C^0)$ norm of a $2\times 2$ block operator being at most the maximum over its rows of the sum of the norms of the blocks), the $\R^4$ component being treated as in Section \ref{sec:C0_operator_norms}. The Fourier coefficients of the columns of $D\mathbb{F}(\bx)\pi^{(N,N_2]}$ are enclosed exactly as those of $\pi^{\leq N}D\mathbb{F}(\bx)\pi^{\leq N}$, using the analysis of Section \ref{sec:fourier_coeff_delayed_term}. Finally, $\lambda$ is obtained from the enclosures of $\bu_2(t \pm i\bar{\rho})$ computed as in \eqref{eq:definition_Uk}, and $C_1, C_2, C_3$ as in Section \ref{sec:tail_coefficients}. In our computations, we choose $N_2 = 2N$.

\subsubsection{\boldmath$Z_2$\unboldmath~bound}

\begin{lemma}\label{lem : Z2 appli}
Let $r>0$ and let $J_r \subset \R^4$ be defined as 
\begin{align}
   J_r \bydef  [\inf \bu_1 - r, \sup \bu_1 +r] \times [\inf \bu_1 - r, \sup \bu_1 +r]  \times [\inf \bu_2 - r, \sup \bu_2 +r]  \times [\bar{\beta}_2 - r, \bar{\beta}_2 +r],
\end{align} 
and assume that $q$ and $v$ are smooth and that $v$ does not vanish on $[\inf \bu_1 - r, \sup \bu_1 +r]$, so that $\mathcal{G}$ is twice continuously differentiable on $J_r$. Denoting, for $k \in \{1,2,3\}$,
\begin{align}\label{def : Mk appli}
    \mathcal{M}_k(r) \bydef \sup_{(y,z,\omega) \in J_r} \Big( &(1+\|\bu_1'\|_{C^0})\left| \partial_1 \partial_k \mathcal{G}(y,z,\omega)\right| + \left| \partial_{z_1} \partial_k \mathcal{G}(y,z,\omega)\right| \nonumber\\
    &+ \left| \partial_{z_2} \partial_k \mathcal{G}(y,z,\omega)\right| + \left| \partial_3 \partial_k \mathcal{G}(y,z,\omega)\right| \Big),
\end{align}
where $|\cdot|$ denotes the maximum of the absolute values of all the entries of the corresponding partial derivative, let $Z_{2i}(r)$ be bounds satisfying
    \begin{align*}
        Z_{21}(r) &\geq \|v\|_{C^0} + \|\bu_2\|_{C^0} \|v'\|_{C^0}\\
        Z_{22}(r) &\geq \left(\|\bu_2\|_{C^0} + r\right)\|v'\|_{C^0} + |\bar{\beta}_2|\left( \|v'\|_{C^0}  + \|\bu_2\|_{C^0} \|v''\|_{C^0} \right)\\
        Z_{23}(r) &\geq \|v\|_{C^0} + |\bar{\beta}_2| \|v'\|_{C^0} \left(1 + \|\bu_1'\|_{C^0}\right)\\
        Z_{24}(r) &\geq \mathcal{M}_3(r)\\
        Z_{25}(r) &\geq \mathcal{M}_1(r)\left(1 + \|\bu_1'\|_{C^0} + r\right) + 2\mathcal{M}_2(r) + \|\bar{v}_1\|_{C^0}\left(2 + \|\bu_1''\|_{C^0}\right),
    \end{align*}
    where $\|v^{(k)}\|_{C^0}$ denotes the supremum of $|v^{(k)}|$ on $[\inf \bu_1 - r, \sup \bu_1 +r]$.
    Then, defining
    \begin{align}
        Z_2(r) \bydef \max\{\|\mathbb{A}\mathbb{D}^\dagger\|_{B(H_0)}, ~ \|\mathbb{D}\mathbb{A}\mathbb{D}^\dagger\|_{B(H_0)}\} \max\left\{Z_{21}(r) + Z_{22}(r) + Z_{23}(r), ~ Z_{24}(r)+Z_{25}(r)\right\},
    \end{align}
    we have that $\|\mathbb{A}(D\mathbb{F}(\bar{x} + x) - D\mathbb{F}(\bar{x}))\|_{B(H)} \leq Z_2(r)r$ for all $x \in B_r(0) \subset H$.
\end{lemma}

\begin{proof}
Let $x = (\beta, u) \in B_r(0) \subset H$ and denote $\tilde{\beta}_2 \bydef \bar{\beta}_2 + \beta_2$, $\tilde{u} \bydef \bu + u$, $\tilde{y} \bydef \tilde{u}_1(t - \tilde{u}_2)$ and $\bar{y} \bydef \bu_1(t-\bu_2)$. Since $\mathbb{I}$ in \eqref{eq : decomposition of DF} does not depend on the point at which the derivative is evaluated, we have $D\mathbb{F}(\bx + x) -D\mathbb{F}(\bx) = \mathbb{D}^\dagger(D\mathbb{Q}(\bx + x) -D\mathbb{Q}(\bx))$, and hence, arguing as in Lemma \ref{lem : Z2 autonomous},
\begin{align}\label{eq : Z2 appli reduction}
    \|\mathbb{A}(D\mathbb{F}(\bx + x) -D\mathbb{F}(\bx ))\|_{B(H)} \leq \max\{\|\mathbb{A}\mathbb{D}^\dagger\|_{B(H_0)}, ~ \|\mathbb{D}\mathbb{A}\mathbb{D}^\dagger\|_{B(H_0)}\}\, \|D\mathbb{Q}(\bx + x) -D\mathbb{Q}(\bx)\|_{B(H,H_0)}.
\end{align}
Moreover, since $H_0$ is equipped with the max norm,
\begin{align}\label{eq : Z2 appli split}
    \|D\mathbb{Q}(\bx + x) -D\mathbb{Q}(\bx)\|_{B(H,H_0)} \leq \max\left\{ \|D\xi(\bx +x) - D\xi(\bx)\|_{B(H,\R^4)}, ~ \|Dg(\bx+x) - Dg(\bx)\|_{B(H,C^0)} \right\}.
\end{align}
Moreover, we notice that
\begin{align*}
    \|\tilde{y} - \bar{y}\|_{C^0} \leq \|u_1\|_{C^0} + \|\bu_1'\|_{C^0} \|u_2\|_{C^0} \leq \left(1 + \|\bu_1'\|_{C^0}\right) r,
\end{align*}
so that, by the mean value inequality and the definition \eqref{def : Mk appli} of $\mathcal{M}_k$,
\begin{equation}\label{eq : Z2 appli mean value}
    \left\| \partial_k \mathcal{G}(\tilde{y},\tilde{u},\tilde{\beta}_2) - \partial_k \mathcal{G}(\bar{y},\bu,\bar{\beta}_2)\right\|_{C^0} \leq \mathcal{M}_k(r) r, \qquad k \in \{1,2,3\},
\end{equation}
since $(\tilde{y},\tilde{u},\tilde{\beta}_2)$ and $(\bar{y},\bu,\bar{\beta}_2)$ both take values in the convex set $J_r$.

We first estimate $\|D\xi(\bx +x) - D\xi(\bx)\|_{B(H,\R^4)}$. Since $\R^4$ is equipped with the max norm, we have
\begin{align}
    \|D\xi(\bx +x) - D\xi(\bx)\|_{B(H,\R^4)} \leq \max_{i \in \{1,2,3\}}\|D\xi_i(\bx +x) - D\xi_i(\bx)\|_{B(H,\R^{d_i})},
\end{align}
with $d_1 = d_2 = 1$ and $d_3 = 2$. First, notice that $\xi_2$ is affine, so that $\|D\xi_2(\bx +x) - D\xi_2(\bx)\|_{B(H,\R^{d_2})} = 0$. Then, given $\hat{x} = (\alpha, h) \in H$ with $\|\hat{x}\|_H \leq 1$, we have
\begin{align}
    |D\xi_1(\bx +x)\hat{x} - D\xi_1(\bx)\hat{x}| \leq |\partial_{\beta_2}\xi_1(\bx +x)\alpha_2 - \partial_{\beta_2}\xi_1(\bx)\alpha_2| + \sum_{i=1}^2 |\partial_{u_i}\xi_1(\bx +x)h_i - \partial_{u_i}\xi_1(\bx)h_i |. 
\end{align}
Looking at the first term, and using that $|\alpha_2| \leq 1$, we get
\begin{align*}
    |\partial_{\beta_2}\xi_1(\bx +x)\alpha_2 - \partial_{\beta_2}\xi_1(\bx)\alpha_2| &\leq \left| \int^0_{-\tilde{u}_2(0)} v(\tilde{u}_1(s)) ds - \int^0_{-\bu_2(0)} v(\bu_1(s)) ds \right|\\
    &\leq \|u_2\|_{C^0} \|v\|_{C^0} + \|\bu_2\|_{C^0} \|v'\|_{C^0} \|u_1\|_{C^0} \leq Z_{21}(r)r.
\end{align*}
For the second term, using \eqref{eq : computation d xi u} and $\|h_1\|_{C^0} \leq 1$,
\begin{align*}
|\partial_{u_1}\xi_1(\bx +x)h_1 - \partial_{u_1}\xi_1(\bx)h_1| &=   \left|\tilde{\beta}_2\int^0_{-\tilde{u}_2(0)} v'(\tilde{u}_1(s)) h_1(s) ds - \bar{\beta}_2\int^0_{-\bu_2(0)} v'(\bu_1(s)) h_1(s) ds \right|\\
&\leq |\beta_2| \left(\|\bu_2\|_{C^0} + r\right)\|v'\|_{C^0}  + |\bar{\beta}_2| \left(\|u_2\|_{C^0} \|v'\|_{C^0}  + \|\bu_2\|_{C^0} \|v''\|_{C^0} \|u_1\|_{C^0} \right)\\
&\leq  Z_{22}(r)r,
\end{align*}
where we split the difference as $\beta_2\int^0_{-\tilde{u}_2(0)} v'(\tilde{u}_1) h_1 + \bar{\beta}_2\big(\int^0_{-\tilde{u}_2(0)} v'(\tilde{u}_1) h_1 - \int^0_{-\bu_2(0)} v'(\bu_1) h_1\big)$. Similarly, for the third term,
\begin{align*}
    |\partial_{u_2}\xi_1(\bx +x)h_2 - \partial_{u_2}\xi_1(\bx)h_2 | &= \left|\tilde{\beta}_2 v(\tilde{u}_1(-\tilde{u}_2(0)))h_2(0)-\bar{\beta}_2 v(\bu_1(-\bu_2(0)))h_2(0)\right|\\
    &\leq |\beta_2| \|v\|_{C^0}  + |\bar{\beta}_2| \|v'\|_{C^0} \left(\|u_1\|_{C^0} + \|\bu_1'\|_{C^0} \|u_2\|_{C^0}\right) \leq Z_{23}(r)r.
\end{align*}
Summing the three previous estimates, we obtain $\|D\xi_1(\bx +x) - D\xi_1(\bx)\|_{B(H,\R)} \leq (Z_{21}(r)+Z_{22}(r)+Z_{23}(r))r$. Finally, since $\xi_3(x) = \mathbb{M}g(\beta_2,u) + \beta_1e_2$ and $|\mathbb{M}w|_\infty \leq \|w\|_{C^0}$, we have
\begin{align*}
        \|D\xi_3(\bx +x) - D\xi_3(\bx)\|_{B(H,\R^2)}  \leq \|Dg(\bx +x) - Dg(\bx)\|_{B(H,C^0)},
\end{align*}
so that it only remains to estimate the latter. Since $\partial_{\beta_1}g = \partial_{\beta_3}g = \partial_{\beta_4}g = 0$ and $|\alpha_2| \leq 1$, we have
\begin{align*}
    \|Dg(\bx +x) - Dg(\bx)\|_{B(H,C^0)} \leq \| \partial_{\beta_2}g(\bx +x) - \partial_{\beta_2}g(\bx)\|_{C^0} + \sup_{\|h\|_{C^1} \leq 1}\| \partial_{u}g(\bx +x)h - \partial_{u}g(\bx)h \|_{C^0}.
\end{align*}
Recalling \eqref{def : bar v4}, the first term is exactly the case $k=3$ of \eqref{eq : Z2 appli mean value}, hence
\begin{align*}
    \| \partial_{\beta_2}g(\bx +x) - \partial_{\beta_2}g(\bx)\|_{C^0} \leq \mathcal{M}_3(r)r \leq Z_{24}(r)r.
\end{align*}
Note that the sup over $J_r$ involved in $\mathcal{M}_k$ is bounded with interval arithmetic, by differentiating the expressions of $\partial_1\mathcal{G}, \partial_2\mathcal{G}$ and $\partial_3\mathcal{G}$ given above and evaluating the resulting second derivatives on a subdivision of $J_r$ into small boxes.

We now turn to the second term. Denoting by $\tilde{v}_1, \tilde{v}_2, \tilde{v}_3$ the quantities defined below \eqref{eq : computation d g u}, but evaluated at $\bx + x$ instead of $\bx$, we have
\begin{align*}
    \partial_{u}g(\bx +x)h - \partial_{u}g(\bx)h = \left(\tilde{v}_1 h_1(t-\tilde{u}_2) - \bar{v}_1 h_1(t-\bu_2)\right) + \left(\tilde{v}_2 - \bar{v}_2\right)h_1 + \left(\tilde{v}_3 - \bar{v}_3\right)h_2.
\end{align*}
Using \eqref{eq : Z2 appli mean value} with $k=1$, the mean value inequality and $\|h\|_{C^1} \leq 1$, the first term satisfies
\begin{align*}
    \left\|\tilde{v}_1 h_1(t-\tilde{u}_2) - \bar{v}_1 h_1(t-\bu_2)\right\|_{C^0} \leq \|\tilde{v}_1 - \bar{v}_1\|_{C^0} \|h_1\|_{C^0} + \|\bar{v}_1\|_{C^0} \|h_1'\|_{C^0} \|u_2\|_{C^0} \leq \left(\mathcal{M}_1(r) + \|\bar{v}_1\|_{C^0}\right) r.
\end{align*}
Since $\bar{v}_2$ and $\tilde{v}_2$ are the first columns of $\partial_2\mathcal{G}$ evaluated at $(\bar{y},\bu,\bar{\beta}_2)$ and $(\tilde{y},\tilde{u},\tilde{\beta}_2)$ respectively, \eqref{eq : Z2 appli mean value} with $k=2$ yields $\|\tilde{v}_2 - \bar{v}_2\|_{C^0} \leq \mathcal{M}_2(r)r$. Finally, recalling that $\bar{v}_3 = -\bar{v}_1 \bu_1'(t-\bu_2) + \partial_2\mathcal{G}(\bar{y},\bu,\bar{\beta}_2)_{col(2)}$, we get
\begin{align*}
    \|\tilde{v}_3 - \bar{v}_3\|_{C^0} &\leq \|\tilde{v}_1 - \bar{v}_1\|_{C^0} \|\tilde{u}_1'\|_{C^0} + \|\bar{v}_1\|_{C^0} \left\| \tilde{u}_1'(t-\tilde{u}_2) - \bu_1'(t-\bu_2) \right\|_{C^0} + \mathcal{M}_2(r)r\\
    &\leq \mathcal{M}_1(r)\left(\|\bu_1'\|_{C^0} + r\right)r + \|\bar{v}_1\|_{C^0}\left(\|u_1'\|_{C^0} + \|\bu_1''\|_{C^0} \|u_2\|_{C^0}\right) + \mathcal{M}_2(r)r\\
    &\leq \left( \mathcal{M}_1(r)\left(\|\bu_1'\|_{C^0} + r\right) + \|\bar{v}_1\|_{C^0}\left(1 + \|\bu_1''\|_{C^0}\right) + \mathcal{M}_2(r) \right) r.
\end{align*}
Summing the three previous estimates, we obtain
\begin{align*}
    \sup_{\|h\|_{C^1} \leq 1}\| \partial_{u}g(\bx +x)h - \partial_{u}g(\bx)h \|_{C^0} &\leq \left( \mathcal{M}_1(r)\left(1 + \|\bu_1'\|_{C^0} + r\right) + 2\mathcal{M}_2(r) + \|\bar{v}_1\|_{C^0}\left(2 + \|\bu_1''\|_{C^0}\right) \right) r\\
    &\leq Z_{25}(r)r,
\end{align*}
which concludes the proof, combining the above with \eqref{eq : Z2 appli reduction} and \eqref{eq : Z2 appli split}.
\end{proof}

\subsection{Computer-assisted proof of existence}\label{sec : CAP threshold}

\begin{theorem}\label{th : CAP threshold}
    Let $a = 3$, $b = 3$, $\mu = 0.3$, $\gamma = 0.13$, $m = 56$, $v^- = 0.5$, $v^+ = 0.2$, $\theta_v = 1$ and $q^- = q^+ = 1$, and let $\bx = (\bar{\beta},\bu)$ be the approximate solution represented in Figure~\ref{fig : threshold}. Then, there exist $\tilde{\beta}_2 > 0$ and $\tilde{u} = (\tilde{u}_1,\tilde{u}_2) \in C^1_{per}([-\pi,\pi])$ such that
    \begin{align*}
        |\tilde{\beta}_2 - \bar{\beta}_2| \leq r_0 \quad \text{ and } \quad \|\tilde{u} - \bu\|_{C^1} \leq r_0, \quad \text{ where } r_0 \bydef 2 \times 10^{-9},
    \end{align*}
    and $(w,\tau) : s \mapsto \left(\tilde{u}_1\left(\frac{s}{\tilde{\beta}_2}\right), \tilde{\beta}_2\tilde{u}_2\left(\frac{s}{\tilde{\beta}_2}\right)\right)$ is a $2\pi\tilde{\beta}_2$-periodic solution to \eqref{eq : appli original}.
\end{theorem}

\begin{proof}
 $\bu$ is constructed as a trigonometric polynomial of order $70$.
    We choose $\bar{\nu} = 1.05$, that is $\bar{\rho} = \ln(1.05)$, and we verify using Lemma \ref{lem : avoiding roots} that $\bar{g}_1, \bar{g}_2$ and $\bar{v}_1, \dots, \bar{v}_4$ are analytic on $S_{\bar{\rho}}$. Choosing the truncation size $N = 1000$,  we construct $\mathbb{A}$ as described in Section \ref{sec : approximate objects and projections}. Then, choosing $N_{\tt FFT} = 2^{14}$, $N_2 = 2N$ and $r_* = 10^{-3}$, we compute the bounds of Lemmas \ref{lem : Y bound toy model}, \ref{lem : Z1 appli} and \ref{lem : Z2 appli}, using Lemma~\ref{lem : from infinite to finite} with $\|\cdot\, e\|_{H_0}$ replaced by $\|\cdot\,(0,(1,0))\|_{H_0} + \|\cdot\,(0,(0,1))\|_{H_0}$ since $u$ has two components, and we obtain
    \begin{align*}
        Y \leq 8.54 \times 10^{-10}, \quad Z_1 \leq 0.569 \quad \text{ and } \quad Z_2(r_*) \leq 3.36 \times 10^{5}.
    \end{align*}
    Since the right-hand sides of the inequalities defining $Z_{21}(r), \dots, Z_{25}(r)$ are non-decreasing in $r$, Lemma \ref{lem : Z2 appli} provides \eqref{eq:Z2_radPolyBanach} with $Z_2 = Z_2(r_*)$. As in the proof of Theorem \ref{th : CAP autonomous}, $\mathbb{A}$ is injective and, since $r_0 \leq r_*$, $Y + Z_1 r_0 + \frac{1}{2}Z_2 r_0^2 \leq r_0$ and $Z_1 + Z_2 r_0 < 1$, Theorem \ref{thm:radPolyBanach} provides a zero $\tilde{x} = (\tilde{\beta},\tilde{u})$ of $\mathbb{F}$ such that $\|\tilde{x} - \bx\|_H \leq r_0$. Finally, $\bar{\beta}_2 \geq 1.78$ and $v >0$ with the above choice of parameters, so that $\tilde{\beta}_2 > 0$, and we conclude using Lemma \ref{lem : zeros of F threshold}.
\end{proof}

\section{Conclusion and perspectives} \label{sec:conclusion}

In this work, we developed a computer-assisted framework for proving the existence of periodic solutions of SDDEs, which does not rely on perturbations of ordinary or constant-delay differential equations. The approach combines a zero-finding problem posed in the low-regularity space $C^1$, a Newton--Kantorovich argument, rigorous enclosures of the Fourier coefficients of state-dependent compositions, and sharp bounds on finite-dimensional operator norms on $C^0$. We illustrated it on a periodically forced toy model (Theorem~\ref{th : CAP forcing}), on an autonomous scalar SDDE with two state-dependent delays (Theorem~\ref{th : CAP autonomous}), and on a threshold-type SDDE arising in gene regulation (Theorem~\ref{th : CAP threshold}).

A natural direction for future work concerns solutions with limited smoothness. As suggested by the numerical results of Section~\ref{sec : CAP forcing}, periodic solutions of SDDEs may fail to be analytic, in which case the constant $C$ in the geometric tail estimate \eqref{eq : tail estimates} grows rapidly as the order of the approximate solution increases. Since the Newton--Kantorovich argument itself is formulated in $C^1$ and does not require analyticity, this limitation only comes from the enclosure of the Fourier coefficients of Section~\ref{sec:fourier_coeff_delayed_term}. One could replace \eqref{eq : tail estimates} by an algebraic estimate of the form
\[
|V_n| \le \frac{C_{\rm alg}}{|n|^s}, \quad n \neq 0, \qquad \text{with } C_{\rm alg} \bydef \frac{1}{2\pi}\|v^{(s)}\|_{L^1},
\]
obtained by integrating by parts $s$ times. The derivatives of $v$ are only needed on the real line, where they can be enclosed with the same grid-based techniques, and $C_{\rm alg}$ remains bounded as the order of the approximate solution increases, provided the solution is $C^s$. This is not restrictive here, since the periodic solutions obtained in this paper are $C^\infty$, by a bootstrapping argument (see Lemma~\ref{lem : zeros of F are solutions}), so that any $s \ge 2$ can be used. Such an approach is well suited to solutions which are $C^\infty$ but not analytic, a situation which is known to occur for delay-differential equations with analytic nonlinearity and analytic time-dependent delay \cite{MR3229655}. Combined with the discrete Poisson summation formula as in Section~\ref{sec:finite_coefficients} (see \cite{vandenberg2026spectral}), this would yield an algebraic analogue of the aliasing bound \eqref{eq:aliasing_control}, with an error of order $C_{\rm alg}(N_{\tt FFT}-|n|)^{-s}$ instead of $C\bar\nu^{|n|-N_{\tt FFT}}$, at the price of a larger $N_{\tt FFT}$. Other natural extensions include the rigorous continuation of branches of periodic orbits with respect to parameters, and the treatment of more general state-dependent delays.

\section*{Acknowledgments}

This project started during the workshop \emph{Towards rigorous results in state-dependent delay equations}, held at the Lorentz Center in Leiden in March 2024. We are grateful to the Lorentz Center for its support and hospitality, and to the participants of the workshop for stimulating discussions. In particular, we thank Jan Sieber for several fruitful conversations, especially about how to properly set up the zero-finding problem \eqref{eq:the_map}.

MB and MC were supported by CAPPS (\mbox{ANR-23-CE40-0004-01}).  MC was also supported by the FMJH
(\mbox{ANR-22-EXES-0013}). JPL was supported by the Natural Sciences and Engineering Research Council of Canada (NSERC).

\section*{AI disclosure statement}




The paper and the accompanying code \cite{code_proof} were written by the authors. Anthropic's Claude Opus 5.5 was used to proofread the paper and the code. 
The only part of the paper for which more significant AI assistance was used is Lemma~\ref{lem : sampling operator norm}. When the approach based on direct interval-valued FFT enclosures discussed in Section~\ref{sec:C0_operator_norms} proved very expensive for the example of Section~\ref{sec:threshold_SDDE}, we asked Claude for a cheaper alternative, and its suggestions led to Lemma~\ref{lem : sampling operator norm}. The authors take full responsibility for the final content of the paper and code.

\bibliographystyle{unsrt}
\bibliography{references}

\end{document}